\documentclass[a4paper, reqno]{amsart}
\usepackage[margin=2.5cm]{geometry}                
\usepackage{amssymb,amsmath, mathtools, mathabx}  
\usepackage[colorlinks,
             linkcolor=blue,
             citecolor=green!70!black,
             pdfproducer={pdfLaTeX},
             pdfpagemode=None,
             bookmarksopen=true
             bookmarksnumbered=true]{hyperref}
\usepackage{cleveref}
\usepackage{hyperref}
\usepackage[mathscr]{euscript}
\usepackage[all]{xy}

\DeclareMathAlphabet{\mathpzc}{OT1}{pzc}{m}{it}

\usepackage[utf8]{inputenc}
\usepackage{tikz}
\usetikzlibrary{matrix,arrows,decorations.pathmorphing,decorations.pathreplacing,decorations.markings}
\tikzset{
    dot/.style={circle,draw,fill,inner sep=1pt},
    arrow/.style={->,thick,shorten <=2pt,shorten >=2pt},
    every label/.append style = {font = \small}
    }
\usepackage{tikz-cd}
\newtheorem{theorem}{Theorem}[section]
\newtheorem*{theorem*}{Theorem}
\newtheorem{prop}[theorem]{Proposition}
\newtheorem*{prop*}{Proposition}
\newtheorem{cor}[theorem]{Corollary}
\newtheorem{lemma}[theorem]{Lemma}

\theoremstyle{definition}
\newtheorem{defn}[theorem]{Definition}
\newtheorem{definition}[theorem]{Definition}
\newtheorem*{defn*}{Definition}
\newtheorem*{const*}{Construction}
\newtheorem*{warning*}{Warning}
\newtheorem{ex}[theorem]{Example}

\newtheorem{remark}[theorem]{Remark}

\newtheorem{notation}[theorem]{Notation}
\newtheorem{construction}[theorem]{Construction}

\usepackage[utf8]{inputenc}
\DeclareFontFamily{U}{min}{}
\DeclareFontShape{U}{min}{m}{n}{<-> udmj30}{}

\usepackage{bbm}

\newcommand\cB{\mathscr B} 
\newcommand\cC{\mathscr C} 
\newcommand\cD{\mathscr D}
\newcommand\cE{\mathscr E} 
\newcommand\cF{\mathscr F}
\newcommand\cG{\mathscr G}
\newcommand\cH{\mathscr H}
\newcommand\cI{\mathscr I}

\newcommand\cL{\mathscr L}
\newcommand\cM{\mathscr M}

\newcommand\cO{\mathscr O}
\newcommand\cP{\mathscr P}
\newcommand\cS{\mathscr S}

\newcommand\cT{\mathscr T}

\newcommand\cV{\mathscr V}
\newcommand\cW{\mathscr W}

\newcommand\rB{\mathrm B}

\newcommand\NN{\mathbb N} \newcommand\bN\NN

\newcommand\bZ{\mathbb Z}

\newcommand\gpd{\mathrm{Gpd}}

\newcommand\id{\mathrm{id}}

\newcommand\mor{\mathrm{Hom}}

\newcommand\inj{\mathrm{inj}}
\newcommand\op{\mathrm{op}}
\newcommand\cat{\mathrm{Cat}}
\newcommand\comp{\mathrm{comp}}

\newcommand\ccat{\mathpzc{Cat}}
\newcommand\cfact{\mathpzc{Fact}}
\newcommand\fact{\mathrm{Fact}}
\newcommand\arr{\mathrm{Ar}}

\newcommand\spanc{\mathrm{Span}}
\newcommand\colim{\mathrm{colim}}
\newcommand\bydef{\overset{\mathrm{def}}{=}}
\newcommand\sq{\mathrm{Sq}}

\newcommand\modl{\mathrm{Mod}}

\newcommand\cart{\mathrm{coCart}}

\newcommand\lax{\mathrm{lax}}

\newcommand\wrr{{\overset{\sim}{\rightarrow}}}

\newcommand\inc{{\mathrm{incl}}}

\newcommand\act{\mathrm{act}}

\newcommand\Act{\mathrm{Act}}
\newcommand\inrt{\mathrm{int}}
\newcommand\el{\mathrm{el}}

\newcommand\pt{\mathrm{pt}}
\newcommand\ep{\mathrm{e.p.}}

\newcommand\fip{\mathrm{f.i.}}
\newcommand\ap{\mathrm{a.p.}}
\newcommand\orr{\mathrm{Or}}
\newcommand{\xinert}[1]{\overset{#1}{\rightarrowtail}}
\newcommand{\xactive}[1]{\overset{#1}{\twoheadrightarrow}}
\newcommand{\xhorizontal}[1]{\overset{#1}{\rightharpoonup}}
\newcommand{\xvertical}[1]{\overset{#1}{\rightharpoondown}}

\title{factorization systems in $\infty$-categories}
\author{Roman Kositsyn}
\date{May 2021}

\begin{document}

\maketitle
\begin{abstract}
    We extend the notion of a factorization system in a category to the realm of $\infty$-categories. To this end, we provide a description of the category of $\infty$-categories with factorization systems as the category of presheaves of spaces on a certain category that satisfy a version of the Segal condition. Additionally, we study the notion of distributive laws between monads and, more generally, lax functors in the context of categories with factorization systems. In particular, we also characterize categories with factorization systems as distributive laws in the bicategory of spans.
\end{abstract}
\section{Introduction}
The principal goal of this paper is to present a homotopy-coherent notion of a factorization system suitable for applications to $\infty$-categories. Roughly speaking, a factorization system on an ordinary category $C$ is the data of two subcategories $(E,M)$ such that every morphism $f$ in $C$ uniquely decomposes as $f=e\circ m$ with $e\in E$ and $m\in M$. Examples of factorization systems abound in mathematics, common general sources of them include vertical/cartesian factorization system on Grothendieck fibration and generic/free factorization systems associated to cartesian monads studied in  \cite{weber2004generic}. In the $\infty$-categorical context we also have factorization systems attached to Cartesian fibrations and Cartesian monads (this time using \cite{chu2019homotopy}), as well as, for example, factorization systems arising from $t$-structures as described in \cite{fiorenza2016t}.\par
Factorization systems were first defined in \cite{isbell1957some}, however numerous equivalent formalizations of the concept have appeared since. The definition closest to ours can be found in \cite{freyd1972categories}. According to op. cit. a factorization system on a category $C$ is given by a pair $(E,M)$ of classes of morphisms of $C$ that satisfy the following conditions:
\begin{enumerate}
    \item Both $E$ and $M$ contain all isomorphisms and are closed under composition;\label{itm:one}
    \item Every morphism $f$ of $C$ can be factored as $f=m\circ e$ with $m\in M$ and $e\in E$;\label{itm:two}
    \item The factorization is functorial in the following sense: for any morphism $(u,v)$ in the arrow category from $(m,e)$ to $(m',e')$ as shown below there is a unique dotted arrow $w$ making the diagram commute 
 \[ \begin{tikzcd}[row sep=huge, column sep=huge]
A\arrow[rr, "e"] \arrow[d, "u"] &{}
& B  \arrow[d, dashed, "w"] \arrow[rr,"m"] &{}
& C \arrow[d,"v"]\\
A'  \arrow[rr,"e'"] & {}
& B'  \arrow[rr,"m'"] &{}
& C'
\end{tikzcd}.\]\label{itm:three}
\end{enumerate}
Alternatively, a factorization system can be defined by lifting properties as in \cite{bousfield1977constructions} or as an algebra for a certain monad on $\cat$ as in \cite{korostenski1993factorization}. We define a factorization system on an $\infty$-category $\cC$ to be given by a pair of subcategories $\cH$ and $\cV$ that satisfy the following condition:
\begin{equation}
    \text{For any morphism $f\in\cC$ the space of factorizations 
    $\{f\cong h\circ v,h\in \cH,v\in \cV\}$
    is contractible.}\label{cond:factorization}\tag{$*$}
\end{equation}
Condition \eqref{cond:factorization} certainly implies \eqref{itm:two} and \eqref{itm:three} above, but it appears to contradict \eqref{itm:one} since it requires every isomorphism to admit a unique decomposition as well. However, this is not such a big problem, at least if we restrict ourselves to complete Segal spaces $\cC$, since in that case we simply consider all invertible morphisms as elements of the underlying space of objects. Alternative definitions of a factorization system for $\infty$-categories can be found in \cite[Definition 5.2.8.8.]{lurie2009higher} and \cite[Section 24]{joyal2008notes}, the only essential difference between those definitions and ours is that they require $\cH$ and $\cV$ to be closed under retracts. Additionally, \cite{loregian2017factorization} provides a treatment of factorization systems in the context of triangulated and stable categories. \par
Despite its simplicity, condition \eqref{cond:factorization} is not very useful in practice. To rectify this, we give an alternative description of $\infty$-categories with factorization systems as models for a certain theory $\cfact$. We use the term "theory" in the same way as in \cite{kositsyn2021completeness}, which is itself an adaptation of an earlier paper \cite{berger2012monads} to the formalism of $\infty$-categories. Roughly speaking, a theory gives us a way to represent the category of algebras for a certain monad as a category of presheaves satisfying a Segal-type condition. Objects that can be represented this way include Segal spaces, commutative monoids and symmetric operads to name a few. Our first main result is \cref{thm:factorization_model} that provides an equivalence between the $\infty$-category of $\infty$-categories with factorization systems and the $\infty$-category of models for the theory $\cfact$.\par
Another object of interest for us is the notion of a distributive law. They were originally defined in \cite{beck1969distributive}. By definition, a distributive law between monads $T_1$ and $T_2$ consists of a natural transformation $\alpha:T_1T_2\rightarrow T_2 T_1$ making the following diagrams commute
\[
\begin{tikzcd}[row sep=huge, column sep=huge]
T_1T_2T_2\arrow[d, "T_1m_2"]\arrow[r, "\alpha T_2"]&T_2T_1T_2\arrow[r, "T_2 \alpha"]&T_2T_2T_1\arrow[d, "m_2 T_1"]\\
T_1T_2\arrow[rr, "\alpha"]&{}&T_2T_1
\end{tikzcd},
\begin{tikzcd}
{}&T_1\arrow[dl, "T_1e_2"]\arrow[dr, "e_2T_1"]\\
T_1T_2\arrow[rr, "\alpha"]&{}&T_2T_1
\end{tikzcd}.
\]
The main motivation behind this notion is that it allows us to endow $T_2T_1$ with the structure of a monad with multiplication given by $T_2T_1T_2T_1\xrightarrow{\alpha}T_2T_2T_1T_1\xrightarrow{m_2*m_1}T_2T_1$. This definition, of course, cannot be used in the higher-categorical context as it lacks the necessary higher coherences. Instead, we conceptualize this notion as a special case of a "lax functor" from a category with factorization system. A lax functor from a category $C$ to a bicategory $B$ (as defined e.g. in \cite{street1972two}) is the data of a morphism $F:C\rightsquigarrow B$ of the  underlying graphs of $C$ and $B$ together with 2-morphisms $\alpha_{f,g}:F(g)\circ F(f)\rightarrow F(g\circ f)$ for every composable pair of morphisms in $C$ and $\beta_x:\id_{F(x)}\rightarrow F(\id_x)$ for every object of $x$ that satisfy certain associativity constraints. In particular, giving a lax functor from a point to $B$ is equivalent to giving a monad in $B$. The $\infty$-categorical generalization of this concept can be found for example in \cite{gagna2020gray}, \cite{gaitsgory2017study} and \cite{kositsyn2021completeness}. Our extension of this notion to categories with factorization systems is quite technical, but roughly speaking a lax functor from $\cF$ to $\cB$ contains all the data associated to the lax functor from the underlying categories $\cV$ and $\cH$ to $\cB$ together with additional 2-morphisms $\gamma_{h,v}:F(v)\circ F(h)\rightarrow F(\widetilde{h})\circ F(\widetilde{v})$ for every factorization $v\circ h\cong \widetilde{h}\circ \widetilde{v}$ in $\cF$. We then define a distributive law to be a lax functor from a singleton category considered as a category with a factorization system. Classical distributive laws can be characterized as monads in the category of monads, and the same is true for our notion of a distributive law by \cref{rem:distributive_equivalence}, so it follows that it reduces to the usual notion for the case of discrete categories. Distributive laws are not the only interesting case of lax functors, however. Our second main result \cref{thm:overcategory} provides an equivalence between the category of lax functors from an $\infty$-category with a factorization system $\cF$ to the flagged bicategory of spans and the overcategory $\fact_{/\cF}$ in the $\infty$-category of $\infty$-categories with factorization systems. In particular, the category $\fact$ itself can be identified with the category of distributive laws in spans (since $*$ is the terminal object in $\fact$, so $\fact_{/*}\cong\fact$). \par
We will now briefly go through the contents of each section. \Cref{sect:two} contains the description of the the theory $\cfact$ as well as \cref{thm:factorization_model} that provides an equivalence between the models of this theory and categories with factorization systems. \Cref{sect:three} contains the definition of lax functors from categories with factorization systems. Additionally, it contains the description of bicategories with factorization system in \cref{prop:bicategory_comparison} as well as an alternative construction of the Gray tensor product. The main result of this section is \cref{prop:lax_for_objects_of_fact} which explicitly describes lax functors with domains lying in $\cfact$. \Cref{sect:four} contains the definition and basic properties of complete categories with factorization systems. \Cref{sect:five} contains our second main result \cref{thm:overcategory} providing an equivalence between an overcategory $\fact_{/\cF}$ and a certain subcategory of lax functors from $\cF$ to $\spanc$.\par
\paragraph{\textbf{Notations and conventions}:} Our standard reference for $\infty$-categories is \cite{lurie2009higher}, however we will also use the equivalence between quasicategories and complete Segal spaces (see e.g. \cite{joyal2007quasi}) without further mention. We will typically call $\infty$-categories simply categories, except in cases where it may lead to confusion. We denote by $\cS$ the category of spaces and by $\cP(\cC)$ the category of presheaves on a category $\cC$. We assume the reader is familiar with the terminology and results of \cite{chu2019homotopy} and \cite{kositsyn2021completeness}. For a theory $\cO$ we will typically denote $\mathrm{O}$ the category $\modl_\cO(\cS)$ of models for this theory. In particular, we will denote by $\ccat$ the algebraic pattern of Segal spaces (i.e. $\Delta^\op$ with the usual active/inert factorization and with elementary objects given by $[0]$ and $[1]$) and by $\cat$ the category of Segal spaces. We will also sometimes refer to the object in $\mathrm{O}$ as "presheaves on $\cO^\op$ satisfying the Segal condition". We use the terms "$n$-fold Segal space" and "$n$-category" interchangeably. In particular, a category in our sense does not have to be a complete Segal space, instead we denote by $\cat^\comp$ the category of complete Segal spaces.
\section{Theory of categories with factorization systems}\label{sect:two}
Our basic objective in this section is to construct a theory $\cfact$ whose models are given by categories with factorization systems. More specifically, we define it to be the category whose objects are finite strings $S$ made of two characters. To such a string we attach a category $\sq(S)$, which can be thought of as a category containing all possible compositions and factorizations of the characters appearing in $S$, and define morphisms in $\cfact$ to be given by morphisms between $\sq(S)$ that preserve the factorization systems. Our first major result is \cref{prop:generators}, which gives a purely combinatorial description of $\cfact$, similar to the well-known description of $\Delta$ as a category generated by faces and degeneracies. This, in particular, allows us to endow $\cfact$ with the structure of a theory with arities in \cref{prop:factorization_theory}. It also contributes to the proof of our main result \cref{thm:factorization_model} that provides an equivalence between the category of models of $\cfact$ and the category of categories with factorization systems.
\begin{defn}\label{def:factorization_categories}
We say that a category $\cC$ admits an factorization system if there are subcategories $\cH$ and $\cV$ of $\cC$ with the same underlying space of objects such that for every $f\in \mathrm{Ar}(\cC)$ the space of decompositions
\[\{f\cong h\circ v,h\in \cH,v\in \cV\}\] 
is contractible. We will refer to the subcategory $\cH$ as the \textit{horizontal} subcategory and to $\cV$ as the \textit{vertical} subcategory of $\cC$, we will typically denote morphisms in $\cH$ by $c\rightharpoonup d$ and morphisms in $\cV$ by $c'\rightharpoondown d'$.
\end{defn}
\begin{definition}
Denote by $\cfact^\op$ the category whose objects are given by sequences \[S\bydef(0,1,...,i_0|0,1,...,j_0|i_0,i_0+1,...i_1|j_0,...,|i_k,i_k+1,...,n|j_s,j_s+1,...,m).\]
Every sequence like this can be identified with a string of length $(n+m)$ containing symbols $h$ and $v$, in which the first $i_0$ symbols are given by $h$, followed by $j_0$ symbols $v$, followed by $(i_1-i_0)$ symbols $h$ etc., and so we would often denote the object $S$ above by $(i_0|j_0|i_1-i_0|j_1-j_0|...|n-i_k|m-j_k)$. Note that to every object of this form we can associate a string of morphisms in the category $[n]\times[m]$ of maximal length.\par
To each chain $S$ we can associate a full subcategory $\sq(S)$ of $[n]\times [m]$ that contains all objects that lie above $S$ viewed as a path from $(0,0)$ to $(n,m)$ in $[n]\times[m]$. Observe that every morphism in $\sq(S)$ admits a decomposition $\overline{h}\circ \overline{v}$, where $\overline{v}$ is a morphism of the form $(\id,s)$ and $\overline{h}$ is a morphism of the form $(t,\id)$. We define a morphism from $S$ to $S'$ in $\cfact^\op$ to be a morphism of categories from $\sq(S)$ to $\sq(S')$ that preserves this factorization system, i.e. that takes vertical morphisms to vertical morphisms and horizontal morphisms to horizontal morphisms. 
\end{definition}
\begin{ex}\label{ex:simple}
The diagram below shows an example of $\sq(S)$ for $S=(1|1|2|1)$:
\[
\begin{tikzcd}[row sep=huge, column sep=huge]
(0,2)\arrow[r]&(1,2)\arrow[r]&(2,2)\arrow[r]&(3,2)\\
(0,1)\arrow[u]\arrow[r]&(1,1)\arrow[u]\arrow[r]&(2,1)\arrow[u]\arrow[r]&(3,1)\arrow[u]\\
(0,0)\arrow[u]\arrow[r]&(1,0)\arrow[u]
\end{tikzcd}.
\]
\end{ex}
\begin{lemma}\label{lem:restriction}
A morphism $F:\sq(S)\rightarrow\sq(S')$ in $\cfact^\op$ is uniquely determined by its restriction to $S\subset\sq(S)$.
\end{lemma}
\begin{proof}
Indeed, denote by $\sq(S)_{\leq1}$ the full subcategory of $\sq(S)$ containing the objects belonging to $S$ as well as those objects that appear in the horizontal/vertical factorization of morphisms of the form $v_s\circ h_t$, where $v_s$ (resp. $h_t$) is an elementary vertical (resp. horizontal) morphism in $S$ (in \cref{ex:simple} the category $\sq(S)_{\leq1}$ contains all objects except $(0,2)$ and $(1,2)$). Then since $F$ preserves the factorization we see that the value of $F$ on $\sq(S)_{\leq1}$ is uniquely determined by its value on $S$.\par
We can then define $\sq(S)_{\leq2}$ as the full subcategory of $\sq(S)$ containing the factorizations of the elementary morphisms in $\sq(S)_{\leq1}$ and extend $F$ to it, then define $\sq(S)_{\leq3}$ etc. Since $\sq(S)$ is a finite category, we will eventually get $\sq(S)_{\leq n}\cong\sq(S)$ for some $n$, which concludes the proof of our claim.
\end{proof}
\begin{construction}\label{constr:elementary morphisms}
We will describe certain morphism in the category $\cfact^\op$. First, denote by $\sigma^h_i$ for $0\leq i\leq n$ the morphism
\[(0,1,...,i,i+1,...,n|j_s,...,m)\rightarrow(0,1,...,i,...,n-1|j_s,...,m)\]
from the string $S$ to the string $S_{\widehat{i}}$ obtained from $S$ by deleting the $i$th elementary horizontal morphism. The morphism itself is given by sending all elementary morphisms in $S$ to the corresponding morphisms in $S_{\widehat{i}}$, except the $i$th horizontal morphism that is sent to the identity. We can also define $\sigma^v_j$ for $0\leq j\leq m$ for vertical morphisms.\par
Similarly, for $0\leq i< n+1$ we can define a morphism
\[(0,1,...,i,i+1,...,n|j_s,...,m)\rightarrow(0,1,...,i,i+1,i+2...,n+1|j_s,...,m)\]
from the string $S$ to the string $S_{\widetilde{i}}$ obtained from $S$ by adding an extra horizontal morphism between $i-1$ and $i$. The morphism itself send the string $S$ into $S_{\widetilde{i}}$ such that the image omits the $i$th object. We can similarly define $\delta^v_j$ for $0<j\leq m+1$.\par
Observe that the morphisms $\sigma^{v,h}_j$ and $\delta^{v,h}_i$ satisfy the usual simplicial identities
\begin{align}\label{eq:one}
\begin{split}
     \delta^{v,h}_i\circ \delta^{v,h}_j\cong \delta^{v,h}_{j+1}\circ \delta^{{v,h}}_i\;\text{    if }i\leq j\\
    \sigma^{{v,h}}_j\circ \sigma^{v,h}_i\cong \sigma^{v,h}_{i}\circ \sigma^{v,h}_{j+1}\; \text{ if }i\leq j\\
    \sigma^{v,h}_j\circ \delta^{v,h}_i\cong 
    \begin{cases}
        \delta^{v,h}_i\circ \sigma^{{v,h}}_{j-1}&\text{    if }i< j\\
        \id_S \;\text{    if }i=j&\text{ or }i=j+1\\
        \delta^{v,h}_{i-1}\circ \sigma^{{v,h}}_{j}&\text{    if }j+1< i\\
    \end{cases}.
\end{split}
\end{align}
Apart from that, they also satisfy the following commutativity relations:
\begin{align}\label{eq:two}
\begin{split}
    \sigma_j^v\circ\delta_i^h&\cong\delta^h_i\circ \sigma^v_j\\
    \sigma_j^h\circ\delta_i^v&\cong\delta^v_i\circ \sigma^h_j\\
    \delta_j^v\circ\delta_i^h&\cong\delta^h_i\circ \delta^v_j\text{ for $(i,j)\neq(0,0)$ and $(i,j)\neq(n+1,m+1)$}\\
    \sigma_j^v\circ\sigma_i^h&\cong\sigma^h_i\circ \sigma^v_j.
\end{split}
\end{align}\par
Finally, there are morphisms 
\[\gamma_{j,i}:(0,1,...,j|i,...,m)\rightarrow(0,1,...,i|j,...,n)\]
from a string $S$ in which the $j$th vertical morphism is followed by the $i$th horizontal morphism to the string $\widetilde{S}$ in which those morphisms switch places. The morphism itself sends all elementary morphisms in $S$ to the corresponding morphisms in $\widetilde{S}$, except for the $i$th horizontal and $j$th vertical morphisms that are sent to the horizontal/vertical decomposition of the corresponding morphisms in $\widetilde{S}$. The diagram below illustrates a particular example of this morphism:
\[
\begin{tikzcd}[row sep=huge, column sep=tiny]
{}&{}&{}&{}\\
{}\arrow[r, "..." description]&(j+1,i-1)\arrow[r]&(j+1,i)\arrow[u, "..." description]\arrow[r, "h_i"]&(j+1,i+1)\arrow[u, "..." description]\\
{}\arrow[r, "..." description]&(j,i-1)\arrow[u]\arrow[r,"h_{i-1}"]&(j,i)\arrow[u, "v_j"]
\end{tikzcd}\overset{\gamma_{j,i}}{\Rightarrow}
\begin{tikzcd}[row sep=huge]
{}&{}&{}&{}\\
{}\arrow[r, "..." description]&(j+1,i-1)\arrow[r]&(j+1,i)\arrow[u, "..." description]\arrow[r, "\gamma_{j,i}(h_i)"]&(j+1,i+1)\arrow[u, "..." description]\\
{}\arrow[r, "..." description]&(j,i-1)\arrow[u]\arrow[r, "\gamma_{j,i}(h_{i-1})"]&(j,i)\arrow[u, "\gamma_{j,i}(v_j)"]\arrow[r]&(j+1,i)\arrow[u]
\end{tikzcd}.
\]\par
Observe that different $\gamma_{j,i}$ commute with each other (in the following formula we implicitly assume that both compositions make sense)
\begin{align}\label{eq:three}
    \gamma_{j,i}\circ \gamma_{k,t}&\cong \gamma_{k,t}\circ \gamma_{j,i}.
\end{align}
Their interactions with various $\sigma_j^{v,h}$ and $\delta^{v,h}_i$ can be described by the following relations:
\begin{align}\label{eq:four}
\begin{split}
    \sigma_{i}^h&\cong \sigma_{i-1}^h\circ \gamma_{j,i}\\
    \sigma_j^v&\cong \sigma_j^v\circ \gamma_{j,i}\\
    \sigma_k^{h}\circ \gamma_{j,i}&\cong\gamma_{j,i-1}\circ \sigma^{h}_k\text{    for $k\neq i$}\\
    \sigma_k^{v}\circ \gamma_{j,i}&\cong\gamma_{j-1,i}\circ \sigma^{v}_k\text{    for $k\neq j$}\\
    \gamma_{i+1,j}\circ \gamma_{i,j}\circ \delta^h_i&\cong\delta^h_i\circ \gamma_{i,j}\text{  for $0<i<n+1$}\\
    \gamma_{i,j+1}\circ \gamma_{i,j}\circ \delta^v_j&\cong\delta^v_j\circ \gamma_{i,j}\text{  for $0<j<m+1$}\\
    \gamma_{i,j}\circ \delta^h_k&\cong\delta^h_k\circ \gamma_{i,j}\text{    for $k\neq i$}\\
    \gamma_{i,j}\circ \delta^v_k&\cong\delta^v_k\circ \gamma_{i,j}\text{    for $k\neq j$}.
\end{split}
\end{align}
Finally, we have the following special formulas for $\delta^{v,h}_0$ and $\delta^{v,h}_{m+1,n+1}$:
\begin{align}\label{eq:five}
\begin{split}
    \gamma_{1,0}\circ \delta^v_0\circ \delta^h_0&\cong \delta^h_0\circ\delta^v_0\\
    \gamma_{n+1,m}\circ \delta^h_{n+1}\circ \delta^v_{m+1}&\cong \delta^v_{m+1}\circ\delta^h_{n+1}.
\end{split}
\end{align}
\end{construction}
\begin{prop}\label{prop:generators}
The morphisms of the category $\cfact^\op$ are generated by the morphisms $\sigma^{v,h}_s$,$\delta^{v,h}_t$ and $\gamma_{i,j}$ and the relations between them described in \cref{constr:elementary morphisms}.
\end{prop}
\begin{proof}
Denote by $\widetilde{\cfact^\op}$ the category with the same objects as $\cfact^\op$, but with morphisms generated by morphisms in \cref{constr:elementary morphisms} and relations between them, then there is a natural morphism $F:\widetilde{\cfact^\op}\rightarrow\cfact^\op$ and we need to prove that it is bijective on morphisms. To do this, we first describe a canonical presentation for a morphism in $\widetilde{\cfact^\op}$. An arbitrary morphism $f$ in $\widetilde{\cfact^\op}$ is given by a composition of a string of morphisms containing various $\sigma^{v,h}_s$,$\delta^{v,h}_t$ and $\gamma_{i,j}$. First, observe that using \cref{eq:one}, \cref{eq:two} and \cref{eq:four} we can rewrite it as
\begin{align}\label{eq:six}
    f\cong \overline{\gamma}\circ \overline{i}\circ a^h\circ a^v\circ s^h\circ s^v,
\end{align}
where $s^h$ (resp. $s^v$) is given by a composition of various $\sigma^h_i$ (resp. $\sigma^v_j$), the morphism $a^h$ (resp. $a^v$) is a composition of $\delta^h_k$ (resp. $\delta^v_l$), the morphism $\overline{i}$ is a composition of morphisms of the form $\delta^{v,h}_0$ and $\delta^{v,h}_{m+1,n+1}$ and $\overline{\gamma}$ is a composition of $\gamma_{i,j}$. This presentation is generally not unique. Assume that there are $n_0$ terms of the form $\delta^h_0$ in $\overline{i}$ and $m_0$ terms of the form $\delta^v_{m+1}$, then we claim that there is a decomposition of the morphism $f$ of the form (\ref{eq:six}) in which there are no morphisms $\gamma_{i,j}$ with $i<n_0$ and $j>m_0$. Indeed, let $\overline{\gamma}\cong \gamma_{i_N,j_N}\circ\gamma_{i_{N-1},j_{N-1}}\circ...\circ \gamma_{i_0,j_0}$ and let $l$ be the minimal index for which $i_l<n_0$. It can only appear in $\overline{\gamma}$ if the $(n_0-i_l)$th morphism $\delta^h_0$ in $\overline{i}$ is followed by $\delta^v_0$. Then we can use the commutativity relations (\ref{eq:three}) and (\ref{eq:four}) as well as (\ref{eq:five}) to cancel them out. This reduces the number of $\gamma_{i,j}$ with $i<n_0$ by one, and we can continue the process by induction to completely get rid of them. The case of $\gamma_{i,j}$ with $j>m_0$ is similar. Finally, using \cref{eq:one} and \cref{eq:two} we can make it so that all the morphisms $\gamma_{i,j}$, $\delta_k$ and $\sigma_t$ are presented in the order of their coefficients, i.e. such that for example $s^h\cong \sigma_{t_M}\circ\sigma_{t_M-1}\circ...\circ \sigma_{t_0}$ with $t_0\leq t_1\leq...\leq t_M$. This decomposition is easily seen to be unique, and we call it \textit{canonical}.\par
Now assume we have a morphism $F:\sq(S)\rightarrow\sq(S')$ in $\cfact^\op$. To prove the proposition it suffices to demonstrate that it admits a unique canonical decomposition as above. First, we can decompose it as $F_i\circ F_s$, where $F_i$ is injective and $F_s$ is surjective. Observe that a surjective morphism $F_s:\sq(S)\rightarrow\sq(S_0)$ sends $S$ onto $S_0\subset\sq(S_0)$, in particular it is easy to see that it is equal to a composition of morphisms of the form $\sigma^{v,h}_j$, so it follows that it admits a canonical decomposition. We can now assume $F$ to be injective. Let $\widetilde{S}$ be some maximal composable string of morphisms in $\sq(S')$ that contains $F(S)$, then $F$ factors as 
\[\sq(S)\xrightarrow{a_{\widetilde{S}}}\sq(\widetilde{S})\xrightarrow{i_{\widetilde{S}}}\sq(S'),\]
where $i_{\widetilde{S}}$ is a natural inclusion and $a_{\widetilde{S}}$ sends $S$ into $\widetilde{S}\subset\sq(\widetilde{S})$. It is easy to see that $a_{\widetilde{S}}$ is given by a composition of $\delta^{v,h}_i$ and $i_{\widetilde{S}}$ by a composition of $\gamma_{i,j}$. However, there are generally more than one $\widetilde{S}$ that satisfy this condition. We endow the set of maximal composable chain in $\sq(S')$ with a partial order for which $S_1<S_2$ if $S_1$ lies below $S_2$ in $\sq(S')\subset[n']\times[m']$. Let $F(0)=(l_1,l_2)$ and $F(n,m)=(k_1,k_2)$, it is easy to see that for any $\widetilde{S}$ the morphism $a_{\widetilde{S}}$ can be further decomposed as 
\[\sq(S)\xrightarrow{b_{S_0}}\sq(S_0)\xrightarrow{j_{S_0}}\sq(\widetilde{S}),\]
where $S_0$ is the unique maximal composable chain in $\sq(S')$ connecting $(l_1,l_2)$ and $(k_1,k_2)$ and containing $F(S)$, $j_{S_0}$ is the natural inclusion and $b_{S_0}$ is an injective morphism whose decomposition only contains $\delta^{h}_i$ for $i\neq0$ and $\delta^v_j$ for $j\neq m'$.\par
It is now easy to see that the minimal (with respect to  the partial order described above) maximal composable string $S_1$ that contains $F(S)$ can be described as follows: it starts at $(0,0)$ and coincides with $S'$ until its horizontal coordinate reaches $l_1$, then it contains a vertical segment which continues until it hits $F(0)=(l_1,l_2)$, it then coincides with $S_0$ until $(k_1,k_2)$, then it contains a horizontal segment until it hits $S'$ again and coincides with $S'$ until the endpoint $(n',m')$. In particular, it follows that $i_{S_1}$ dies nor contain $\gamma_{i,j}$ with $i<l_1$ and $j>l_2$, so in particular $i_{S_0}\circ a_{S_0}\circ F_s$ admits a canonical decomposition, and it is easy to see that it is uniquely determined by it.
\end{proof}
\begin{ex}
The diagram below illustrates the canonical decomposition of a morphism  $F:(2|1|1)\rightarrow(1|1|3|1|1|3)$ such that $F(0,0)=F(1,0)=(2,2)$, $F(2,0)=(3,2)$, $F(2,1)=(3,4)$ and $F(3,1)=(4,4)$:
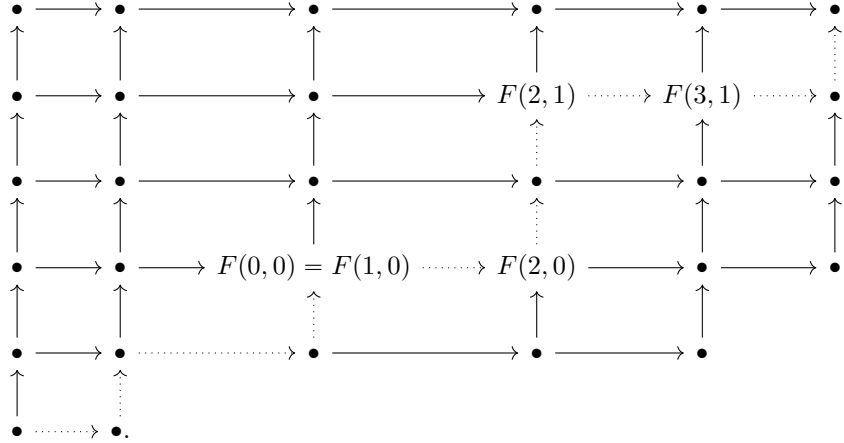
\[
\begin{tikzcd}
\bullet\arrow[r]&\bullet\arrow[r]&\bullet\arrow[r]&\bullet\arrow[r]&\bullet\arrow[r]&\bullet\\
\bullet\arrow[r]\arrow[u]&\bullet\arrow[r]\arrow[u]&\bullet\arrow[r]\arrow[u]&F(2,1)\arrow[r,dotted]\arrow[u]&F(3,1)\arrow[r,dotted]\arrow[u]&\bullet\arrow[u, dotted]\\
\bullet\arrow[r]\arrow[u]&\bullet\arrow[r]\arrow[u]&\bullet\arrow[r]\arrow[u]&\bullet\arrow[r]\arrow[u, dotted]&\bullet\arrow[r]\arrow[u]&\bullet\arrow[u]\\
\bullet\arrow[r]\arrow[u]&\bullet\arrow[r]\arrow[u]&F(0,0)=F(1,0)\arrow[r,dotted]\arrow[u]&F(2,0)\arrow[r]\arrow[u, dotted]&\bullet\arrow[r]\arrow[u]&\bullet\arrow[u]\\
\bullet\arrow[r]\arrow[u]&\bullet\arrow[r, dotted]\arrow[u]&\bullet\arrow[u, dotted]\arrow[r]&\bullet\arrow[r]\arrow[u]&\bullet\arrow[u]\\
\bullet\arrow[u]\arrow[r, dotted]&\bullet\arrow[u, dotted].
\end{tikzcd}
\]
The dotted line depicts the composable sequence $S_1$ in the notation of \cref{prop:generators}. The decomposition itself is given by
\[F\cong \gamma_{3,2}\circ \gamma_{2,2}\circ\gamma_{4,3}\circ\gamma_{3,3}\circ \gamma_{4,4}\circ\gamma_{3,4}\circ \delta^v_3\circ\delta^h_4\circ \delta^v_0\circ\delta^h_0\circ\delta^v_0\circ \delta^v_1\circ \sigma^h_0.\]
\end{ex}
\begin{notation}\label{not:factorization_arities}
We will call a morphism in $\cfact^\op$ \textit{active} if its canonical decomposition only contains morphisms of the form $\sigma^{v,h}_j$ and $\delta^h_i$ for $i\neq0$ and $i\neq n+1$ and $\delta^v_k$ for $k\neq 0$ and $k\neq m+1$. It follows from \cref{eq:one} and \cref{eq:two} that they are closed under composition and so form a subcategory, which we will denote by $\cfact^\act$. We will call a morphism \textit{inert} if its canonical decomposition only contains $\gamma_{i,j}$, $\delta^{v,h}_0$, $\delta^v_{m+1}$ and $\delta^h_{n+1}$. It once again follows from \cref{eq:one}, \cref{eq:two} and \cref{eq:five} that they form a subcategory $k:\cfact^\inrt\hookrightarrow\cfact$. We call an inert morphism an \textit{inclusion} if its canonical decomposition does not contain $\gamma_{i,j}$, and we denote the subcategory of inclusions by $j:\cfact^\inc\hookrightarrow\cfact^\inrt$. We will call an inert morphism a \textit{permutation} if its canonical decomposition only contains morphisms of the form $\gamma_{i,j}$.\par
We will call an object of $\cfact^\inrt$ \textit{elementary} if the corresponding string $S$ has length $l(S)\leq1$, we denote by $i:\cfact^\el\hookrightarrow\cfact^\inc$ the full subcategory on elementary objects. For an object $S\in\cfact$ we also denote by $\cfact^\el_{S/}$ the category $S/i$. Denote by $\cE_\cfact$ the full subcategory of $\mor_\cat(\cfact^\inrt,\cS)$ on those $\cF:\cfact^\inrt\rightarrow\cS$ for which 
\[j^*\cF\cong i_*i^*j^*\cF.\]
\end{notation}
\begin{prop}\label{prop:factorization_theory}
The active and inert morphisms form a factorization system on $\cfact$. Moreover, the subcategory $\cE_\cfact$ gives the pair $(k:\cfact^\inrt\hookrightarrow \cfact)$ the structure of a theory with arities. 
\end{prop}
\begin{proof}
The fact that active and inert morphisms form a factorization system follows immediately from \cref{prop:generators}. To prove the second claim we need to show that the monad $k^*k_!$ om $\mor_\cat(\cfact^\inrt,\cS)$ preserves $\cE$. For a given objects $S\in\cfact$ denote by $\Act_\cfact(S)$ the space of active morphisms with target $S$. Using the active/inert factorization, it can be considered as a functor $\Act_\cfact(-):\cfact^\inrt\rightarrow\cS$, although we would be more interested in its restriction to $\cfact^\inc$, which we would use the same notation for. Observe that it is a full subcategory of $k/S$ on active morphisms $S'\twoheadrightarrow S$ and moreover, since active and inert morphisms form a factorization system, we see that this subcategory is cofinal. It follows that
\[k^*k_!\cF(S)\cong\underset{(S'\rightarrow S)\in(k/s)}{\colim}\cF(S')\cong\underset{(S'\twoheadrightarrow S)\in\Act_\cfact(S)}{\colim}\cF(S').\]
Now let $S\bydef(i_0|j_0|...|n-i_k|m-j_k)$, then it follows easily from the definition of active morphisms that
\begin{align}\label{eq:seven}
    \Act_\cfact(S)\cong\Act_\ccat([n])\times\Act_\ccat([m]).
\end{align}
 Observe that we also have the following isomorphism
 \begin{align}\label{eq:eight}
     \cfact^\el_{S/}\cong \ccat^\el_{[i_0]/}\coprod_{(i_0,0)}\ccat^\el_{[j_0]/}\coprod_{(i_0,j_0)}...\coprod_{(n-i_k,m-j_k)}\ccat^\el_{/[n-i_k]}\coprod_{(n,m-j_k)}\ccat^\el_{/[m-j_k]}.
 \end{align}
It follows from \cref{eq:seven}, \cref{eq:eight} and the basic properties of $\ccat$ that 
\[\Act_\cfact(S)\cong\underset{(S\rightarrowtail e)\in\cfact^\el_{s/}}{\lim}\Act_\cfact(e).\]
We can also use \cref{eq:eight} and the relative Segal condition for $\ccat$ to conclude that for any factorization square 
\[
\begin{tikzcd}[row sep=huge, column sep=huge]
S'\arrow[r, two heads, "a"]\arrow[d, tail, "a^* i"]&S\arrow[d, tail, "i"]\\
i_!S'\arrow[r, two heads,"i_! a"]&e
\end{tikzcd}
\]
in which $i$ is an inclusion and for every $\cF:\cfact^\el\rightarrow\cS$ we have
\[\underset{(S'\rightarrowtail e')\in\cfact^\el_{S'/}}{\lim}\cF(e)\cong \underset{(S\overset{i}{\rightarrowtail} e)\in\cfact^\el_{S/}}{\lim}\underset{(i_!S'\rightarrowtail e_0)\in\cfact^\el_{i_!S'/}}{\lim}\cF(e_0).\]
Putting everything together, we get for every $\cF\in\cE_\cfact$
\begin{align*}
    k^*k_!\cF(S)\cong &\underset{(S'\twoheadrightarrow S)\in\Act_\cfact(S)}{\colim}\cF(S')\\
    \cong & \underset{\underset{(S\overset{i}{\rightarrowtail} e)\in\cfact^\el_{S/}}{\lim}\Act_\cfact(e)}{\colim}\underset{(S'\rightarrowtail e')\in\cfact^\el_{S'/}}{\lim}\cF(e')\\
    \cong&\underset{\underset{(S\overset{i}{\rightarrowtail} e)\in\cfact^\el_{S/}}{\lim}\Act_\cfact(e)}{\colim}\underset{(S\overset{i}{\rightarrowtail} e)\in\cfact^\el_{S/}}{\lim}\underset{(i_!S'\rightarrowtail e_0)}{\lim}\cF(e_0)\\
    \cong&\underset{\underset{(S\rightarrowtail e)\in\cfact^\el_{S/}}{\lim}\Act_\cfact(e)}{\colim}\underset{(S\overset{i}{\rightarrowtail} e)\in\cfact^\el_{S/}}{\lim}\cF(i_! S')\\
    \cong& \underset{(S\overset{i}{\rightarrowtail} e)\in\cfact^\el_{S/}}{\lim}\underset{(i_!S\twoheadrightarrow e)\in\Act_\cfact(e)}{\colim}\cF(i_!S)\\
    \cong&\underset{(S\overset{i}{\rightarrowtail} e)\in\cfact^\el_{S/}}{\lim}k^*k_!\cF(e)\cong i_*i^*k^*k_!\cF(S),
\end{align*}
where the isomorphism in the next to last row follows from \cite[Corollary 7.17]{chu2019homotopy}.
\end{proof}
\begin{notation}
We will denote the category of models $\modl_\cfact(\cS)$ by $\fact$.
\end{notation}
\begin{cor}\label{cor:active_permutation}
The morphisms $f$ in $\cfact$ for which the inert part of the active/inert decomposition is a permutation form a subcategory, which we will denote by $\cfact^\ap$.
\end{cor}
\begin{proof}
To prove the claim it suffice to prove that $\cfact^\ap$ is closed under composition. To prove it we need to demonstrate that if in a factorization square 
\[
\begin{tikzcd}[row sep=huge, column sep=huge]
S'\arrow[r, two heads, "a"]\arrow[d, tail, "a^* i"]&S\arrow[d, tail, "i"]\\
i_!S'\arrow[r, two heads,"i_! a"]&S''
\end{tikzcd}
\]
in $\cfact$ the morphism $i$ is a permutation, then so is $a^*i$. Since every permutation and every active morphism admits a canonical decomposition, it would suffice to prove the claim for $i$ of the from $\gamma_{i,j}$ and for $a$ equal to an elementary active morphism $\sigma^{v,h}_j$ or $\delta^{v,h}_i$, however in this case the claim follows immediately from the identities described in \cref{constr:elementary morphisms}. 
\end{proof}
\begin{prop}\label{prop:coverings}
For a given inert morphism $j:S\rightarrowtail S'$ denote by $\cfact^\inc/j$ the category whose objects are factorizations $j\cong i\circ j'$ in which $i$ is an inclusion and morphisms are commutative diagrams of the form 
\[
\begin{tikzcd}
{}&S_0\arrow[dr, tail, "i_0"]\arrow[dd,tail,  "i'"]\\
S\arrow[ur,tail,  "j_0"]\arrow[dr, tail, "j_1"]&{}&S'\\
{}&S_1\arrow[ur, tail, "i_1"]
\end{tikzcd}.
\]
Then this category admits an initial object.
\end{prop}
\begin{proof}
The required initial object is given by the inclusion $i:S_i\hookrightarrow S'$ of a minimal substring of $S'$ such that all the morphism in the image of $j$ in $\sq(S')$ admit a factorization which only includes morphisms from $S_i$. More explicitly it can be constructed as follows: assume that $S$ (resp. $S'$) is a maximal composable chain of morphisms in $[k]\times[l]$ (resp. in $[k']\times[l']$), assume that $j(0,0)=(s,t)$ and $j(k,l)=(u,v)$. We then define $(x,y)$ (resp. $(z,w)$) to be the points in $\sq(S')$ satisfying the following properties:
\begin{itemize}
    \item $(x,y)$ (resp. $(z,w)$) lies on $S'$;
    \item $x=s$ (resp. $w=v$);
    \item there are no points in $\sq(S')$ that satisfy the two conditions above and have lower (resp. higher) horizontal or vertical coordinates.
\end{itemize}
We then define $S_i$ to be the substring of $S'$ lying between $(x,y)$ and $(z,w)$. Since the image of $\sq(S)$ under $j$ lies in $\sq(S_i)\subset\sq(S')$ by construction, we also have a morphism $j_i:\sq(S)\rightarrowtail\sq(S_i)$ such that $j\cong i\circ j_i$.
\end{proof}
\begin{definition}\label{def:coverings}
We will call an inert morphism $j$ in $\cfact$ a \textit{covering} if the initial object of $\cfact^\inc/j$ described in \cref{prop:coverings} is given by the identity.
\end{definition}
\begin{prop}\label{prop:coverings_factorization}
Coverings and inclusions form a factorization system in $\cfact^\inrt$.
\end{prop}
\begin{proof}
\Cref{prop:coverings} provides us with a decomposition $j\cong i\circ j_i$ for an arbitrary inert morphism $j$, in which $i$ is an inclusion and $j_i$ is a covering, so to prove our claim it would suffice to prove that coverings form a subcategory of $\cfact^\inrt$. Since identity morphisms are obviously coverings, it remains to show that they are closed under composition. So assume that we have a composable chain $S_0\xinert{j_0}S_1\xinert{j_1}S_2$ in $\cfact$ in which both $J_0$ and $J_1$ are coverings, we need to prove that $j_1\circ j_0$ is a covering as well. Assume that $\sq(S_i)\subset [k_i]\times [l_i]$ for $i\in\{0,1,2\}$, then it follows from the definition that $j_1\circ j_0 (0,0)$ and $j_1(0,0)$ (resp. $j_1\circ j_0 (k_0,l_0)$ and $j_1(k_1,l_1)$) have the same horizontal (resp. vertical) coordinate, and so it follows from the construction of $j_i$ in \cref{prop:coverings} that if $j_1$ is a covering, then so is $j_1\circ j_0$.
\end{proof}
\begin{remark}
The morphisms given by a composition of an active morphism and a covering do \textit{not} form a subcategory of $\cfact$. For example, both $\gamma_{1,0}\circ \delta^v_0\circ \delta^h_1$ and $\sigma^v_0$ satisfy this property, however
\[\sigma^v_0\circ\gamma_{1,0}\circ \delta^v_0\circ \delta^h_1\cong \sigma^v_0\circ \delta^v_0\circ \delta^h_1\cong \delta^h_1\]
does not. It follows that the subcategory of inclusions in $\cfact$ is not a part of a factorization system.
\end{remark}
\begin{remark}
Observe that $\fact$ is a localization of $\cP(\cfact^\op)$ at the set of morphisms
\[\underset{(S\rightarrowtail e)\cfact^\el_{S/}}{\colim}h_e\rightarrow h_S,\]
where $S$ is an arbitrary object of $\cfact^\op$. In particular. it is an accessible localization of a presheaf category and it follows from \cite[Proposition 5.5.4.15.]{lurie2009higher} that it is presentable.
\end{remark}
\begin{prop}\label{prop:subcategories}
There is a functor $h:\ccat\rightarrow\cfact$ (resp. $v:\ccat\rightarrow\cfact$, $t:\ccat\rightarrow\cfact$) that sends $[0]$ to the unique string of length zero and sends $[1]$ to the string $h$ (resp. $v$, $vh$). Moreover, it induces a functor $h^*:\fact\rightarrow\cat$ (resp. $v^*:\fact\rightarrow\cat$, $t^*:\fact\rightarrow\cat$).
\end{prop}
\begin{proof}
The construction of $h$ and $v$ is similar and we only define the first one. By definition, $h$ sends $[n]$ to the string $S^h_n$ consisting of $n$ consecutive horizontal morphisms. To define $h$ on morphisms it suffices to describe $h(\delta_i)$ and $h(\sigma_j)$, so we define 
\[h(\sigma_j)\bydef \sigma^h_j,h(\delta_i)\bydef\delta^h_i.\]
To prove that this is indeed a functor it suffices to demonstrate that $h(\delta_i)$ and $h(\sigma_j)$ satisfy simplicial identities, however this follows directly from \cref{eq:one}. To prove the second claim, observe that for every $\cF\in\fact$ we have
\[\cF(S^h_n)\cong\overbrace{\cF(h)\times_{\cF(*)}\cF(h)\times_{\cF(*)}...\times_{\cF(*)}\cF(h)}^\text{$n$ times}\]
and so the Segal condition for $h^*\cF$ follows.\par
The case of $t$ is somewhat more difficult. We define $t([n])$ to be the string 
\[S^t_n\bydef\overbrace{vhvh...vh}^\text{$n$ times}.\]
A morphism of the form $\sigma_j$ is then sent to $\sigma^v_j\circ \sigma^h_j$, the morphisms $\delta_0$ and $\delta_{n+1}$ to $\delta^v_0\circ \delta^h_0$ and $\delta^h_{n+1}\circ \delta^v_{n+1}$ respectively and the morphisms of the form $\delta_i$ for $i\neq0,n+1$ to the composition
\[vh...vh...vh\xrightarrow{\delta^v_i\circ\delta^h_i}vh...vvhh...vh\xrightarrow{\gamma_{i,i}}vh...vhvh...vh.\]
The simplicial identities then follow relatively easily from the relations in \cref{constr:elementary morphisms}, we will only demonstrate the least trivial ones:
\begin{align*}
    t(\sigma_i\circ\delta_i)\cong\sigma^v_i\circ\sigma^h_i\circ\gamma_{i,i}\circ \delta^v_i\circ\delta^h_i\cong \sigma^h_i\circ\sigma^v_i\circ\gamma_{i,i}\circ \delta^v_i\circ\delta^h_i\cong \sigma^h_i\circ\sigma^v_i\circ\delta^v_i\circ\delta^h_i\cong \id\\
    t(\sigma_{i-1}\circ \delta_{i})\cong \sigma^v_{i-1}\circ\sigma^h_{i-1}\circ\gamma_{i,i}\circ \delta^v_i\circ\delta^h_i\cong\sigma^v_{i-1}\circ\sigma^h_{i} \circ \delta^v_i\circ\delta^h_i\cong \sigma^v_{i-1}\circ\sigma^h_{i} \circ \delta^h_i\circ\delta^v_i\cong \id.
\end{align*}
To prove the second claim in this case, observe that for every $\cF\in\fact$ we have
\begin{align*}
    \cF(S^t_n)\cong& \overbrace{\cF(v)\times_{\cF(*)}\cF(h)\times_{\cF(*)}\cF(v)\times_{\cF(*)}\cF(h)\times_{\cF(*)}...\times_{\cF(*)}\cF(v)\times_{\cF(*)}\cF(h)}^\text{$n$ times}\\
    \cong& \overbrace{(\cF(v)\times_{\cF(*)}\cF(h))\times_{\cF(*)}(\cF(v)\times_{\cF(*)}\cF(h))\times_{\cF(*)}...\times_{\cF(*)}(\cF(v)\times_{\cF(*)}\cF(h))}^\text{$n$ times},
\end{align*}
and so the Segal condition for $t^*\cF$ follows.
\end{proof}
\begin{theorem}\label{thm:factorization_model}
There is an equivalence between the category $\fact$ and the category whose objects are  categories with factorization systems in the sense of \cref{def:factorization_categories} and morphisms are morphisms preserving the factorization system. Moreover, for an object $\cF\in\fact$ the Segal spaces $h^*\cF$ and $v^*\cF$ can be identified under the above equivalence with the horizontal and vertical subcategories of $\cF$ respectively in the sense of \cref{def:factorization_categories}.
\end{theorem}
\begin{proof}
We temporarily denote by $\widetilde{\fact}$ the category described in the statement of the proposition. Observe that, for given an object $\cF\in\fact$, the Segal space $t^*\cF$ of \cref{prop:subcategories} admits a factorization system by definition. This defines a functor $F:\fact\rightarrow\widetilde{\fact}$.\par
Now assume that $\cC$ is a category with a factorization system $(\cH,\cV)$. We can consider it as a functor \[(\cC_\act:\cfact^\act\rightarrow\cS):(i_0|j_0|...|n-i_k|m-i_k)\mapsto\cH_{i_0}\times_{\cC_0}\cV_{j_0}\times_{\cC_0}...\times_{\cC_0}\cH_{n-i_k}\times_{\cC_0}\cV_{m-j_k}\]
for which every morphism of the form $\sigma^{v,h}_j$ or $\delta^{v,h}_i$ acts as a corresponding unit or multiplication map for $\cV$ and $\cH$ respectively. By the unique factorization property any morphism $f:S\rightarrow \cC$ that sends horizontal segments of $S$ to horizontal morphisms in $\cC$ and vertical segments to vertical morphisms admits a unique extension to the morphism $\widetilde{f}:\sq(S)\rightarrow\cC$ in $\widetilde{\fact}$ (using the same construction as in \cref{lem:restriction}). Observe that all inert morphisms $i:S\rightarrowtail S'$ in $\cfact^\op$ induce an inclusion of categories $\sq(S)\overset{i}{\hookrightarrow}\sq(S')$. This allows us to consider $\cC$ as a functor $\cC_\inrt:\cfact^\inrt\rightarrow\cS$ that sends every inert morphism $i$ as above to the morphism
\[(i^*:\mor(S',\cC)\rightarrow\mor(S,\cC)):(f:S'\rightarrow\cC)\mapsto(S\hookrightarrow\sq(S)\overset{i}{\hookrightarrow}\sq(S')\xrightarrow{\widetilde{f}}\cC).\]
We now need to show that $\cC_\act$ and $\cC_\inrt$ glue together to define a functor $\cC:\cfact\rightarrow\cS$. For that it suffices to show that for any factorization square
\[
\begin{tikzcd}[row sep=huge, column sep=huge]
S_0\arrow[r, two heads, "a"]\arrow[d, tail, "\widetilde{i}"]&S_1\arrow[d, tail, "i"]\\
S_2\arrow[r, two heads,"\widetilde{a}"]&S_3
\end{tikzcd}
\]
in $\cfact^\op$ we have $\widetilde{i}^*\circ \widetilde{a}^*\cong a^*\circ i^*$. Using \cref{prop:generators} we can express each of the morphism in this diagram as a composition of the elementary morphisms described in \cref{constr:elementary morphisms}, so it suffices to prove this statement under the assumption that $\widetilde{a}$ and $\widetilde{i}$ are elementary morphisms. In this case the claim reduces to checking the relations described in the \cref{constr:elementary morphisms}. Observe that if $\widetilde{i}$ is an elementary inert morphism of the form $\sigma^{v,h}_0$ or $\sigma^{v,h}_{m+1,n+1}$, then the relations we need to check reduce to simplicial identities, which hold since $\cH$ and $\cV$ were assumed to be subcategories of $\cC$. It remains to prove the claim for $\widetilde{i}=\gamma_{i,j}$. In this case we are reduced to checking the equivalences in \cref{eq:four}. Observe that the third and fourth as well as the seventh and eights equivalences in \cref{eq:four} are trivial. Proving that $\sigma^{h,*}_i\cong \sigma^{h,*}_{i-1}\circ \gamma^*_{j,i}$ amounts to proving that the following diagram
\[
\begin{tikzcd}[row sep=huge, column sep=huge]
a\arrow[r, equal]&b\\
c\arrow[u, rightharpoondown, "v"]\arrow[r, equal]&d\arrow[u, rightharpoondown, "v"]
\end{tikzcd}
\]
is a factorization square for any vertical morphism $v$, which is obvious. The isomorphism $\sigma^{v,*}_j\cong \sigma^{v,*}_{j}\circ \gamma^*_{j,i}$ reduces to an analogous statement for the horizontal subcategory. Similarly, the proving isomorphism $\gamma^*_{i+1,j}\circ \gamma^*_{i,j}\circ \delta^{h,*}_i\cong\delta^{h,*}_i\circ \gamma^*_{i,j}$ amounts to proving that for a pair of factorization squares
\[
\begin{tikzcd}[row sep=huge, column sep=huge]
a\arrow[r, rightharpoonup,"h'_1"]&b\arrow[r, rightharpoonup, "h'_2"]&c\\
d\arrow[u, rightharpoondown, "v''"]\arrow[r, rightharpoonup,"h_1"]&e\arrow[u, rightharpoondown, "v'"]\arrow[r, rightharpoonup,"h_2"]&f\arrow[u, rightharpoondown, "v"]
\end{tikzcd}
\]
in $\cC$ the diagram 
\[
\begin{tikzcd}[row sep=huge, column sep=huge]
a\arrow[r, rightharpoonup,"h'_2\circ h'_1"]&c\\
d\arrow[u, rightharpoondown, "v''"]\arrow[r, rightharpoonup,"h_2\circ h_1"]&f\arrow[u, rightharpoondown, "v"]
\end{tikzcd}
\]
is also a factorization square, which follows easily from the uniqueness of factorization. Thus we have defined a functor $G:\widetilde{\fact}\rightarrow\fact$.\par
Finally, the fact that $F$ and $G$ are inverse to each other as well as the last claim of the proposition follow easily from the construction.
\end{proof}
\begin{prop}\label{lem:taotal_cofinality}
The morphism $t:\ccat\rightarrow \cfact$ of \cref{prop:subcategories} is cofinal.
\end{prop}
\begin{proof}
To prove the claim we need to show that for every object $S\in\cfact$ the category $(S/t)$ is contractible. The objects of this category are given by morphisms $t([n])\xrightarrow{f}\sq(S)$ in $\fact$. Since every morphism in $\sq(S)$ admits a unique decomposition $f\cong h'\circ v'$ with a horizontal morphism $h'$ and a vertical morphism $v'$, we see that a morphism $t([n])\rightarrow\sq(S)$ in $\fact$ can be identified with a morphism $[n]\rightarrow\sq(S)$ in $\cat$. Every morphism $[n]\xrightarrow{f}\sq(S)$ admits a unique factorization
\[[n]\xactive{f_s}[n_i]\xrightarrow{f_i}\sq(S)\]
in which $f_s$ is a surjective morphism and $f_i$ is injective. It follows that the full subcategory $(S/t)^\inj\hookrightarrow(S/t)$ on injective morphisms is cofinal. This in turn implies that the classifying spaces for both categories are homotopy equivalent, so it now suffices to prove that $(S/t)^\inj$ is contractible.\par
An object of $(S/t)^\inj$ can be identified with a sequence of distinct points $(a_1,a_2,...,a_n)$ in $\sq(S)$ with non-decreasing horizontal and vertical coordinates. There is a unique morphism from $(a_1,a_2,...,a_n)$ to $(b_1,b_2,...,b_m)$ if $\{a_1,...,a_n\}$ is a subset of $\{b_1,...,b_m\}$. Denote by $(S/t)^\act$ the full subcategory of $(S/t)^\inj$ on those sequences that contain both initial and final objects of $\sq(S)$. For an arbitrary object $\overline{a}$ of $(S/t)^\inj$ denote by $L\overline{a}$ the sequence obtained from $\overline{a}$ by adding the final and initial objects if they do not belong to $\overline{a}$ already. It is obvious from the description of $(S/t)^\inj$ above that $L$ is a left adjoint to the inclusion $(S/t)^\act\hookrightarrow(S/t)^\inj$. In particular, the classifying spaces of $(S/t)^\inj$ and $(S/t)^\act$ are also equivalent. Finally, observe that $(S/t)^\act$ contains an initial object given by the sequence containing only the initial object and the final object of $\sq(S)$, so $(S/t)^\act$ is contractible.
\end{proof}
\begin{lemma}\label{lem:trivial_factorization}
Both $h_!$ and $v_!$ of \cref{prop:subcategories} restrict to functors 
\[h_!,v_!:\cat\rightarrow\fact.\]
\end{lemma}
\begin{proof}
We will only prove the statement for $h$ leaving the other case to the reader. By definition for $\cC\in\cat$ we have
\[h_!\cC(S)\cong\underset{(\sq(S)\rightarrow h([n])\in(h/S)}{\colim}\cC([n]).\]
Observe that every morphism $\sq(S)\rightarrow h(n)$ uniquely factors as 
\[\sq(S)\rightarrow h([n_S])\xrightarrow{h(f)}[n],\]
where $[n_S]$ is an interval obtained from $S$ by contracting all vertical morphisms and $f$ is a morphism in $\ccat$. It follows that 
\[h_!\cC(S)\cong \cC([n_S]).\]
The fact that it satisfies the Segal condition now follows elementary.
\end{proof}
\begin{remark}\label{rem:products}
It follows from \cref{lem:trivial_factorization} that we have a morphism 
\[(h_!,v_!):\cat\times\cat\rightarrow\fact.\]
Explicitly, for a pair $(\cC,\cD)\in\cat\times\cat$ and $\sq(S)\subset[n]\times[m]$ we have
\[(h_!,v_!)(\cC,\cD)(S)\cong \cC_n\times\cD_m.\]
This factorization system can be identified with a natural factorization system on $\cC\times\cD$ having horizontal morphisms of the form $(f,\id_d)$ and vertical morphisms of the form $(\id_c,g)$, and so we will often implicitly identify $\cC\times\cD$ with an object of $\fact$.
\end{remark}
\section{Lax functors and distributive laws}\label{sect:three}
The goal of this section is to define lax functors from categories with factorization systems and prove their basic properties. We define this notion in a few stages. First, we define the notion of a bicategory with a factorization system in \cref{def:bicategory_factorizatoin} and prove that the category $2-\fact$ of bicategories with factorization systems is also equivalent to a category of models of a certain theory in \cref{prop:bicategory_comparison}. We then construct an endofunctor $\cL$ of $2-\fact$ in \cref{constr:lax_factorization}. The actual construction is somewhat complicated, but roughly speaking $\cL\cF$ is obtained from $\cF$ by adding 2-morphisms $\gamma_{h,v}:v\circ h\rightarrow \widetilde{h}\circ \widetilde{v}$ for every factorization $v\circ h\cong \widetilde{h}\circ \widetilde{v}$ in $\cF$. Our next step is the construction of $L^\lax:\fact\rightarrow 2-\fact$ in \cref{constr:lax_functors1}. This functor adds the "compositor" and "unitor" 2-morphisms to all morphisms and objects in the vertical and horizontal subcategories of $\cF$. Finally, we define $\cL^\lax$ to be the composition $\cL L^\lax$ and define a lax functor from $\cF$ to a bicategory $\cB$ to be a morphism $\cL^\lax\cF\rightarrow \cB$.\par
The main result of this section is \cref{prop:lax_for_objects_of_fact}, which gives an explicit description of $\cL^\lax\sq(S)$ for objects $S\in\cfact$. Unfortunately, the proof is quite technical and takes up most of the section. Near the end of the section we use it to define a distributive law between monads in \cref{def:distributive_law}.
\begin{construction}\label{constr:2-fact}
Denote by $2-\cfact$ the category $\cfact\times \Delta^\op$. Observe that it has a natural active/inert factorization for which a morphism is active (resp. inert) if both of its components in $\cfact$ and $\Delta^\op$ are active (resp. inert) . We also declare a morphism an inclusion if its projection to $\cfact$ is an inclusion and its projection to $\Delta^\op$ is inert, we denote by $j:2-\cfact^\inc\hookrightarrow2-\cfact$ the subcategory with the same objects and inclusions as morphisms. Finally, we declare an object to be elementary if it is of the form $(v,[i])$ for $i\in\{0,1\}$, $(h,[j])$ for $j\in\{0,1\}$ or $(*,[0])$, we denote by $i:2-\cfact^\el\hookrightarrow2-\cfact^\inc$ the inclusion of the full subcategory on elementary objects. Finally, we denote by $2-\fact$ the full subcategory of $\mor_\cat(\cfact,\cS)$ on those $\cF$ for which
\[j^*\cF\cong i_*i^*j^*\cF.\]
\end{construction}
\begin{definition}\label{def:bicategory_factorizatoin}
For a given bicategory $\cB$ and a positive integer $n$ denote by $\arr^{n,\id}(\cB)$ the Segal space $i_n^*\cB$, where $i_n:\{n\}\times\Delta^\op\hookrightarrow\Delta^\op\times\Delta^\op$ is the natural inclusion. In other words, the objects of $\arr^{n,\id}(\cB)$ are given by composable sequences 
\[b_0\xrightarrow{f_1}b_1\xrightarrow{f_2}b_2\xrightarrow{f_3}...\xrightarrow{f_n}b_n\]
and morphisms by diagrams
\[
\begin{tikzcd}
b_0\arrow[r, "f_1"{name=O1, swap}, bend right=30]\arrow[r, "g_1"{name=C1}, bend left=30]&b_1\arrow[r, "f_2"{name=O2, swap}, bend right=30]\arrow[r, "g_2"{name=C2}, bend left=30]&b_2\arrow[r, "..."{name=O3, description}, bend right=30]\arrow[r, "..."{name=C3, description}, bend left=30]&b_n\arrow[Rightarrow, from=C1,to=O1, "\gamma_1"]\arrow[Rightarrow, from=C2,to=O2, "\gamma_2"]\arrow[Rightarrow, from=C3,to=O3]
\end{tikzcd}.
\]
Assume that $\cB$ has a pair of subcategories $(\cH,\cV)$, denote by $\arr^\fact(\cH,\cV)$ the  subcategory of $\arr^{2,\id}(\cB)$ on objects of the form $h\circ v$ with $v\in\cV$ and $h\in\cH$ and morphisms of the form $\gamma_h*\gamma_v$ with $\gamma_h\in\cH$ and $\gamma_v\in\cV$. Observe that the composition of morphisms induces a functor \[m:\arr^\fact(\cH,\cV)\rightarrow\arr^\id(\cB).\]
We will say that $\cH$ and $\cV$ form a \textit{factorization system} on $\cB$ if $m$ is an isomorphism. If $\cE$ is another bicategory with factorization system $(\cL,\cM)$ and $F:\cB\rightarrow\cE$ is a functor, we say that $F$ preserves the factorization system if it sends $\cH$ to $\cL$ and $\cV$ to $\cM$. We will denote by $\widetilde{2-\fact}$ the subcategory of $2-\cat$ on  bicategories with factorization systems and functors that preserve the factorization systems.
\end{definition}
\begin{prop}\label{prop:bicategory_comparison}
There is an equivalence
\[\widetilde{2-\fact}\xrightarrow{\sim}2-\fact.\]
\end{prop}
\begin{proof}
First, we define a functor $G:2-\fact\rightarrow\widetilde{2-\fact}$. The construction is similar to \cref{prop:subcategories}: we denote by $\widetilde{t}$ the morphism 
\[(t,\id):\Delta^\op\times\Delta^\op\rightarrow \cfact\times \Delta^\op\cong 2-\cfact.\]
It is easy to check that $\widetilde{t}^*$ restricts to a functor $2-\fact\rightarrow2-\cat$ and moreover for any $\cF\in2-\fact$ we have an isomorphism of Segal spaces
\[\widetilde{t}^*\cF_1\cong \cF_v\times_{\cF_0}\cF_h,\]
where $\cF_h$ (resp. $\cF_v$) is a Segal space given by the restriction of $\cF$ to $\{h\}\times\Delta^\op\subset2-\fact$ (resp. $\{v\}\times\Delta^\op\subset2-\fact$). It follows that $\widetilde{t}^*\cF$ is an object of $\widetilde{2-\fact}$.\par
We now construct a functor $H:\widetilde{2-\fact}\rightarrow 2-\fact$. First, for any object $(S,l)$ of $2-\cfact$, where $S$ is a string with $n$ horizontal and $m$ vertical segments, denote by $\sq(S,l)$ the the full subcategory of $\theta_{n,l}\times\theta_{m.l}$ (where $\theta_{n,l}$ denotes the category $\theta(n;l,l,...,l)$ in the notation of \cref{constr:theta_factorization}) on objects that lie above $S$, these bicategories can then be considered as objects of $\widetilde{2-\fact}$. Observe that every morphism $f:(S,l)\rightarrow(S',l')$ in $2-\cfact^\op$ induces a morphism $\sq(S,l)\rightarrow\sq(S',l')$, in particular for every $\cF\in\widetilde{2-\fact}$ we can define $H(\cF):2-\cfact\rightarrow\cS$ by
\[H(\cF):(S,l)\mapsto\mor_{\widetilde{2-\fact}}(\sq(S,l),\cF).\]
Observe that every morphism $\sq(S,l)\rightarrow\cF$ in $\widetilde{2-\fact}$ is uniquely determined by its restriction to $(S,l)\subset\sq(S,l)$ (where $(S,l)$ denotes the full subcategory on objects belonging to $S\subset [n]\times[m]$), it follows easily from this that $H(\cF)$ satisfies the Segal condition for $2-\fact$. Finally, it is elementary to check that $H$ and $G$ are inverse to each other. 
\end{proof}
\begin{remark}\label{rem:biproducts}
A similar argument to \cref{lem:trivial_factorization} shows that for morphisms \[\widetilde{h}\bydef(h,\id):\Delta^\op\times\Delta^\op\rightarrow2-\cfact\]
and
\[\widetilde{v}\bydef(v,\id):\Delta^\op\times\Delta^\op\rightarrow2-\cfact\]
the morphisms $\widetilde{h}_!$ and $\widetilde{v}_!$ restrict to morphisms from $2-\cat$ to $2-\fact$. In particular, just like in \cref{rem:products} we can consider \[(\widetilde{h}_!,\widetilde{v}_!):2-\cat\times2-\cat\rightarrow2-\fact\]
which endows a product $\cB\times\cE$ of bicategories with the natural factorization system.
\end{remark}
\begin{construction}\label{constr:lax_factorization}
Given an object $(S,l)\in2-\cfact$, denote by $\sq^\lax(S,l)$ the full subcategory of $\theta_{n,l}\otimes\theta_{m,l}$ on objects that lie above $S$ considered as a path from $(0,0)$ to $(n,m)$ in $[n]\times[m]$. It is easy to see that every morphism $(S,l)\rightarrow(S',l')$ induces a functor $\sq^\lax(S,l)\rightarrow\sq^\lax(S',l')$, so $\sq^\lax(-)$ defines a functor $2-\cfact^\op\rightarrow 2-\cat$. We can consider this as a correspondence $\widetilde{\cL}:2-\cfact^\op\nrightarrow \Delta^\op\times \Delta^\op$ by viewing $2-\cat$ as a full subcategory of $\cP(\Delta\times\Delta)$ on the presheaves satisfying the Segal condition. Finally, we denote by $\cL$ the composition 
\[2-\fact\xrightarrow{\widetilde{\cL}}\cP(\Delta\times\Delta)\xrightarrow{L}2-\cat, \]
where $L$ denotes the left adjoint to the fully faithful inclusion $2-\cat\hookrightarrow \cP(\Delta\times\Delta)$ (which exists since $2-\cat$ is an accessible localization of $\cP(\Delta\times\Delta)$).
\end{construction}
\begin{lemma}\label{lem:colimit_lax_factorization}
For any $\cF\in2-\fact$ and $\cB\in2-\cat$ we have
\[\mor_{2-\cat}(\cL\cF,\cB)\cong\underset{(\sq(S,l)\rightarrow \cF)\in2-\cfact/\cF}{\lim}\mor_{2-\cat}(\sq^\lax(S),\cB).\]
\end{lemma}
\begin{proof}
First, observe that we have the following isomorphism in $\cP(2-\cfact^\op)$:
\[\cF\cong \underset{(\sq(S,l)\rightarrow \cF)\in2-\cfact/\cF}{\colim}\cF(S)\times h_{\sq(S)}.\]
Now the result follows from the following sequence of isomorphisms
\begin{align*}
    \mor_{2-\cat}(\cL\cF,\cB)\cong& \mor_{2-\cat}(L\widetilde{\cL}\cF,\cB)\\
    \cong&\mor_{\cP(\Delta\times\Delta)}(\widetilde{\cL}\cF,\cB)\\
    \cong&\mor_{\cP(\Delta\times\Delta)}(\widetilde{\cL}\underset{(\sq(S,l)\rightarrow \cF)\in2-\cfact/\cF}{\colim}\cF(S)\times h_{\sq(S)},\cB)\\
    \cong&\mor_{\cP(\Delta\times\Delta)}(\underset{(\sq(S,l)\rightarrow \cF)\in2-\cfact/\cF}{\colim}\cF(S)\times \sq^\lax(S),\cB)\\
    \cong&\underset{(\sq(S,l)\rightarrow \cF)\in2-\cfact/\cF}{\lim}\mor_{\cP(\Delta\times\Delta)}(\sq^\lax(S),\cB)\\
    \cong&\underset{(\sq(S,l)\rightarrow \cF)\in2-\cfact/\cF}{\lim}\mor_{2-\cat}(\sq^\lax(S),\cB).
\end{align*}
\end{proof}
\begin{notation}\label{not:gray_product}
For a pair of twofold Segal spaces $(\cB,\cE)\in(2-\cat)^{\times2}$ denote by $\cB\otimes\cE$ the twofold Segal space $\cL(\cB\times\cE)$, where we consider $\cB\times\cE$ to be an object of $2-\fact$ through the use of \cref{rem:biproducts}.
\end{notation}
\begin{cor}\label{cor:gray_colimit}
For $(\cB,\cE)\in(2-\cat)^{\times2}$ we have an isomorphism
\[\cB\otimes\cE\cong\underset{(\theta_{n,m}\rightarrow\cB)\in2-\ccat_{/\cB}}{\colim}\underset{(\theta_{s,t}\rightarrow\cE)\in2-\ccat_{/\cE}}{\colim}(\cB_{n,m}\times\cE_{s,t})\times(\theta_{n,m}\otimes\theta_{s,t}).\]
\end{cor}
\begin{proof}
Indeed, this follows from the definition of $\cB\otimes\cE$, \cref{lem:colimit_lax_factorization} and the fact that $2-\cat\times2-\cat$ is generated under colimits by $\theta_{n,m}\times\theta_{s,t}$.
\end{proof}
\begin{remark}
Since $\cB\otimes\cE$ commutes with colimits in both variables by \cref{cor:gray_colimit} and coincides with the usual Gray product (constructed for $\infty$-categories in \cite{gagna2020gray}) on $(\theta_{n,m},\theta_{s,t})$, it follows that it coincides with the usual Gray tensor product for all pairs $(\cB,\cE)$.
\end{remark}
\begin{prop}\label{prop:gray_universal_property}
For any pair $(\cB,\cE)\in(2-\cat)^{\times2}$ and any $\cT\in2-\cat$ there is an isomorphism
\[\mor_{2-\cat}^\lax(\cB\otimes\cE,\cT)\cong \mor_{2-\cat}^\lax(\cB,\mor_{2-\cat}^\lax(\cE,\cT))\]
in the notation of \cite[Definition 4.6.]{kositsyn2021completeness}.
\end{prop}
\begin{proof}
First, observe that it follows from \cref{cor:gray_colimit} that it suffices to prove the statement for pairs of the form $(\theta_{n,m},\theta_{s,t})$. Moreover, it follows from the colimit decomposition of \cite[Lemma 4.5.]{kositsyn2021completeness} that we can further restrict our attention to $(n,m,s,t)\in\{0,1\}^4$, in which case the result becomes elementary.
\end{proof}
\begin{construction}\label{constr:theta_factorization}
Recall from \cite{haugseng2018equivalence} the following description of the category $\Theta_2$: its objects are given by pairs $([n],a_1,a_2,...,a_n)$ where all $a_i$ are non-negative integers, a morphism from $([n],a_1,a_2,...,a_n)$ to $([m],b_1,b_2,...,b_m)$ is given by a morphism $f:[n]\rightarrow [m]$ in $\Delta$ together with morphisms $f_{i,j}:[a_i]\rightarrow[b_j]$ for $1\leq i\leq n$ and $f(i)\leq j\leq f(i+1)$. To every such object we can associate a bicategory $\theta(n;a_1,...,a_n)$ with $n+1$ objects such that $\mor_{\theta(n,a_1,...,a_n)}(i,j)\cong [a_i]\times[a_{i+1}]\times...\times[a_j]$ if $i\leq j$ and is empty otherwise. Now assume we have an object $S\in\cfact$ that contains $n$ horizontal segments and $m$ vertical segments together with non-negative integers $(h_1,..., h_n)$ and $(v_1,..., v_m)$, to this data we can associate a bicategory $\sq(S;h_1,...,h_n;v_1,...,v_m)$ which is by definition the full subcategory of $\theta(n;h_1,...,h_n)\times \theta(m,v_1,...,v_m)$ on objects lying above $S$ viewed as a path from $(0,0)$ to $(n,m)$. Observe that $\sq(S;h_1,...,h_n;v_1,...,v_m)$ has a natural structure of an object of $2-\fact$, we will denote by $(\Theta^\cfact_2)^\op$ the full subcategory of $2-\fact$ on objects of the form $\sq(S;h_1,...,h_n;v_1,...,v_m)$. If $S=e_1e_2...e_{n+m}$, where $e\in\{h,v\}$, we will also sometimes denote $(S;h_1,...,h_n;v_1,...,v_m)$ by $(S;l_1,l_2,...,l_{n+m})$, where $l_i$ is the number corresponding to the segment $e_k$.
\end{construction}
\begin{lemma}
The category $\Theta^\cfact_2$ admits an active/inert factorization system.
\end{lemma}
\begin{proof}
First, observe that a morphism \[f:\sq(S;h_1,...,h_n;v_1,...,v_m)\rightarrow\sq(S';h'_1,...,h'_{n'};v'_1,...,v'_{m'})\]
in $(\Theta^\cfact_2)^\op$ can equivalently be described by the data of a morphism $\overline{f}:\sq(S)\rightarrow\sq(S')$ in $\cfact^\op$ together with morphisms $f_{i,j}:[h_i]\rightarrow [h'_j]$ for $0\leq i\leq n$ and $\overline{f}_h(i)\leq j\leq \overline{f}_h(i+1)$ and $g_{k,t}:[v_k]\rightarrow [v_t]$ for $0\leq k\leq m$ and $\overline{f}_v(k)\leq t\leq \overline{f}_v(k+1)$ (where $f_{h,v}$ denotes the horizontal and vertical projections of $f$). We declare a morphism $f$ to be inert (resp. active) if $\overline{f}$ is inert (resp. active) as a morphism in $\cfact^\op$ and all morphisms $f_{i,j}$ and $g_{s,t}$ are inert (resp. active) as morphisms in $\Delta$. The claim now follows from the existence of the active/inert factorization systems in $\Delta$ and $\cfact^\op$.
\end{proof}
\begin{notation}\label{not:theta_factorization_theory}
We will denote by $k:(\Theta_2^\cfact)^\inrt\hookrightarrow\Theta_2^\cfact$ the inclusion of the subcategory containing inert morphisms. We will call an inert morphism $(i;i_1,...,i_n;j_1,...,j_m)$ an inclusion if the morphism $i$ in $\cfact$ is an inclusion, we denote by $j:(\Theta_2^\cfact)^\inc\hookrightarrow(\Theta_2^\cfact)^\inrt$ the subcategory of inclusions. We declare an object $(e;s)$ to be elementary if $e\in\cfact^\el$ and $s\in\{0,1\}$ and denote by $i:(\Theta_2^\cfact)^\el\hookrightarrow(\Theta_2^\cfact)^\inc$ the full subcategory on elementary objects. Finally, we denote by $\cE_{\Theta_2^\cfact}$ the full subcategory of $\mor_\cat((\Theta_2^\cfact)^\inrt,\cS)$ on those functors $\cF:(\Theta^\cfact_2)^\inrt\rightarrow \cS$ for which
\[j^*\cF\cong i_*i^*j^*\cF.\]
\end{notation}
\begin{lemma}\label{lem:double_active}
For a given $S\in\cfact$ denote by $\cfact^\act_{/S}$ the full subcategory of $\cfact_{/S}$ on active morphisms, then we have
\[\cfact^\act_{/S}\cong\underset{(S\rightarrowtail e)\in\cfact^\el_{/S}}{\lim}\cfact^\act_{/e}.\]
\end{lemma}
\begin{proof}
First, observe that $\Act_{\cfact}(S)$ is the underlying space of $\cfact^\act_{/S}$ and we have already proved the claim for $\Act_\cfact(S)$ in \cref{prop:factorization_theory}. It would now suffice to provide an isomorphism on morphisms. Observe that a morphism in $\cfact^\act_{/S}$ is by definition a diagram of the form 
\[
\begin{tikzcd}[row sep=huge, column sep=huge]
S_1\arrow[dr, "a_1"]\arrow[rr, "a"]&{}&S_2\arrow[dl, "a_2" swap]\\
{}&S
\end{tikzcd}.
\]
This diagram can be identified with an object of $\Act^2_\cfact(S)$ - the space of composable pairs of active morphisms in $\cfact$, so it suffices to prove that it is also a Segal space. This follows from the following sequence of isomorphisms
\begin{align*}
    \Act^2_\cfact(S)\cong&\underset{(S'\twoheadrightarrow S)\in\Act_\cfact(S)}{\colim}\Act_\cfact(S')\\
    \cong&\underset{\underset{(S\overset{i}{\rightarrowtail} e)}{\lim}(i_!S'\twoheadrightarrow e)\in\Act_\cfact(e)}{\colim}\Act_\cfact(S')\\
    \cong&\underset{\underset{(S\overset{i}{\rightarrowtail} e)}{\lim}(i_!S'\twoheadrightarrow e)\in\Act_\cfact(e)}{\colim}\underset{(S\overset{i}{\rightarrowtail} e)}{\lim}\Act_\cfact(i_! e)\\
    \cong & \underset{(S\overset{i}{\rightarrowtail} e)}{\lim} \underset{(S''\twoheadrightarrow e)\in\Act_\cfact(e)}{\colim}\Act_\cfact(S'')\\
    \cong&\underset{(S\overset{i}{\rightarrowtail} e)}{\lim} \Act^2_{\cfact}(e),
\end{align*}
where the first isomorphism follows from the definition of $\Act^2_\cfact(-)$, the second since $\Act_\cfact(-)$ satisfies the Segal condition, the third from the relative Segal condition of \cref{prop:factorization_theory}, the fourth from \cite[Corollary 7.17.]{chu2019homotopy} and the last again from the definition.
\end{proof}
\begin{prop}\label{prop:model_equivalence}
The constructions of \cref{not:theta_factorization_theory} endow $\Theta_2^\cfact$ with a structure of a theory with arities in the sense of \cite[Definition 6.2.]{kositsyn2021completeness}, moreover the category of models for $\Theta^\cfact_2$ is isomorphic to $2-\cfact$.
\end{prop}
\begin{proof}
To prove the first claim we need to show that the endomorphism $k^*k_!$ of $\mor_\cat((\Theta_2^\cfact)^\inrt,\cS)$ sends $\cE_{\Theta^\cfact_2}$ to itself. Observe that since active and inert morphisms form a factorization system on $\Theta_2^\cfact$ we have
\[k^*k_!\cF(S;\overline{n};\overline{m})\cong \underset{((S';\overline{n}';\overline{m}')\rightarrow(S;\overline{n};\overline{m}))\in(k/(S;\overline{n};\overline{m}))}{\colim}\cF(S';\overline{n}';\overline{m}')\cong \underset{((S';\overline{n}';\overline{m}')\twoheadrightarrow(S;\overline{n};\overline{m}))\in\Act_{\Theta_2^\cfact}(S;\overline{n};\overline{m})}{\colim}\cF(S';\overline{n}';\overline{m}').\]
Assume $S=e_1e_2...e_{n+m}$, then it is easy to see that
\begin{align}\label{eq:nine}
    (\Theta^\cfact_2)^\el_{(S;l_1,...,l_{n+m})/}\cong \ccat^\el_{[l_1]/}\coprod_* \ccat^\el_{[l_1]/}\coprod_*...\coprod_*\ccat^\el_{[l_{n+m}]/}.
\end{align}
Observe that by the definition of an active morphism $\Act_{\Theta^\cfact_2}(e,l)$ for $e\in\{h,v\}$ is isomorphic to the space of pairs $(a:[1]\twoheadrightarrow[s],\{b_i:[l]\twoheadrightarrow [t_i]\}|_{i\in\{1,2,...,s\}})$. Now observe that, since both $\Act_\cfact(S)$ and $\Act_\ccat([n])$ satisfy the respective Segal conditions, we have
\[\Act_{\Theta_2^\cfact}(S;\overline{l})\cong \Act_{\Theta_2^\cfact}(e_1,1)^{\times l_1}\times \Act_{\Theta_2^\cfact}(e_2,1)^{\times l_2}\times...\times \Act_{\Theta_2^\cfact}(e_{n+m},1)^{\times l_{n+m}}.\]
It follows from \cref{eq:nine} that this is indeed a $\Theta_2^\cfact$ Segal space. Now assume that 
\[
\begin{tikzcd}[row sep=huge, column sep=huge]
(S';\overline{l}')\arrow[r, two heads, "a"]\arrow[d, tail, "a^* i"]&(S;\overline{l})\arrow[d, tail, "i"]\\
i_!(S';\overline{l}')\arrow[r, two heads,"i_! a"]&(e;1)
\end{tikzcd}
\]
is a factorization square in $\Theta_2^\cfact$ in which $i$ corresponds to the inclusion of an elementary 2-cell in $S$. Assume that the restriction of $a$ to $e$ is given by a pair $([1]\twoheadrightarrow[s];\{[1]\twoheadrightarrow [t_i]\}|_{i\in\{1,2,...,s\}})$ and that $a\circ i(0)=j$, then $i_!(S';\overline{l}')$ is a subcategory of $(S';\overline{l}')$ containing all objects between $j$ and $j+n$ and all 2-cells between $j+i$ and $j+i+1$ that lie in the subinterval $[t_i]\subset[l_{j+i}]$. Those subcategories cover $(S';\overline{l}')$ and only intersect along 0-cells and 1-cells. It follows from \cref{eq:nine} that the natural morphism
\[\underset{((S;\overline{l})\overset{i}{\rightarrowtail} e)\in(\Theta^\cfact_2)^\el_{(S;\overline{l})/}}{\lim}\underset{(i_!(S';\overline{l}')\rightarrowtail e_0)\in(\Theta^\cfact_2)^\el_{i_!(S';\overline{l}')/}}{\lim}\cF(e_0)\rightarrow\underset{((S';\overline{l}')\rightarrowtail e')\in(\Theta^\cfact_2)^\el_{(S';\overline{l}')/}}{\lim}\cF(e)\]
is an isomorphism for all functor $\cF:\Theta_2^\cfact\rightarrow\cS$. With this we can conclude the proof of the first claim just like in \cref{prop:factorization_theory}.\par
To prove the claim about the category of models, we first observe that there is a natural functor $q:2-\cfact\rightarrow \Theta^\cfact_2$ that sends $(S,n)$ to \[(S;\overbrace{n,n,...,n}^\text{n+m times}).\]
We will prove that $q_!$ restricts to an isomorphism $\fact\xrightarrow{\sim}\modl_{\Theta^\cfact_2}(\cS)$. We first prove that for every $\cF\in\modl_{\Theta^\cfact_2}(\cS)$ we have $q_!q^*\cF\cong\cF$. By definition we have
\[q_!q^*\cF(S;\overline{l})\cong\underset{(q(S',n')\rightarrow (S;\overline{l}))\in(q/(S;\overline{l}))}{\colim}\cF(S',n').\]
Denote by $(q/(S;\overline{l}))^\ep$ the full subcategory of $(q/(S;\overline{l}))$ on those morphisms $(S;\overline{l})\xrightarrow{f}q(S,n)$ in $(\Theta_2^\cfact)^\op$ for which the underlying morphism in $\cfact^\op$ is active. Observe that every morphism $(S;\overline{l})\xrightarrow{s}q(S',n')$ in $(\Theta_2^\cfact)^\op$ uniquely decomposes as
\[(S;\overline{l})\xrightarrow{f}q(S'',n'')\xrightarrow{q(i,\id)}q(S',n'),\]
where $f\in(q/(S;\overline{l}))^\ep$ and $(i,\id)$ is  morphism in $(2-\cfact)^\op\cong\cfact^\op\times\Delta$ for which $i$ is an inert morphism in $\cfact^\op$. It follows that $(q/(S;\overline{l}))^\ep\hookrightarrow(q/(S;\overline{l}))$ is cofinal. Also observe that it follows from \cref{lem:double_active} that
\[(q/(S;\overline{l}))^\ep\cong(q/(e_1;l_1))^\ep\times(q/(e_2;l_2))^\ep\times...\times(q/(e_{n+m};l_{n+m}))^\ep.\]
Every active morphism $a:S\twoheadrightarrow S''$ in $\cfact^\op$ defines a decomposition of $S''$ into substrings $S''_i$ for $i\in\{1,2,...,l(S)\}$ where each $S''_i$ is given by the image of the $i$th elementary segment of S and $l(S)$ denotes the length of the string $S$. Moreover, it follows from the Segal condition that
\[\cF(q(S'',n''))\cong\cF(q(S_1,n''))\times_{\cF(*)}\cF(q(S_2,n''))\times_{\cF(*)}...\times_{\cF(*)}\cF(q(S_{n+m},n'')).\]
Putting everything together we get the following string of isomorphisms
\begin{align*}
    q_!q^*\cF(S;\overline{l})\cong&\underset{(q(S',n')\rightarrow (S;\overline{l}))\in(q/(S;\overline{l}))}{\colim}\cF(S',n')\\
\cong&\underset{(q(S'',n'')\rightarrow (S;\overline{l}))\in(q/(S;\overline{l}))^\ep}{\colim}\cF(S'',n'')\\
\cong&\underset{\bigtimes_{i=1}^{l(S)}(q/(e_i;l_i))^\ep}{\colim}\cF(q(S_1,n''))\times_{\cF(*)}\cF(q(S_2,n''))\times_{\cF(*)}...\times_{\cF(*)}\cF(q(S_{n+m},n''))\\
\cong& \underset{(q/(e_1;l_1))^\ep}{\colim}\cF(q(S_1,n''))\times_{\cF(*)}...\times_{\cF(*)}\underset{(q/(e_{l(S)};l_{l(S)}))^\ep}{\colim}\cF(q(S_{l(S)},n''))\\
\cong&q_!q^*\cF(e_1;l_1)\times_{q_!q^*\cF(*)}q_!q^*\cF(e_2;l_2)\times_{q_!q^*\cF(*)}...\times_{q_!q^*\cF(*)}q_!q^*\cF(e_{l(S)};l_{l(S)}).
\end{align*}
It follows that it suffices to prove the isomorphism for strings $(e;l)$ of length one. However, observe that in this case $(q/(e;l))^\ep$ has a final object given by the identity morphism and the claim follows. Finally, observe that the calculations above also imply that $p^*p_!\cF\cong \cF$ for $\cF\in2-\fact$ thus concluding the proof.
\end{proof}
\begin{lemma}\label{lem:cofinality}
The functor $q:2-\cfact\rightarrow\Theta^\cfact_2$ constructed in \cref{prop:model_equivalence} is cofinal.
\end{lemma}
\begin{proof}
We need to prove that for every object $(S;\overline{n};\overline{m})$ of $\Theta_2^\cfact$ the category $((S;\overline{n};\overline{m})/q)$ is contractible. An object of this category can be identified with a morphism $q(S',r)\xrightarrow{f}(S;\overline{n};\overline{m})$ in $(\Theta_2^\cfact)^\op$. Denote by $((S;\overline{n};\overline{m})/q)^\fip$ the full subcategory of $((S;\overline{n};\overline{m})/q)$ on those morphisms $q(S',r)\xrightarrow{i}(S;\overline{n};\overline{m})$ for which the underlying morphism $S'\rightarrow S$ in $\cfact^\op$ is inert. Observe that we can uniquely factor any morphism $q(S',r)\xrightarrow{f}(S;\overline{n};\overline{m})$ as
\[q(S,r)\xrightarrow{(q(a),\id)}q(S',r')\xrightarrow{i}(S;\overline{n};\overline{m}),\]
where $a$ is an active morphism in $\cfact^\op$ and $i$ belongs to $((S;\overline{n};\overline{m})/q)^\fip$. It follows that $((S;\overline{n};\overline{m})/q)^\fip\hookrightarrow((S;\overline{n};\overline{m})/q)$ is cofinal, in particular it induces an equivalence on classifying spaces.\par
Observe that there is a natural forgetful functor $U:((S;\overline{n};\overline{m})/q)^\fip\rightarrow \cfact^\inrt_{S/}$ which is moreover a coCartesian fibration. Since $\cfact^\inrt_{S/}$ is obviously contractible, to prove the claim it suffices to demonstrate that the fiber of $U$ over a given $S'\overset{i}{\rightarrowtail}S$ in $\cfact^\op$ is contractible. Assume that $S'$ has $n'$ horizontal segments and $m'$ vertical segments and that $i_h(0)=j_0$, $i_h(n')=j_1$, $i_v(0)=k_0$ and $i_v(m')=k_1$. Observe that an object of $U^{-1}(i)$ can be identified with an object $[l]\in\Delta$ together with morphisms $f_i:[l]\rightarrow[n_i]$ for $i\in\{j_0,j_0+1,...,j_1\}$ and $g_s:[n]\rightarrow [m_s]$ for $s\in\{k_0,k_0+1,...,k_1\}$. A morphism in $U^{-1}(i)$ is given by a morphism $[l]\rightarrow[l']$ making an obvious diagram commute. In other words, $U^{-1}(i)$ is isomorphic to $((\Delta^\op)^{k_1-k_0+j_1-j_0}/\Delta)$, where $\Delta:\Delta^\op\rightarrow(\Delta^\op)^{k_1-k_0+j_1-j_0}$ is the diagonal morphism. Finally, observe that $((\Delta^\op)^{k_1-k_0+j_1-j_0}/\Delta)$ is contractible since $\Delta^\op$ is sifted \cite[Lemma 5.5.8.4.]{lurie2009higher}.
\end{proof}
\begin{construction}
Given  pair of objects $\theta(n;a_1,...,a_n)$ and $\theta(m;b_1,...,b_m)$ we can consider a bicategory $\theta(n;a_1,..,a_n)\otimes \theta(m;b_1,...,b_m)$, where $\otimes$ denotes the Gray product of \cref{not:gray_product}. More explicitly, the objects of $\theta(n;a_1,..,a_n)\otimes \theta(m;b_1,...,b_m)$ are given by pairs $(i,j)$ with $0\leq i\leq n$ and $0\leq j\leq m$, a morphism from $(i,j)$ to $(i',j')$ is given by the data of a morphism $f:i\rightarrow i'$ in $\theta(n;\overline{a})$, a morphism $s:j\rightarrow j'$ in $\theta(m;\overline{b})$ and $i'-i$ marked points $j\leq j_1\leq...\leq j_{i'-i}\leq j'$. Now, given an object $(S;h_1,...,h_n;v_1,...,v_m)$ of $\Theta^\cfact_2$ we can associate to it a bicategory $\sq^\lax(S;\overline{h};\overline{v})$ given by a full subcategory of $\theta(n;h_1,..,h_n)\otimes \theta(m;v_1,...,v_m)$ on objects lying above $S$. This defines a functor $\Theta_2^\cfact\rightarrow 2-\cat$ which we can view as a correspondence $\widehat{\cL}:\Theta_2^\cfact\nrightarrow \Delta^\op\times \Delta^\op$.
\end{construction}
\begin{cor}\label{cor:L-equivalence}
Given an object $\cF\in\cfact$, \cref{prop:model_equivalence} allows us to view it both as a functor $\cF_0:2-\cfact\rightarrow\cS$ and as a functor $\cF_1:\Theta_2^\cfact\rightarrow\cS$. Then there is an isomorphism
\[\widetilde{\cL}\cF_0\cong\widehat{\cL}\cF_1.\]
\end{cor}
\begin{proof}
By definition we have
\[\widehat{\cL}\cF_1(s,t)\cong\underset{\sq^\lax(S;\overline{h};\overline{v})\in\Theta^\cfact_2)}{\colim}\mor_{2-\cat}(\theta(s;\overline{a}), \sq^\lax(S;\overline{h};\overline{v}))\times \cF_1(S;\overline{h};\overline{v}).\]
Recall from \cref{lem:cofinality} that the natural functor $q:2-\cfact\rightarrow \Theta_2^\cfact$ is cofinal, so we can replace the colimit over $\Theta^\cfact_2$ by a colimit over $2-\cfact$. The claim now follows, since by definition $\cF_0$ and $\cF_1$ agree on $2-\cfact$.
\end{proof}
\begin{construction}\label{constr:lax_functors1}
Given objects $\cF\in\fact$ and $\cM\in2-\fact$, we can consider both as coCartesian fibrations over $\cfact$ that satisfy the Segal conditions. We will call a \textit{categorical lax functor} from $\cF$ to $\cM$ a morphism of categories over $\cfact$ that takes coCartesian morphisms over inert morphisms in $\cfact$ to coCartesian morphisms. We will denote by $\mor^\lax_\cat(\cF,\cM)$ the space of categorical lax functors.
\end{construction}
\begin{prop}\label{prop:lax_functor1}
There is a bicategory with a factorization system $L^\lax\cF$ such that
\[\mor_{2-\fact}(L^\lax\cF,\cM)\cong\mor^\lax_\cat(\cF,\cM).\]
Moreover, we have an isomorphism 
\[L^\lax\cF\cong \cF\times_{\cfact}\arr^\act(\cfact),\]
where $\arr^\act(\cfact)$ denotes the coCartesian fibration over $\cfact$ sending $S$ to $\cfact^\act_{/S}$.
\end{prop}
\begin{proof}
A lax functor can be identified with a morphism $\cF\rightarrow \cM$ in $\cart^{\cfact^\inrt}(\cfact)$ in the notation of \cite[Construction 4.10.]{kositsyn2021completeness}, the fact that the required left adjoint exists now follows from \cite[Proposition 4.11.]{kositsyn2021completeness}. Since active and inert morphisms form a factorization system on $\cfact$, the same argument as in \cite[Proposition 4.23.]{kositsyn2021completeness} shows that the formula for $L^\lax\cF$ is correct. To finish the proof it only remains to demonstrate that $\cF\times_{\cfact}\arr^\act(\cfact)$ satisfies the Segal condition, however this follows from \cref{lem:double_active}.
\end{proof}
\begin{notation}
Given an object $\cF\in\fact$ we will denote by $\cL^\lax\cF$ the composition $\cL L^\lax\cF$ of functors of \cref{constr:lax_functors1} and \cref{constr:lax_factorization}. For an arbitrary twofold Segal space $\cB$ we will introduce the following notation 
\[\mor^\lax_\fact(\cF,\cB)\bydef\mor^\lax_{2-\cat}(\cL^\lax\cF,\cB)\]
(where $\mor^\lax_{2-\cat}(\cL^\lax\cF,\cB)$ is the twofold Segal space of functors and lax natural transformations of \cite[Definition 4.6.]{kositsyn2021completeness}) and call $\mor^\lax_\fact(\cF,\cB)$ the \textit{ bicategory of lax functors}. We refer to objects of $\mor^\lax_{2-\cat}(\cL^\lax\cF,\cB)$ as \textit{lax functors} and denote them by $F:\cF\rightsquigarrow \cB$.
\end{notation}
\begin{prop}\label{prop:lax_functors_respect_colimits}
For any $\cF\in\fact$ and $\cB\in2-\cat$ we have
\[\mor^\lax_{\fact}(\cF,\cB)\cong\underset{(\sq(S)\rightarrow \cF)\in\cfact/\cF}{\lim}\mor^\lax_{\fact}(\sq(S),\cB).\]
\end{prop}
\begin{proof}
This follows from \cref{lem:colimit_lax_factorization} and he same argument as in \cite[Proposition 4.34]{kositsyn2021completeness}.
\end{proof}
\begin{construction}\label{constr:orientals}
Fix an integer $n\geq0$, we will construct of a certain discrete $n$-category $\orr(n)$ called the \textit{nth oriental}. This construction was originally presented in \cite{street1987algebra}, however numerous other expositions exist, including \cite{steiner2006orientals},\cite{street1991parity},\cite{kapranov1991combinatorial},\cite{verity2008weak}. The $k$-morphisms of $\orr(n)$ can be identified with certain sets of faces of the $n$-simplex of dimension $\leq k$ (by a face of dimension $l$ we mean an injective morphism $[l]\rightarrow [n]$), so we first introduce some notation: given a $k$-face $u:[k]\rightarrow [n]$ we denote by $\partial_+ u$ (resp. $\partial_- u$) the union of even (resp. odd) $(k-1)$-faces of $u$, in other words
\[\partial_+u\bydef \bigcup_{0\leq i\leq n, i\; \mathrm{even}} u\circ \delta_i,\;\;\;\;\; \partial_-u\bydef \bigcup_{0\leq i\leq n, i \;\mathrm{odd}} u\circ \delta_i.\]
For a set $S\bydef\{u_j\}_{j\in J}$ of $k$-faces we define $\partial_\pm$ to be the union of $\partial_\pm u_j$ over $J$.\par
Now we define the set of $k$-morphisms to be the set of sequences 
\[x\bydef(x_0^+,x_0^-|x_1^-, x_1^+|...|x_k^+,x_k^-)\]
in which $x_i^\pm$ are (possibly empty) sets of $i$-faces of $[n]$ such that
\begin{itemize}
    \item $x_0^\pm$ are singletons;
    \item No object $u\in x_i^\pm$ belongs to $\partial_+ u_1$ and $\partial_+ u_2$ (resp. $\partial_- u_1$ and $\partial_- u_2$) for a pair of distinct objects  $u_{1,2}$ of $x_{i+1}^+$ or $x_{i+1}^-$;\par
    \item $\partial_+x_{i+1}^+\bigcup\partial_- x_{i+1}^+\backslash (\partial_+x_{i+1}^+\bigcap\partial_- x_{i+1}^+)\cong \partial_+x_{i+1}^-\bigcup\partial_- x_{i+1}^-\backslash (\partial_+x_{i+1}^-\bigcap\partial_- x_{i+1}^-)\cong x_i^+\bigcup x_i^-$.
\end{itemize}
For a positive number $p\leq k$ we define the $p$th source $d_p^- x$  (resp. $p$th target $d_p^+x$) to be given by 
\[(x_0^+,x_0^-|x_1^-, x_1^+|...|x_{p-1}^+,x_{p-1}^-|x_p^-,x_p^-)\]
(resp. by \[(x_0^+,x_0^-|x_1^-, x_1^+|...|x_{p-1}^+,x_{p-1}^-|x_p^+,x_p^+)).\]
If $d^+_p x\cong d^-_p y\cong w$, then we define $y*_p x$ to be given by 
\[(w_0^-,w_1^+|...|w_{p-1}^-,w_{p-1}^+|x_p^-\cup y_p^-,x_p^+\cup y_p^+|...|x_k^-\cup y_k^-,x_k^+\cup y_k^+).\]\par
Now, assume that we are given an $n$-truncated simplicial space $X_\bullet:\Delta^\op_{\leq n}\rightarrow\cS$, we will associate to it an $n$-category $n-\cat(X_\bullet)$. To do this, first observe that to any $k$-morphism $x$ of $\orr(n)$ we can associate a simplicial set $s(x)$ given by the union of all faces contained in $x$. Similarly, for any composable sequence of morphisms $a\in \orr(n)_{\overline{m}}$ for $[\overline{m}]\bydef([m_1],[m_2],...,[m_n])\in(\Delta^\op)^n$ (here we view $\orr(n)$ as a functor $(\Delta^\op)^n\rightarrow\cS$) we define $s(a)$ to be the simplicial set associated to their composition. Given a $p$-composable pair of $k$-morphism $x_1$ and $x_2$ observe that
\[s(x_1*_p x_2)\cong s(x_1) \coprod_{s(d+_p(x_1))} s(x_2).\]
Iterating this observation we get that for $a\in \orr(n)_{\overline{m}}$ we have
\[s(a)\cong\underset{(i:[\overline{m}]\rightarrowtail e)\in(\Delta^{\op, n})^\el_{[\overline{m}]/}}{\colim} s(i_![\overline{m}]).\]
Now consider a functor $F:(\Delta^\op)^n\rightarrow\cat$ sending $[\overline{m}]$ to the category whose objects are given by $a\in \orr(n)_{\overline{m}}$ and whose morphisms are given by \textit{surjective} morphisms between $s(a)$. Observe that, if for a $p$-composable pair of $k$-morphisms $x_{1,2}$ we are given surjective morphisms $f:s(x_1)\rightarrow s(y_1)$ and $g:s(x_2)\rightarrow s(y_2)$ that agree on $d^+_p x_1$, we can define a surjective morphism
\[f\sqcup g:s(x_1)\coprod_{s(d+_p(x_1))} s(x_2)\rightarrow s(y_1)\coprod_{s(d+_p(y_1))} s(y_2) .\]
It follows that \[F([1],...,\overset{p}{[2]},...,\overset{k}{[1]},...,[0])\cong [F([1],...,\overset{p}{[1]},...,\overset{k}{[1]}...,[0])\times_{[F([1],...,\overset{p}{[1]},...,[0])}[F([1],...,\overset{p}{[1]},...,\overset{k}{[1]},...,[0]).\]
By iterating this statement we get 
\[F([\overline{m}])\cong\underset{([\overline{m}]\rightarrowtail e)\in(\Delta^{\op, n})^\el_{[\overline{m}]/}}{\lim} F(e).\]
Finally, we define the functor $n-\cat(X_\bullet):(\Delta^\op)^n\rightarrow \cS$ by 
\[n-\cat(X_\bullet)([\overline{m}])\bydef \underset{a\in F([\overline{m}])^\op}{\colim}\mor_{\cP(\Delta_{\leq n})}(s(a), X_\bullet).\]
Combining the observations made above we see that
\begin{align*}
    n-\cat(X_\bullet)([\overline{m}])\bydef& \underset{a\in F([\overline{m}])^\op}{\colim}\mor_{\cP(\Delta_{\leq n})}(s(a), X_\bullet)\\
    \cong& \underset{a\in \underset{([\overline{m}]\rightarrowtail e)\in(\Delta^{\op, n})^\el_{[\overline{m}]/}}{\lim} F(e)^\op}{\colim}\mor_{\cP(\Delta_{\leq n})}(\underset{(i:[\overline{m}]\rightarrowtail e)\in(\Delta^{\op, n})^\el_{[\overline{m}]/}}{\colim} s(i_![\overline{m}]), X_\bullet)\\
    \cong & \underset{a\in \underset{([\overline{m}]\rightarrowtail e)\in(\Delta^{\op, n})^\el_{[\overline{m}]/}}{\lim} F(e)^\op}{\colim}\underset{(i:[\overline{m}]\rightarrowtail e)\in(\Delta^{\op, n})^\el_{[\overline{m}]/}}{\lim}\mor_{\cP(\Delta_{\leq n})}( s(i_![\overline{m}]), X_\bullet)\\
    \cong &\underset{(i:[\overline{m}]\rightarrowtail e)\in(\Delta^{\op, n})^\el_{[\overline{m}]/}}{\lim} \underset{a\in F(e)^\op}{\colim}\mor_{\cP(\Delta_{\leq n})}(s(i_!a), X_\bullet)\\
    \cong& \underset{(i:[\overline{m}]\rightarrowtail e)\in(\Delta^{\op, n})^\el_{[\overline{m}]/}}{\lim} n-\cat(X_\bullet)(e).
\end{align*}
It follows that $n-\cat(X_\bullet)$ is indeed an $n$-category.\par
Observe that for every morphism $X_\bullet\rightarrow Y_\bullet$ of $n$-truncated simplicial spaces there is an induced morphism $n-\cat(X_\bullet)\rightarrow n-\cat(Y\bullet)$. In particular, for a simplicial object $X_\bullet:\Delta^\op\rightarrow\cS$ we have morphisms $n-\cat(X_{\leq s})\rightarrow n-\cat(X_{\leq s+1})$ for all $s\geq 0$, and we define
\[\infty-\cat(X_\bullet)\bydef\colim (0-\cat(X_{\leq0})\rightarrow1-\cat(X_{\leq1})\rightarrow...s-\cat(X_{\leq s})\rightarrow...),\]
where the colimit is taken in the $\infty-\cat$.
\end{construction}
\begin{remark}
The $\infty$-category $\infty-\cat(X_\bullet)$ can essentially be described as a free $\infty$-category on a simplicial space: its objects are the elements of $X_0$, its (non-unital) morphisms are given by composable sequences of nondegenerate elements in $X_1$, for every nondegenerate $x_2\in X_2$ there is a 2-morphism $\delta_2 x_2\circ \delta_0 x_2\rightarrow \delta_1 x_2$ etc. In general, $k$-morphisms in $\infty-\cat(X_\bullet)$ are generated under composition by $k$-simplices of $X_\bullet$, with degenerate $k$-simplices corresponding to identity morphisms.
\end{remark}
\begin{notation}\label{not:localization}
Denote by $p_1:(\Delta^\op)^\infty\rightarrow\Delta^\op$ the projection to the first coordinate. For a simplicial space $X_\bullet$ we denote 
\[\widetilde{L_\cat}(X_\bullet)\bydef p_{1,!} (\infty-\cat (X_\bullet)).\]
In particular, we see that
\[\widetilde{L_\cat}(X_\bullet)_0\cong \underset{(\Delta^\op)^\infty}{\colim}X_0\cong X_0.\]
More generally, observe that we have the following sequence of isomorphisms
\begin{align*}
    \widetilde{L_\cat}(X_\bullet)_n\cong&\underset{[\overline{m}]\in\Delta^{\infty,\op}}{\colim}\infty-\cat(X_\bullet)([n],[\overline{m}])\\
    \cong &\underset{[\overline{m}]\in\Delta^{\infty,\op}}{\colim}\infty-\cat(X_\bullet)([1],[\overline{m}])\times_{X_0}...\times_{X_0}\infty-\cat(X_\bullet)([1],[\overline{m}])\\
    \cong & \underset{([\overline{m}_i])|_{i\in\{1,2,...,n\}}\in(\Delta^{\infty,\op})^n}{\colim}\infty-\cat(X_\bullet)([1],[\overline{m}_1])\times_{X_0}...\times_{X_0}\infty-\cat(X_\bullet)([1],[\overline{m}_n])\\
    \cong & \underset{[\overline{m}_1]\in\Delta^{\infty,\op}}{\colim}\infty-\cat(X_\bullet)([1],[\overline{m}_1])\times_{X_0}...;\times_{X_0}\underset{[\overline{m}_n]\in\Delta^{\infty,\op}}{\colim}\infty-\cat(X_\bullet)([1],[\overline{m}_n])\\
    \cong & \widetilde{L_\cat}(X_\bullet)_1\times_{\widetilde{L_\cat}(X_\bullet)_0}...\times_{\widetilde{L_\cat}(X_\bullet)_0}\widetilde{L_\cat}(X_\bullet)_1,
\end{align*}
so it follows that $\widetilde{L_\cat}(X_\bullet)$ is a Segal space.
\end{notation}
\begin{prop}\label{prop:localization}
Denote by $L_\cat:\mor_\cat(\Delta^\op,\cS)\rightarrow\cat$ the localization functor, then there is an isomorphism
\[L_\cat\cong\widetilde{L_\cat}.\]
\end{prop}
\begin{proof}
The functor $L_\cat$ is an idempotent monad on $\mor_\cat(\Delta^\op,\cS)$. The category of algebras for an idempotent monad $L$ on a category $\cC$ is a full reflective subcategory $i:\cI\hookrightarrow\cC$ on objects of the form $Lc$, and moreover, the monad itself can be uniquely reconstructed from $\cI$ as the left adjoint to the inclusion $i$. In the argument that follows we will prove that $\widetilde{L_\cat}$ is an idempotent monad and that its category of algebras is isomorphic to $\cat$, thus proving the claim in the proposition.\par
First, by the observations made at the end of \cref{not:localization} we have that \[\widetilde{L_\cat}\mor(\Delta^\op,\cS)\subset\cat.\]
We will now prove that $\widetilde{L_\cat}\cC\cong\cC$ for $\cC\in\cat$. This will conclude the proof of the proposition since it would follow that $\widetilde{L_\cat}$ is idempotent and that 
\[\cat\subset \widetilde{L_\cat}\mor(\Delta^\op,\cS).\]
Now, observe that $\infty-\cat(\cC)([1],[\overline{m}])$ is an $\infty$-category whose objects are identity arrows and composable sequences \[(c_0\xrightarrow{f_1}c_1\xrightarrow{f_2}...\xrightarrow{f_m}c_m)\]
of non-unital morphisms in $\cC$ and 
\[\mor_{\infty-\cat(\cC)([1],[\overline{m}])}((c_0\xrightarrow{f_1}...\xrightarrow{f_m}c_m),(c_0\xrightarrow{f_1}...\xrightarrow{f_k}c_k\xrightarrow{f_{k+s-1}\circ...\circ f_k}c_{k+s}\xrightarrow{f_{k+s}}...\xrightarrow{f_m}c_m))\cong \orr(s-2).\]
It follows that $\cC_1$ is a full subcategory of $\infty-\cat(\cC)([1],[\overline{m}])$ and moreover (since all $\orr(m)$ are contractible) every object $(c_0\xrightarrow{f_1}...\xrightarrow{f_m}c_m)$ admits a unique-up-to-homotopy morphism to $(c_1\xrightarrow{f_m\circ...\circ f_1}c_m)\in\cC_1$, and so we indeed have $\widetilde{L_\cat}\cC\cong\cC$.
\end{proof}
\begin{construction}\label{construction:simplicial_square}
For a given $S\in\cfact$ we will denote by $\sq^\Delta(S)$ the following simplicial set: its set of $m$-simplices is the set of morphisms $[m]\times[m]\rightarrow \sq(S)$ in $\fact$. More explicitly, it can be identified with a sequence of morphisms \[x_0\xhorizontal{h_1}x_1\xhorizontal{h_2}...\xhorizontal{h_m}x_m\xvertical{v_1}x_{m+1}\xvertical{v_2}...\xvertical{v_m}x_{2m}\]
in $\sq(S)$ considered as a category with a factorization system. The degeneracy maps $\sigma_j$ then act by sending $(h_1,...,h_m;v_1,...,v_m)$ to $(h_1,...,h_j,\id,..., h_m;v_1,...,v_j,\id,...,v_m)$ and $\delta_i$ sends it to $(h_1,...,h_{i+1}\circ h_i,..., h_m;v_1,...,v_{i+1}\circ v_i,...,v_m)$.
\end{construction}
\begin{prop}\label{prop:square_contractibility}
$\sq^\Delta(S)$ are contractible simplicial sets.
\end{prop}
\begin{proof}
We will first prove the claim in the case $\sq(S)\cong[n]\times [l]$ for some $n\geq0$ and $l\geq 0$. Denote by $N([n]\times[l])$ the simplicial nerve of the (ordinary) category $[n]\times[l]$. Using the fact that $[n]\times[l]$ admits a factorization system, we can identify its set of $m$-simplices with the set of sequences 
\[y_0\xhorizontal{\widetilde{h}_1}y_1\xvertical{\widetilde{v}_1}y_2\xhorizontal{\widetilde{h}_2}...\xhorizontal{\widetilde{h}_m}y_{2m-1}\xvertical{\widetilde{v}_m}y_{2m}\]
in $[n]\times[l]$. The set of such sequences can be equivalently described as the set of sequences of positive integers of the form $(0\leq i_1\leq...\leq i_m\leq n,0\leq j_1\leq ...\leq j_m\leq l)$. However, it is easy to see that exactly the same data also describes an element of $\sq^\Delta([n]\times[l])_m$. Moreover, an elementary inspection reveals that $\delta_i$ and $\sigma_j$ act the same way on $\sq^\Delta([n]\times[l])$ and $N([n]\times[l])$, so those simplicial sets are isomorphic. Finally, observe that $N([n]\times[l])$ is contractible since $[n]\times[l]$ admits a final object $(n,l)$.\par
Now assume that $S$ is a general object of $\cfact$. By definition there exist $n\geq 0$ and $l\geq0$ such that $S$ can be identified with a path from $(0,0)$ to $(n,l)$ in $[n]\times[l]$ and $\sq(S)$ can be identified with the full subcategory of $[n]\times[l]$ on objects lying above $S$. We will prove that the simplicial homology $H_i(\sq^\Delta(S))$ vanishes for $i>0$, the claim would then follow from the Hurewicz theorem. Observe that for $n$ and $l$ as above the simplicial chain complex $C_\bullet(\sq^\Delta(S))$ can be identified with a subcomplex of $C_\bullet(\sq^\Delta([n]\times[l]))$. In particular, since $\sq^\Delta([n]\times[l])$ is contractible we see that for every chain $c\in C_m(\sq^\Delta(S))$ such that $\partial c=0$ there exists $c'\in C_{m+1}(\sq^\Delta([n]\times[l]))$ such that $\partial c'=c$. We will temporarily denote for $k\geq0$ by $C_k^+$ the subcomplex of $C_{k}(\sq^\Delta([n]\times[l]))$ on those $f:[k]\times[k]\rightarrow[n]\times[l]$ whose image belongs to $\sq(S)$ (so $C_k^+$ is isomorphic to $ C_k(\sq^\Delta(S))$) and by $C_k^-$ the subcomplex on those $f$ whose image does not belong to $\sq(S)$. We then have a decomposition $c'=c'_+ +c'_-$ with $c'_\pm\in C_{m+1}^\pm$. Our goal now is to prove that , by possibly replacing $c'$ with $c''\bydef c'+\partial \widetilde{c}$ (so that $\partial c''=\partial c'$), we can ensure $\partial c''_-=0$, since if this is true, then we can further replace $c''$ with $(c''-c''_-)\in C_{m+1}(\sq^\Delta(S))$ which will prove that $H_m(\sq^\Delta(S))=0$.\par
We can further decompose $\partial c'_\pm=(\partial c'_\pm)_+ +(\partial c'_\pm)_-$. Observe that $(\partial c'_+)_-=0$. Since $c\in c_m^+$, we see that 
\[(\partial c')_-=(\partial c'_+)_- +(\partial c'_-)_-=(\partial c'_-)_-=0.\]
It follows that it suffices to make $(\partial c''_-)_+=0$. Recall that by definition $\partial c'_-=\sum_{i=0}^{m+1} \delta_i c'_-$. Observe that, if $c'\in C_{m+1}^-$, then $\delta_i c'\in C_m^-$ for $0<i<m+1$. Indeed, an element $f:[m]\times[m]\rightarrow[n]\times[l]$ belongs to $C_m^-$ if and only if $f(m,0)$ lies below $S$ in $[n]\times[l]$ and it is easy to see that $f(m,0)=\delta_i f(m,0)$ for $0<i<m+1$. We denote by $C^l_{m+1}$ (resp. $C^r_{m+1}$) the subcomplex of $C^-_{m+1}$ on those $a$ for which $\delta_0 a $ (resp. $\delta_{m+1}a$) belongs to $C^+_m$ and by $C^l_{m+1}(c')$ (resp. $C^r_{m+1}(c')$) the set of those basis elements $a\in C^r_{m+1}$ (resp. $a\in C^l_{m+1}$) that appear in $c'$ with nonzero coefficient, then we have
\[(\partial c'_-)_+ =\sum_{a\in C^l_{m+1}(c')} c_a \delta_0 a + (-1)^{m+1}\sum_{a'\in C^r_{m+1}(c')} c_{a'} \delta_{m+1} a',\]
where $c_a\in\bZ$ are coefficients. We will finish our proof by showing that by replacing $c'$ with $c''= c'+\partial \widetilde{c}$ we can reduce the number of elements of $C^l_{m+1}(c'')\bigcup C^r_{m+1}(c'')$ to zero. Indeed, assume that $a\in C^l_{m+1}(c')$, then denote $\widetilde{a}\bydef \sigma_0 a\in C_{m+2}([n]\times[l])$. Observe that $\delta_0 \widetilde{a}=a$ and $\delta_i \widetilde{a}$ do not lie in $C^l_{m+1}$ for $i>0$. It follows that for $\widehat{c}\bydef c'-\partial(c_a a)$ the set $C^l_{m+1}(\widehat{c})$ has one less element than $C^l_{m+1}(c')$. Also observe that if $a\notin C^r_{m+1}$ then none of the elements $\delta_i \widetilde{a}$ are either, so we can use the procedure above to remove elements from $C^l_{m+1}(c')$ without enlarging $C^r_{m+1}(c')$. Finally, observe that, after we are done with $C^l_{m+1}(c')$, we can reduce the size of $C^r_{m+1}(c')$ to zero using a very similar procedure except for setting $\widetilde{a}\bydef \sigma_{m+1}a$.
\end{proof}
\begin{construction}\label{constr:lax_square}
For $m\geq1$ we will define certain $m$-categories $\openbox^\lax_m$. We first denote by $\cW_m$ the walking $m$-morphism. More explicitly, we define $\cW_1$ to be the category $[1]$ and inductively define $\cW_m$ to be the category with two objects $0$ and $1$ with no non-trivial endomorphisms such that
\[\mor_{\cW_m}(0,1)\cong \cW_{m-1}.\]
We will also denote by $\cW_0$ the singleton category $\{0\}$.\par
Now we begin describing $\openbox^\lax_m$. First, we set $\openbox^\lax_0\bydef [1]$. We then define $\openbox^\lax_1$ to be the category with four objects $\{(i,j)\}_{i,j\in\{0,1\}}$ that have no non-trivial endomorphisms and such that \[\mor_{\openbox^\lax_1}((0,0),(0,1))\cong\mor_{\openbox^\lax_1}((1,0),(1,1))\cong \mor_{\openbox^\lax_1}((0,0),(1,0))\cong \mor_{\openbox^\lax_1}((0,1),(1,1))\cong \pt\]
and $\mor_{\square^\lax_1}((0,0),(1,1))\cong[1]$. The composition operation sends the unique composable pair of morphisms from $(0,0)$ to $(1,1)$ passing through $(0,1)$ to $\{0\}\in[1]$ and the one passing through $(1,0)$ to $\{1\}$. We also define the morphisms $k^1_{i}:\cW_1\cong[1]\rightarrow \openbox^\lax_1$ for $i\in\{0,1\}$ to be the inclusions of the full subcategories generated by $(i,0)$ and $(i,1)$. Now assume that we have already defined $\openbox^\lax_{s}$ for $s<m$ as well as inclusions $k^s_i:\cW_i\hookrightarrow\openbox^\lax_s$, we define $\openbox^\lax_m$ to be the $m$-category with four objects such that
\[\mor_{\openbox^\lax_m}((0,0),(1,0))\cong \mor_{\openbox^\lax_m}((0,1),(1,1))\cong \pt,\]
\[\mor_{\openbox^\lax_m}((0,0),(0,1))\cong \mor_{\openbox^\lax_m}((1,0),(1,1))\cong \cW_{m-1}\]
and $\mor_{\openbox^\lax_m}((0,0),(1,1))\cong \openbox^\lax_{m-1}$, with the composition operators sending the upper (resp. lower) copy of $\cW_{m-1}$ into $\openbox^\lax_{m-1}$ by means of $k^{m-1}_0$ (resp. $k^{m-1}_1$). We define $k^m_{0,1}:\cW_m\hookrightarrow \openbox^\lax_m$ to be the inclusions of the full subcategories on $(0,0)$ and $(0,1)$ and on $(1,0)$ and $(1,1)$. \par
We will also fix notation for certain natural inclusions among $\cW_i$ and $\openbox^\lax_j$. Denote by $j^1_k$ for $k\in\{0,1\}$ the natural inclusions $\cW_0\cong\{k\}\hookrightarrow \cW_1\cong [1]$; we claim that it can be extended to a system of morphisms $j^m_k:\cW_m\rightarrow \cW_{m+1}$. Indeed, assuming we have already defined $j^{i<m}_k$, we define $j^m_k$ to be identity on objects and to be given by $j^{m-1}_k$ on the category of morphisms from $0$ to $1$. Observe that we also have natural inclusions $i^0_k:\openbox^\lax_0\hookrightarrow \openbox^\lax_1$ for $k\in\{0,1\}$ sending $[1]$ isomorphically onto a full subcategory of $\openbox^\lax_1$ containing $(k,0)$ and $(k,1)$. Now, define $i^m_k$ to be given by the identity on objects and morphisms from $(0,0)$ to $(1,0)$ and from $(0,1)$ to $(1,1)$, by $j^{m-1}_k$ on morphisms from $(0,0)$ to $(0,1)$ and from $(1,0)$ to $(1,1)$ and by $i^{m-1}_k$ on morphisms from $(0,0)$ to $(1,1)$. Finally, we denote by $p^m_k$ the natural inclusions $[1]\hookrightarrow\openbox^\lax_m$ mapping $[1]$ isomorphically onto the full subcategory generated by $(0,k)$ and $(1,k)$.
\end{construction}
\begin{construction}\label{constr:elementary_lax}
For a given object $S\in\cfact$ we will temporarily denote by $\widetilde{\cL^\lax\sq(S)}$ the following bicategory: its objects are the objects of $\sq(S)$, a morphism from $(i,j)$ to $(i',j')$ is given by a morphism $f:\sq(S')\rightarrow\sq(S)$ in $\cfact^\op$ such that $f(0,0)=(i,j)$ and $f(n',m')=(i',j')$. A 2-morphism is given by a commutative diagram
\[
\begin{tikzcd}[row sep=huge, column sep=huge]
\sq(S')\arrow[dr, "f"]\arrow[rr, "s"]&{}&\sq(S'')\arrow[dl, "g" swap]\\
{}&\sq(S)
\end{tikzcd}
\]
in $\cfact^\op$ in which $s$ is a morphism in $\cfact^\ap$ in the notation of \cref{cor:active_permutation}.
\end{construction}
\begin{prop}\label{prop:lax_for_objects_of_fact}
For every $S\in\cfact$ there is an equivalence
\[\widetilde{\cL^\lax\sq(S)}\cong \cL^\lax\sq(S)\]
\end{prop}
\begin{proof}
We begin by calculating $\widehat{\cL}L^\lax\sq(S)$. By definition we have 
\[\widehat{\cL}L^\lax\sq(S)(\theta(s;\overline{t}))\cong\underset{(\theta(s;\overline{t})\rightarrow\sq^\lax(S';\overline{h};\overline{v}))\in(\theta(s;\overline{t})/\widehat{\cL})}{\colim}L^\lax \sq(S) (S;\overline{h};\overline{v}).\]
A morphism $\theta(s;\overline{t})\rightarrow\sq^\lax(S';\overline{h};\overline{v})$ is given by the following data: $(s+1)$ marked points $(i_k,j_k)|_{k\in\{0,1,...,s\}}$ in $\sq(S')$, $t_k$ paths $S_{k,l}|_{k\in\{0,1,...s-1\},l\in\{0,1,...,t_k-1\}}$ between $(i_k,j_k)$ and $(i_{k+1},j_{k+1})$ in $\sq(S')$ for every $k$, $(i_k-i_{k-1})$ morphisms $f_p:[t_k]\rightarrow [h_{i_k-i_{k-1}+p}]$ for $p\in\{0,1,...i_k-i_{k-1}-1\}$ and $(j_k-j_{k-1}) $ morphisms $g_q:[t_k]\rightarrow [v_{j_k-j_{k-1}+q}]$ for $q\in\{0,1,...j_{k}-j_{k-1}-1\}$. We will call a morphism $a:\theta(s;\overline{t})\rightarrow \sq^\lax(S;\overline{h};\overline{v})$ active if $i_0=j_0=0$, $i_s=n'$, $j_s=m'$, the paths $S_{k,0}$ lie on $S'\in\sq(S')$ and all $f_p$ and $g_q$ are active; we will denote by $(\theta(s;\overline{t})/\widehat{\cL})^\act$ the full subcategory of $(\theta(s;\overline{t})/\widehat{\cL})$ on active morphisms. Observe that every morphism between objects in $(\theta(s;\overline{t})/\widehat{\cL})^\act$ is of the from $\widehat{\cL}(b)$ for an active morphism $b$ of $\Theta^\cfact_2$. It follows from the above description that every morphism $f:\theta(s;\overline{t})\rightarrow\sq^\lax(S';\overline{h};\overline{v})$ uniquely factors as
\[\theta(s;\overline{t})\overset{a}{\twoheadrightarrow}\sq^\lax(S'';\overline{h}';\overline{v}')\overset{\widehat{\cL}(i)}{\rightarrowtail}\lax(S';\overline{h};\overline{v}),\]
where $a$ is active and $i$ is an inert morphism in $(\Theta_2^\cfact)^\op$. It follows that $(\theta(s;\overline{t})/\widehat{\cL})^\act$ is a cofinal subcategory of $(\theta(s;\overline{t})/\widehat{\cL})$.\par
Observe that we have the following isomorphism
\[(\theta(s;\overline{t})/\widehat{\cL})^\act\cong\overbrace{(\theta(1;t_1)/\widehat{\cL})^\act\times(\theta(1;t_2)/\widehat{\cL})^\act\times...\times(\theta(1;t_s)/\widehat{\cL})^\act}^\text{$s$ times}\]
obtained by sending morphisms $a_i:\theta(1;t_i)\rightarrow\sq^\lax(S'_i;\overline{h}_i;\overline{v}_i)$ to a morphism $a:\theta(s;\overline{t})\rightarrow\sq^\lax(S';\overline{h};\overline{v})$, where $\sq^\lax(S';\overline{h};\overline{v})$ is obtained by joining $\sq^\lax(S'_i;\overline{h}_i;\overline{v}_i)$ along the endpoints of $S_i$. Similarly, observe that every active morphism $a:\theta(s;\overline{t})\rightarrow\sq^\lax(S';\overline{h};\overline{v})$ induces a partition of $(S';\overline{h};\overline{v})$ into segments $(S'_k;\overline{h}_k;\overline{v}_k)$ from $(i_{k-1},j_{k-1})$ to $(i_k,j_k)$ for $1\leq i\leq s$. Moreover, since $L^\lax\sq(S)$ satisfies the Segal condition we have
\[L^\lax\sq(S)(S';\overline{h};\overline{v})\cong L^\lax\sq(S)(S_1';\overline{h}_1;\overline{v}_1)\times_{L^\lax\sq(S)(*)}...\times_{L^\lax\sq(S)(*)}L^\lax\sq(S)(S_s';\overline{h}_s;\overline{v}_s).\]
Altogether we get the following string of isomorphisms:
\begin{align*}
    \widehat{\cL}L^\lax\sq(S)(\theta(s;\overline{t}))\cong&\underset{(\theta(s;\overline{t})\rightarrow\sq^\lax(S';\overline{h};\overline{v}))\in(\theta(s;\overline{t})/\widehat{\cL})}{\colim}L^\lax \sq(S) (S;\overline{h};\overline{v})\\
    \cong &\underset{(\theta(s;\overline{t})\rightarrow\sq^\lax(S';\overline{h};\overline{v}))\in(\theta(s;\overline{t})/\widehat{\cL})^\act}{\colim}L^\lax \sq(S) (S;\overline{h};\overline{v})\\
    \cong& \underset{\bigtimes_{k=1}^s(\theta(1;t_k)/\widehat{\cL})^\act}{\colim}L^\lax\sq(S)(S_1';\overline{h}_1;\overline{v}_1)\times_{L^\lax\sq(S)(*)}...\times_{L^\lax\sq(S)(*)}L^\lax\sq(S)(S_s';\overline{h}_s;\overline{v}_s)\\
    \cong&\underset{(\theta(1;t_1)/\widehat{\cL})^\act}{\colim}L^\lax\sq(S)(S_1';\overline{h}_1;\overline{v}_1)\times_{L^\lax\sq(S)(*)}...\times_{L^\lax\sq(S)(*)}\underset{(\theta(1;t_s)/\widehat{\cL})^\act}{\colim}L^\lax\sq(S)(S_s';\overline{h}_s;\overline{v}_s)\\
    \cong& \widehat{\cL}L^\lax\sq(S)(\theta(1;t_1))\times_{\widehat{\cL}L^\lax\sq(S)(*)}...\times_{\widehat{\cL}L^\lax\sq(S)(*)}\widehat{\cL}L^\lax\sq(S)(\theta(1;t_s)).
\end{align*}
It follows that it suffice to prove the isomorphism $\widehat{\cL}L^\lax\sq(S)\cong \widetilde{\cL^\lax\sq(S)}$ on objects of the form $\theta(1;t)$.\par
An object of $(\theta(1;t)/\widehat{\cL})^\act$ can be identified with an object $(S';\overline{h};\overline{v})$ of $\Theta_2^\cfact$, $t$ composable permutation morphisms 
\[S'\overset{i_1}{\rightarrowtail}S'_1\overset{i_2}{\rightarrowtail}...\overset{i_{t-1}}{\rightarrowtail}S'_t\]
in $\cfact$ as well as active morphisms $a_k:[t]\twoheadrightarrow [h_k]$ for $1\leq k\leq n'$ and $b_l:[t]\twoheadrightarrow[v_l]$ for $1\leq l\leq m'$. The data of active morphisms $a_k$ and $b_l$ can equivalently be described as an active morphism $a_{S'}:(S';\overline{h};\overline{v})\twoheadrightarrow q(S',t)$ for which the underlying morphism $S'\twoheadrightarrow S'$ is an identity, where $q$ is the functor described in \cref{prop:model_equivalence}. Denote by $(\theta(1;t)/\widehat{\cL})^q$ the full subcategory of $(\theta(1;t)/\widehat{\cL})^\act$ on objects for which $a_{S'}=\id$, then it follows from the description above that $(\theta(1;t)/\widehat{\cL})^q$ is cofinal. Explicitly, an object of $(\theta(1;t)/\widehat{\cL})^q$ can be identified with a string $(S'_1\overset{i_1}{\rightarrowtail}...\overset{i_{t-2}}{\rightarrowtail}S'_{t-1}\overset{i_{t-1}}{\rightarrowtail}S')$ of permutation morphisms, and a morphism in $(\theta(1;t)/\widehat{\cL})^q$ is given by a commutative diagram 
\[
\begin{tikzcd}[row sep=huge, column sep=huge]
S''_1\arrow[r, tail, "i'_1"]&S''_2\arrow[r, tail, "..."description]&S''_{t-1}\arrow[r, tail, "i'_{t-1}"]&S''\\
S_1' \arrow[u, two heads, "a_1"]\arrow[r, tail, "i_1"]&S'_2\arrow[r, tail, "..." description]\arrow[u, two heads, "a_2"]&S'_{t-1}\arrow[u, two heads, "a_{t-1}"]\arrow[r, tail, "i_{t-1}"]&S'\arrow[u, two heads, "a"]
\end{tikzcd}
\]
in $\cfact^\op$. Finally, the value of $L^\lax\sq(S)$ on $q(S',t)$ is given by the set of strings of morphisms
\[S'\overset{a_1}{\twoheadrightarrow}\widetilde{S}'_{1}\overset{a_{2}}{\twoheadrightarrow}\widetilde{S}'_2\overset{a_{3}}{\twoheadrightarrow}...\overset{a_{t+1}}{\twoheadrightarrow}\widetilde{S}'_{t+1}\xrightarrow{f}S\]
in $\cfact^\op$. It follows that the value 
\[\widehat{\cL}L^\lax\sq(S)(\theta(1;t))\cong \underset{(\theta(1;t)/\widehat{\cL})^q}{\colim}L^\lax\sq(S) (q(S',t))\]
is given by the classifying space of the category, which we denote $\widehat{\cL}_S$, whose objects are strings 
\[S'_1\overset{i_1}{\rightarrowtail}S'_{2}\overset{i_{2}}{\rightarrowtail}...\overset{i_{t-1}}{\rightarrowtail}S'\overset{a_1}{\twoheadrightarrow}\widetilde{S}'_{1}\overset{a_{2}}{\twoheadrightarrow}...\overset{a_{t+1}}{\twoheadrightarrow}\widetilde{S}'_{t+1}\xrightarrow{f}S\]
 of morphisms in $\cfact^\op$ and morphisms are given by active morphisms $b:S'\twoheadrightarrow S''$ in $\cfact^\op$ making the following diagrams commute:
\[
\begin{tikzcd}[row sep=huge, column sep=huge]
S'_1\arrow[r, tail, "i_1"]\arrow[d, two heads, "b_1"]&S'_{2}\arrow[d, two heads, "b_{2}"]\arrow[r, tail, "..." description]&S'\arrow[d, two heads, "b"]\arrow[r, two heads, "a_1"]&\widetilde{S}'_1\arrow[r, two heads, "a_2"]\arrow[d, equal]&\widetilde{S}'_2\arrow[d, equal]\arrow[r, two heads, "..." description]&\widetilde{S}'_{t+1}\arrow[d, equal]\arrow[dr, "f"]\\
S''_1\arrow[r, tail, "i'_1"]&S_{2}'' \arrow[r, tail, "..." description]&S''\arrow[r, two heads, "a'_1"]&\widetilde{S}''_1\arrow[r, two heads, "a'_2"]&\widetilde{S}''_2\arrow[r, two heads, "..." description]&\widetilde{S}''_{t+1}\arrow[r, "g"]&s
\end{tikzcd}.
\]
Denote by $\widehat{\cL}^\id_S$ the full subcategory on those objects for which $a_1\cong \id$. Observe that $\widehat{\cL}^{\id}_{S}$ is both a cofinal subcategory of $\widehat{\cL}_S$ and a set, so it follows that the required classifying space is isomorphic to $\widehat{\cL}^\id_S$. In other words, we see that $\widehat{\cL}L^\lax\sq(S)(\theta(1;t))$ is given by the set of strings
\begin{align}\label{eq:ten}
    S'_1\overset{i_1}{\rightarrowtail}S'_{2}\overset{i_{2}}{\rightarrowtail}...\overset{i_{t-1}}{\rightarrowtail}S'\overset{a_1}{\twoheadrightarrow}\widetilde{S}'_{2}\overset{a_{2}}{\twoheadrightarrow}...\overset{a_{t}}{\twoheadrightarrow}\widetilde{S}'_{t}\xrightarrow{f}S
\end{align}
in $\cfact^\op$. \par
Observe that $\widehat{\cL}L^\lax\sq(S)(\theta(1;t))$ is \textit{not} a Segal space, so we will have to calculate its localization \[L_\cat(\widehat{\cL}L^\lax\sq(S)(\theta(1;t))).\] 
We will do this using \cref{prop:localization}, however before that we will describe the structure of $\widehat{\cL}L^\lax\sq(S)(\theta(1;t))$ as a simplicial space. Every morphism in $\Delta$ can be written as a composition of $\sigma_j$ and $\delta_i$, so it suffices to describe the action of these morphisms. For $0\leq j\leq t$ the morphism $\sigma_j$ sends 
the string (\ref{eq:ten}) to 
\[S'_1\overset{i_1}{\rightarrowtail}...\overset{i_j}{\rightarrowtail}S'_{j+1}=\joinrel=S'_{j+1}\overset{i_{j+1}}{\rightarrowtail}...\overset{i_{t-1}}{\rightarrowtail}S'\overset{a_1}{\twoheadrightarrow}...\overset{a_j}{\twoheadrightarrow}\widetilde{S}'_{j+1}=\joinrel=\widetilde{S}'_{j+1}\overset{a_{j+1}}{\twoheadrightarrow}...\overset{a_{t}}{\twoheadrightarrow}\widetilde{S}'_{t}\xrightarrow{f}S.\]
For $0<k<t$ the morphism $\delta_i$ sends (\ref{eq:ten}) to 
\[S'_1\overset{i_1}{\rightarrowtail}...\overset{i_{k-1}}{\rightarrowtail}S'_{k}\overset{i_{k+1}\circ i_k}{\rightarrowtail}S'_{k+2}\overset{i_{k+2}}{\rightarrowtail}...\overset{i_{t-1}}{\rightarrowtail}S'\overset{a_1}{\twoheadrightarrow}...\overset{a_{k-1}}{\twoheadrightarrow}\widetilde{S}'_{k}\overset{a_{k+1}\circ a_k}{\twoheadrightarrow}\widetilde{S}'_{k+2}\overset{a_{k+2}}{\twoheadrightarrow}...\overset{a_{t}}{\twoheadrightarrow}\widetilde{S}'_{t}\xrightarrow{f}S.\]
It now remains to describe the action of $\delta_t$ and $\delta_0$; The morphism $\delta_0$ sends (\ref{eq:ten}) to the string
\[S''_2\xinert{i_2'}S_3''\xinert{i_3'}...\xinert{i_{t-1}'}\widetilde{S}'_2\xactive{a_2}\widetilde{S}_3'\xactive{a_3}...\xactive{a_t}\widetilde{S}'_t\xrightarrow{f}S\]
appearing in the diagram
\[
\begin{tikzcd}[row sep=huge, column sep=huge]
{}&{}&{}&\widetilde{S}'_t\arrow[r, "f"]&S\\
{}&{}&{}&\widetilde{S}_3'\arrow[u, two heads, "..."description]\\
S_1''\arrow[r, tail, "i_1'"]&S_2''\arrow[r, tail, "i_2'"]&S_3''\arrow[r, tail, "..." description]&\widetilde{S}'_2\arrow[u, two heads, "a_2"]\\
S_1'\arrow[r,tail, "i_1"]\arrow[u, two heads]&S'_2\arrow[u, two heads]\arrow[r, tail, "i_2"]&S'_3\arrow[u, two heads]\arrow[r, tail, "..."description]&S'\arrow[u, two heads, "a_1"]
\end{tikzcd}
\]
in which all the squares at the bottom are factorization squares in $\cfact^\op$. Similarly, $\delta_t$ sends (\ref{eq:ten}) to the string
\[S_1'\xinert{i_1}S_2'\xinert{i_2}...\xinert{i_{t-2}}S'_{t-1}\xactive{a_1'}\widetilde{S}'_2\xactive{a'_2}...\xactive{a'_{t-2}}\widetilde{S}''_{t-1}\xrightarrow{f\circ a_{t-1}\circ j_{t-1}}S\]
appearing in the diagram
\[
\begin{tikzcd}[row sep=huge, column sep=huge]
{}&{}&\widetilde{S}''_t\arrow[r, tail, "j_t"]&\widetilde{S}'_t\arrow[r, "f"]&S\\
{}&{}&\widetilde{S}''_{t-1}\arrow[r, tail, "j_{t-1}"]\arrow[u, two heads, "a'_{t-1}"]&\widetilde{S}'_{t-1}\arrow[u, two heads, "a_{t-1}"]\\
{}&{}&\widetilde{S}''_2\arrow[u, two heads, "..." description]\arrow[r, tail, "j_2"]&\widetilde{S}'_2\arrow[u, two heads,"..." description]\\
S'_1\arrow[r, tail, "i_1"]&S_2'\arrow[r, tail, "..."description]&S'_{t-1}\arrow[u, two heads, "a'_1"]\arrow[r, tail, "i_{t-1}"]&S'\arrow[u, two heads, "a_1"]
\end{tikzcd}.
\]
With this description we can now prove that 
\[L_\cat(\widehat{\cL}L^\lax\sq(S)(\theta(1;t)))([1])\cong \widetilde{\cL^\lax\sq(S)}(\theta_{1,1}).\] 
The space $L_\cat(\widehat{\cL}L^\lax\sq(S)(\theta(1;t)))([1])$ is the classifying space of the $\infty$-category
\[\cL^\lax(S)\bydef(\infty-\cat(\widehat{\cL}L^\lax\sq(S)(\theta(1;t))))([1]).\]
The objects of $\cL^\lax(S)$ are strings of morphisms in $\cfact^\op$ of the form 
 \begin{equation}\label{eq:eleven}
 \begin{tikzcd}[row sep=huge, column sep=huge]
{}&S_1'\arrow[dr, two heads, "a_1"]&{}&S_2'\arrow[dr, two heads, "..." description]&{}&S_n'\arrow[dr, two heads, "a_n"]&S\\
S_1\arrow[ur, tail, "i_1"]&{}&S_2\arrow[ur, tail, "i_2"]&{}&S_n\arrow[ur, tail, "i_n"]&{}&S_{n+1}\arrow[u, "f" swap]
\end{tikzcd}  
\end{equation}
in which all $i_j$ are permutations. We will denote a string like this by $((i_1;a_1),(i_2,a_2),...,(i_n;a_n),f)$. Elementary $p$-morphisms in $\cL^\lax(S)$ are given by object in $\widehat{\cL}L^\lax\sq(S)(\theta(1;p+1))$, in particular, an elementary 1-morphism is given by a diagram
\[
\begin{tikzcd}
{}&{}&S''\arrow[dr, two heads, "a_3"]\\
{}&S'_1\arrow[dr, two heads, "a_1"]\arrow[ur, tail, "i_3"]&{}&S_2'\arrow[dr, two heads, "a_2"]\\
S_1\arrow[ur, tail, "i_1"]&{}&S_2\arrow[ur, tail, "i_2"]&{}&S_3\arrow[r, "f"]&S
\end{tikzcd},
\]
the target of this morphism is then the string $((i_3\circ i_1; a_2\circ a_3),f)$  and the source is the string $((i_1;a_1),(i_2;a_2),f)$. We will denote this morphism by $((i_2,i_1;a_2,a_3),f)$ and we will more generally denote by \[((i_{n+1},i_{n},...,i_1;a_{n+1},a_{n},...a_1),f)\]
an elementary $n$-morphism represented by the string \[(S_0\xinert{i_1}S_1\xinert{i_2}...\xinert{i_{n+1}}S_{n+1}\xactive{a_1}S_{n+2}\xactive{a_2}...\xactive{a_{n+1}}S_{2n+2}\xrightarrow{f}S).\]
We will sometimes omit the morphism $f$ from the notation if it is obvious from the context.\par 
Before moving on, we need to introduce some notation. By \cref{cor:active_permutation} the morphisms in $\cfact$ given by a composition of a permutation and an active morphism form a subcategory $\cfact^\ap$. This category admits an obvious factorization system (given by permutations and active morphisms), and so we can view it as a functor $\cfact^\ap:\cfact\rightarrow\cS$. An elementary $n$-morphism in $\cL^\lax(S)$ can then be identified with a pair $(s,g)$ consisting of an element $s\in\cfact^\ap([n+1]\times[n+1])$ and a morphism $g:s(n+1,n+1)\rightarrow S$. Now, assume that we have a morphism $h:\sq(S)\rightarrow [n+1]\times[n+1]$ in $\fact$, where we assume that $S\subset [l]\times[k]$ is a string with $k$ horizontal and $l$ vertical morphisms, then we will denote by $h^*(s,f)$ the pair consisting of an element $h^*s\in \cfact^\ap(S)$ together with a morphism 
\[h^*g:h^*s(k,l)\cong s(h(k),h(l))\xrightarrow{p} s(n+1,n+1)\xrightarrow{g}S,\]
in which $p$ is a morphism given by composing all morphism in $s$ along some path between $h(k,l)$ and $(n+1,n+1)$ in $[n+1]\times[n+1]$.\par
We will now define a morphism 
\[F:\cL^\lax(S)\rightarrow \widetilde{\cL^\lax(S)}\bydef \infty-\cat(\widetilde{\cL^\lax\sq(S)}(\theta(1;a)))([1]).\]
The objects of $\widetilde{\cL^\lax(S)}$ are formal compositions of morphisms in $\widetilde{\cL^\lax(\sq(S))}$, and those are of the form $S_0\xactive{a} S_1\xinert{i}S_2\xrightarrow{f}S$. Given an object $x\bydef((i_1;a_1),...,(i_l,a_l),f)$ of $\cL^\lax(S)$, we define $F(x)$ to be the composition of $f$ and $\widetilde{i}_k\circ\widetilde{a}_k$, where $\widetilde{i}_k\circ \widetilde{a_k}$ is the factorization of $a_k\circ i_k$ in $\cfact^\ap$. The elementary $n$-morphisms of $\widetilde{\cL^\lax(S)}$ are given by strings
\[S_0\xactive{a_1}S_1\xinert{i_1}S_2\xactive{a_2}S_3...\xinert{i_{n+1}}S_{2n+2}\xrightarrow{f}S\]
of morphisms in $\cfact^\op$ in which all $i_j$ are permutation. Thus an elementary $n$-morphism in $\widetilde{\cL^\lax(S)}$ can be identified with a pair $(s,f)$, where $s$ is an element of $\cfact^\ap(S'_{n+1})$, where $S'_{n+1}$ is a string of the form \[\overbrace{vhvh...vh}^\text{$(n+1)$ times}\]
and $f:S_{2n}\rightarrow S$ is a morphism in $\cfact$. To define the required morphism $F$ it suffices to give its values on elementary morphisms. Observe that there is a natural morphism 
\[\gamma:\sq(S'_{n+1})\rightarrow [n+1]\times[n+1]\]
in $\fact$, and we denote the required morphism by sending an elementary $n$-morphism in $\cL^\lax(S)$ of the form $(s,g)$ with $s\in\cfact^\ap([n]\times[n])$ to the elementary morphism in $\widetilde{\cL^\lax(S)}$ given by $\gamma^*(s,g)$.\par
We will now construct a morphism $G:N(\widetilde{\cL^\lax(S)})\rightarrow N(\cL^\lax(S))$, where $N(-)$ denotes the classifying space of an $\infty$-category. Once again, it suffices to define it on objects and elementary morphisms. We first define it on objects: an object of $\widetilde{\cL^\lax(S)}$ is a formal composition of the form $f\circ (i_l\circ a_l)\circ(i_{l-1}\circ a_{l-1})\circ...\circ(i_1\circ a_1)$; the functor $G$ then sends it to the object $((\id;a_1),(i_1;\id),...,(\id;i_l),(a_l;\id),f)$ of $\cL^\lax(S)$. Before we describe the action of $G$ on morphisms, we make the following observation: given an elementary $n$-morphism in $\widetilde{\cL^\lax(S)}$ corresponding to a pair $(s,f)$ as above, to every morphism $h:[m+1]\times[m+1]\rightarrow S'_{n+1}$ we can associate $h^*(s,f)$, which we can identify with an elementary $m$-morphism in $\cL^\lax(S)$. It follows that to every elementary $n$-morphism (which we can consider as a morphism $[n]\xrightarrow{(s,f)}N(\widetilde{\cL^\lax(S)})$) we can associate a morphism $\beta_{(s,f)}:\sq^\Delta(S'_n)\rightarrow N(\cL^\lax(S))$. Now assume that we are given an elementary $n$-morphism in $\widetilde{\cL^\lax(S)}$ corresponding to the string $((i_1;a_1),...,(i_{n+1},a_{n+1}),f)$. Denote by $S''_{n+1}$ the following string
\[\overbrace{hh....h}^\text{$(n+1)$ times}\overbrace{vv....v}^\text{$(n+1)$ times}\]
considered as an object of $\cfact$. Observe that there is a natural morphism \[b:\sq(S''_{n+1})\rightarrow\sq(S'_{n+1})\]
in $\fact$ given by the inclusion of the top left corner of $\sq(S'_{n+1})$. Assume that the value of $b^* s$ is given by the string $i'_{n+1}\circ i'_n\circ... i'_m\circ a'_{n+1}\circ...\circ a'_1$. We now define $G(s,f)$ to be given by the following string of morphisms in $N(\cL^\lax(S))$:
\begin{align*}
&((\id;a_1),(i_1;\id),(\id;a_2)...,(\id;a_{n+1}),(i_{n+1};\id),f)\xrightarrow{(i_1,\id;\id,a_1)*\id}((\id;a_1),(i_1;a_2),...,(a_l;\id),(\id;i_l),f)\rightarrow\\
&\xrightarrow{((\id,i_1;a_2,\id)*\id)^{-1}}((\id;a_1),(\id;\widetilde{a}_2),(\widetilde{i}_1;\id),(\id;a_3),...,f)\xrightarrow{(\widetilde{i}_1,\id;\id,a_3)*\id}((\id;a_1),(\id;\widetilde{a}_2),(\widetilde{i}_1;a_3),...,f)\xrightarrow{...}\\
&\xrightarrow{...}((\id;a'_1),...,(\id;a'_{n+1}),(i'_1;\id),...,(i'_{n+1};\id),f)\xrightarrow{(\id,...,\id;a'_1,...,a'_{n+1})*(i'_1,...,i'_{n+1};\id,...,\id)}((a'_{n+1}\circ...\circ a'_1;\\&i'_{n+1}\circ...\circ i'_1),f),
\end{align*}
where $a_2\circ i_1\cong \widetilde{i}_1\circ \widetilde{a}_2$ is the active/inert factorization in $\cfact^\ap$. Before moving on, we need to prove that it actually defines a functor. In order to do this it suffices to prove that for every elementary $n$-morphism $(s,f)$ as above and $0\leq q<n$ we have
\[d_q^\pm (G(s,f))\cong G(d_q^\pm (s,f)).\]
To do this, observe that by their definition both $d_q^\pm (G(s,f))$ and $G(d_q^\pm (s,f))$ are $q$-morphisms in the image of $\beta_{(s,f)}:\sq^\Delta(S'_n)\rightarrow N(\cL^\lax(S))$. Since by \cref{prop:square_contractibility} all $\sq^\Delta(S'_n)$ are contractible, we see that all $q$-morphisms in this image are isomorphic.\par
Observe that $FG$ sends an object of $N(\widetilde{\cL^\lax(S)})$ of the form $((a_1;i_1),...,(a_l;i_l),f)$ to \[((a_1;\id),(\id;i_1),...,(a_l;\id),(\id;i_l),f).\] Since $N(\widetilde{\cL^\lax(S)})\cong \widetilde{\cL^\lax(\sq(S))}(\theta_{1,1})$, we see that $FG\cong \id$. We will now provide a homotopy $\eta:GF\xrightarrow{\sim}\id$. More specifically, we will provide the following data:
\begin{itemize}
    \item For every object $x\in \cL^\lax(S)$ a morphism $\eta_x:x\rightarrow GF x$;
    \item For every (possibly non-commutative) diagram
    \[
    \begin{tikzcd}[row sep=huge, column sep=huge]
    x\arrow[r, bend right=15, "(\epsilon)" ]\arrow[r, bend left = 15]&y\\
    GFx\arrow[u, "\eta_x"]\arrow[r, bend right=15, "(GF\epsilon)"]\arrow[r, bend left = 15]&y\arrow[u, "\eta_y"]
    \end{tikzcd}
    \]
    in which $\epsilon$ is an $n$-morphism a morphism $\alpha_\epsilon:\openbox^\lax_n\rightarrow N(\cL^\lax(S))$ such that $\alpha_\epsilon\circ p^n_{0,1} \cong \eta_{x,y}$, $\alpha_\epsilon\circ k^n_0 \cong \epsilon$, $\alpha_\epsilon\circ k^n_1 \cong GF\epsilon$ and $\alpha_\epsilon\circ i^n_{0,1}\cong \alpha_{d^\pm_{n-1}\epsilon}$ in the notation of \cref{constr:lax_square};
\end{itemize}
such that $\alpha_\bullet$ is additionally compatible with the compositions of $n$-morphisms. Now, since the $n$-morphisms of $\cL^\lax(S)$ are freely generated under composition by elementary morphisms, it suffices to provide $\alpha_\epsilon$ only for elementary morphisms $\epsilon$. \par
We first describe the 1-morphisms $\eta_x$. First, assume that $x$ is represented by a morphism $((i;a),f)$ in $\infty-\cat(\widehat{\cL}L^\lax\sq(S)(\theta(1;a)))$, in that case we set 
\[\eta_x\bydef ((\id,i;a,\id),f).\]
The target of this morphism is $x$ and the source is the string $((\id;\widetilde{a})(\widetilde{i};\id),f)$, where $\widetilde{i}\circ \widetilde{a}\cong a\circ i$ is the active/inert decomposition in $\cfact^\ap$, which is equal to $GFx$. For a more general object of the form $y\bydef((i_1;a_1),...,(i_l,a_l),f)$ we set
\[\eta_y\bydef \eta_{(i_1;a_1)}*\eta_{(i_2;a_2)}*...*\eta_{i_l;a_l}.\]
Before we move on to constructing $\alpha_\epsilon$ we make the following observation: to every morphism elementary $n$-morphism $(s',g)$ in $\cL^\lax(S)$ we can associate a morphism $\gamma_{(s',g)}:\sq^\Delta([n+1]\times[n+1])\rightarrow N(\cL^\lax(S))$ in the same manner we have constructed $\beta_{(s,f)}$ earlier. Now, we will construct the required morphisms $\alpha_\epsilon$ by induction on the dimension of $\epsilon$, starting with the case of a $1$-morphism. Assume that $\epsilon$ is an elementary 1-morphism given by $((i_1,i_2;a_1,a_2),f)$, moreover assume that we have the following factorization diagram in $\cfact^\ap$:
\[
\begin{tikzcd}[row sep=huge,column sep=huge]
\bullet\arrow[r, tail, "i_5"]&\bullet\arrow[r, tail, "i_6"]&\bullet\\
\bullet\arrow[u, two heads, "a_6"]\arrow[r, tail, "i_3"]&\bullet\arrow[r, tail, "i_4"]\arrow[u, two heads, "a_4"]&\bullet\arrow[u, two heads, "a_2"]\\\
\bullet\arrow[u, two heads, "a_5"]\arrow[r, tail, "i_1"]&\bullet\arrow[r, tail, "i_2"]\arrow[u, two heads, "a_3"]&\bullet\arrow[u, two heads, "a_1"]
\end{tikzcd}.
\]
Then by untangling definitions we see that $\alpha_{(i_1,i_2;a_1,a_2),f}$ should fit into the following diagram:
\[
\begin{tikzcd}[row sep=huge, column sep=huge]
((i_1;a_3),(i_4;a_2))\arrow[r, "{(i_1,i_2;a_1,a_2)}"]&(i_2\circ i_1;a_2\circ a_1)\\
((\id;a_5),(i_3;\id),(\id;a_4),(i_6;\id))\arrow[u, "{(\id,i_1;a_3,\id)*(\id,i_4;a_2,\id)}"]\arrow[d, "{(i_3,\id;\id,a_4)}" swap]&((\id;a_6\circ a_5),(i_6\circ i_5;\id))\arrow[u, "{(\id,i_2\circ i_1;a_2\circ a_1,\id)}" swap]\\
((\id;a_5),(i_3;a_4),(i_6;\id))\arrow[r, "{(\id,i_3;a_4,\id)^{-1}}"]&((\id;a_5),(\id;a_6),(i_5\id),(i_6;\id))\arrow[u, "{(\id,\id;a_5,a_6)*(i_5,i_6;\id,\id)}" swap]
\end{tikzcd}.
\]
However, observe that all the morphisms in this diagram are lie in the image of $\sq^\Delta([2]\times[2])$ under $\gamma_{((i_1,i_2;a_1,a_2),f)}$. Since $\sq^\Delta([2]\times[2])$ is contractible by \cref{prop:square_contractibility}, there is a $2$-morphism connecting the two composites in the above diagram moreover it may be chosen to lie in the image of $\gamma_{((i_1,i_2;a_1,a_2),f)}$, we define $\alpha_{(i_1,i_2;a_1,a_2),f}$ to be this 2-morphism. Now assume that we have already defined $\alpha_{\epsilon}$ for all elementary morphisms of dimension $<n$, and moreover that all $\alpha_\epsilon$ lie in the image of $\gamma_\epsilon$. Assume we are given an elementary $n$-morphism $\epsilon_n$. In this case we see that $\alpha_{d^\pm_p \epsilon_n}$  for $0\leq p<n$ collectively define a morphism $\alpha_{\partial\epsilon_n}:\partial \openbox^\lax_n\rightarrow N(\cL^\lax(S))$ and moreover by our assumptions it factors as
\[
\begin{tikzcd}[row sep=huge, column sep=huge]
\partial \openbox^\lax_n\arrow[r, "\widetilde{\alpha}_{\partial\epsilon_n}"]\arrow[dr, "\alpha_{\partial\epsilon_n}"]&\sq^\Delta([n+1]\times[n+1])\arrow[d, "\gamma_{\epsilon_n}"]\\
{}&N(\cL^\lax(S))
\end{tikzcd}.
\]
Finally, observe that since $\sq^\Delta([n+1]\times[n+1])$ is contractible and $\partial \openbox^\lax_n$ is homotopy equivalent to an $n$-sphere we see that $\alpha_{\partial\epsilon_n}$ extends to $\alpha_{\epsilon_n}:\openbox^\lax_n\rightarrow N(\cL^\lax(S))$ and moreover it factors through $\gamma_{\epsilon_n}$.
\end{proof}
\begin{definition}\label{def:distributive_law}
Assume we are given monads $T_{1,2}:\rB \Delta_a\rightarrow\cB$. Observe that there are natural inclusions $i_{1,2}:\rB\Delta\rightarrow \rB\Delta_a\otimes\rB\Delta_a$ given by sending $[1]$ to $[1]\otimes[0]$ and $[0]\otimes[1]$ respectively. We say that a morphism $D:\rB\Delta_a\otimes\rB\Delta_a\rightarrow\cB$ is a \textit{distributive law} between $T_1$ and $T_2$ if $D\circ i_1\cong T_1$ and $D\circ i_2\cong T_2$. 
\end{definition}
\begin{remark}
Using the description of $L^\lax\cF$ in \cref{prop:lax_functor1} and the equivalence $\Delta^{\act,\op}\cong \Delta_a$ of \cite[Lemma 4.29.]{kositsyn2021completeness}, we see that $L^\lax*\cong \rB\Delta_a\times \rB\Delta_a$. It follows that $\rB\Delta_a\otimes \rB\Delta_a\cong \cL^\lax*$, so a distributive law in a bicategory $\cB$ can e equivalently described as a lax functor from $*$ viewed as a category with a factorization system. Using \cref{prop:lax_for_objects_of_fact} we see that $\cL^\lax*$ is a category with a single object $*$ whose category of endomorphisms is equivalent to $\cfact^\ap$ with the monoidal structure given by concatenation of strings. Observe that the morphism $t:\ccat\rightarrow\cfact$ of \cref{prop:subcategories} induces a functor $t^\act:\ccat^\act\rightarrow\cfact^\ap$. Once again using the equivalence $\Delta^{\act,\op}\cong \Delta_a$ we see that $t^\act$ induces a functor $\rB \Delta_a\rightarrow\cL^\lax*\cong\rB\Delta_a\otimes \rB\Delta_a$. Assuming we have a distributive law $D:\rB\Delta_a\otimes \rB\Delta_a\rightarrow\cB$ we see that $t$ sends the morphism $[1]$ in $\rB\Delta_a$ to $T_1 T_2$. It follows that, if we are given a distributive law between monads  $T_1$ and $T_2$, their composition also admits a structure of a monad.
\end{remark}
\begin{remark}\label{rem:distributive_equivalence}
Using \cref{prop:gray_universal_property}, we see that giving a distributive law $D$ in $\cB$ is equivalent to giving a monad in the category of monads. This is the characterization of distributive laws used for example in section 6 of \cite{street1972formal}.
\end{remark}
\section{Completeness for categories with factorization systems}\label{sect:four}
The goal of this section is to develop the notion of a groupoid object in $\fact$ and a complete object in $\fact$, similarly to the corresponding notions in $\cat$. We define these concepts in \cref{def:completeness}. We prove that the inclusion $\fact^\gpd\hookrightarrow\fact$ admits a right adjoint in \cref{prop:max_subgroupoid}, which is analogous to the "maximal subgroupoid" functor in $\cat$. However, most of this section is dedicated to constructing an auxiliary theory $\cfact^\wrr$ and describing its category of models. The theory $\cfact^\wrr$ is defined in \cref{constr:wrr_factorization}. Its category of models $\fact^\wrr$ turns out to be isomorphic to a category of morphisms $(F:\cG\rightarrow\cF)$ in $\fact$ that induce an isomorphism on the space of objects, which is proved in \cref{wrr_models} and \cref{prop:wrr_comparison}. Finally, we use this category to construct a completion functor for categories with factorization systems in \cref{cor:completion}.
\begin{lemma}\label{lem:pushout}
For an arbitrary active morphism $a:S\twoheadrightarrow S'$ and a surjective active morphism $s:S\twoheadrightarrow S''$ in $\cfact^\op_{/\cF}$ there exists an object $S_0\in\cfact^\op$ fitting in the following pushout diagram
\[
\begin{tikzcd}[row sep=huge, column sep=huge]
S\arrow[r, two heads, "a"]\arrow[d, two heads, "s"]&S'\arrow[d, two heads, "p"]\\
S''\arrow[r, two heads,"b"]&S_0
\end{tikzcd}
\]
in $\cfact^\op$.
\end{lemma}
\begin{proof}
Since $s$ is a surjective morphism, the string $S''$ is obtained from $S$ by contracting some of its elementary segments. Since $a$ is an active morphism, every elementary segment of $S'$ is in the image of a unique elementary segment of $S$. We define $p:S'\twoheadrightarrow S_0$ to be the surjective morphism that contracts all elementary segments that lie in the image of an elementary segment contracted by $s$. This also uniquely determines the morphism $b$: indeed, since $s$ is surjective, every elementary segment of $S''$ is in the image of a unique elementary segment of $S$, and we define $b$ to send an elementary segment of $S''$ to the image under $p\circ a$ of its preimage in $S$. \par
To prove that this is indeed a pushout, observe that a cocone over this pushout diagram can be identified with a morphism $f:S'\rightarrow S_1$ in $\cfact^\op$ such that if $e\in S$ is an elementary segment such that $s(e)$ is an identity morphism and $a(e)\cong e'_n\circ e'_{n-1}\circ...\circ e'_1$, then $f(e'_n\circ e'_{n-1}\circ...\circ e'_1)\cong \id$. However, this is only possible if all $e'_i$ map to identities, so $f$ must factor through $S_0$. 
\end{proof}
\begin{warning*}
The pushout described in \cref{lem:pushout} generally does not remain a pushout in $\fact$.
\end{warning*}
\begin{construction}\label{constr:wrr_factorization}
We denote by $\cfact^\wrr$ the opposite of the following category: its objects can be identified with morphisms $\sq(S')\xactive{s}\sq(S)$ in which $s$ is a surjective active morphism and a morphism from $(s_0,f_0)$ to $(s_1,f_1)$ is given by a commutative diagram of the form
\begin{equation}\label{eq:thirteen}
\begin{tikzcd}[row sep=huge, column sep=huge]
\sq(S_0)\arrow[d, two heads, "s_0"]\arrow[r, "h"]&\sq(S_1)\arrow[d, two heads, "s_1"]\\
\sq(S'_0)\arrow[r, "h'"]&\sq(S'_1)
\end{tikzcd}
\end{equation}
in $\fact$, such that $h$ additionally satisfies the following conditions:
\begin{itemize}
    \item if $x_0\in S'_0$ is such that $h'(x_0)\notin S'_1\subset\sq(S'_1)$, then for every elementary segment $i:e\rightarrowtail S_0$ such that $s_0\circ i= \id_{x_0}$ the morphism $h\circ i$ is a constant functor;
    \item if $h'(x_0)\in S_1'$, then the following condition is satisfied: assume that  $i:e\rightarrowtail S_0$ is an elementary segment of $S_0$ such that $s_0\circ i= \id_{x_0}$, denote by $h\circ i\cong j\circ c\circ a$ the unique decomposition in which $a$ is active, $c$ is a covering morphism and $j$ is an inclusion. Then the image of $s_1\circ j$ in $\sq(S'_1)$ is equal to $\{h'(x_0)\}$. 
\end{itemize} 
\par We will call a morphism inert (resp. an inclusion) if the morphism $h$ is inert (resp. an inclusion) and denote the corresponding subcategory of $\cfact^\wrr$ by $(\cfact^\wrr)^\inrt$ (resp. $(\cfact^\wrr)^\inc$). We will denote by $k:(\cfact^\wrr)^\inrt\hookrightarrow\cfact^\wrr$ and  $j:(\cfact^\wrr)^\inc\hookrightarrow(\cfact^\wrr)^\inrt$ the inclusions of the corresponding subcategories. We will call a morphism $(h,h')$ active if $h$ is an active morphism in $\cfact^\op$ and the commutative diagram (\ref{eq:thirteen}) is a pushout square.\par
We will declare an object elementary if it either has the from $e\twoheadrightarrow*$ or $e=\joinrel= e$ for an elementary object $e$ of $\cfact$, we will denote by $i:(\cfact^\wrr)^\el\hookrightarrow(\cfact^\wrr)^\inc$ the full subcategory of $(\cfact^\wrr)^\inc$ on elementary objects. For an object $s$ of $\cfact^\wrr$ we will denote by $(\cfact^\wrr)^\el_{s/}$ the category $s/i$. Denote by $\cE_{\cfact^\wrr}$ the full subcategory of $\mor_\cat((\cfact^\wrr)^\inrt,\cS)$ on those functors $\cG:(\cfact^\wrr)^\inrt\rightarrow\cS$ for which
\[j*\cG\cong i_*i^* j^*\cG.\]
\end{construction}
\begin{remark}
To an object $s_0$ of $\cfact^\wrr$ we can associate a pair consisting of the object $\sq(S_0)$ of $\fact$ and the factorization subcategory $\cI_{s_0}\hookrightarrow \sq(S_0)$ that contains all objects of $\sq(S_0)$, but only those morphisms that are given by factorizations and compositions of elementary morphisms $e\in S_0$ that are sent to identity morphisms by $s_0$. A morphism $h:\sq(S_0)\rightarrow \sq(S_1)$ as in the diagram above obtains a structure of a morphism in $\cfact^\wrr$ if and only if it takes $\cI_{s_0}$ to $\cI_{s_1}$.
\end{remark}
\begin{prop}\label{prop:wrr_theory}
The active and inert morphisms form a factorization system on $\cfact^\wrr$.
\end{prop}
\begin{proof}
Assume that the morphism $h$ in diagram \ref{eq:thirteen} has the active/inert factorization of the form $(a,i)$. By \cref{lem:pushout} we then have the following commutative diagram in $\cfact^\op$
\[
\begin{tikzcd}[row sep=huge, column sep=huge]
\sq(S_0)\arrow[d, two heads, "s_0"]\arrow[r, two heads, "a"]&\sq(S_a)\arrow[r, tail, "i"]\arrow[d, two heads, "s_a"]&\sq(S_1)\arrow[d, two heads, "s_1"]\\
\sq(S'_0)\arrow[r, two heads, "a'"]&\sq(S'_a)\arrow[r, "i'"]&\sq(S'_1)
\end{tikzcd}
\]
in which the left square is a pushout diagram and the morphism $i'$ is obtained by the universal property of the pushout from $(s_1\circ i, h')$. We now need to prove that the morphisms given by this diagram actually lie in $\cfact^\wrr$. Since an active morphism either takes an elementary morphism to an identity morphism or to a composition of elementary morphisms, we see that the conditions of \cref{constr:wrr_factorization} are satisfied for the left square in the diagram, so it remains to prove that the right square represents a morphism in $\cfact^\wrr$. Recall that 
\[\Act_\cfact(S)\cong \prod_{e\in S} \Act_\cfact(e),\]
where the product runs over all elementary segments of $S$. The morphism $a'$ corresponds to the element $(a_{s_0^{-1}(e_1)},a_{s_0^{-1}(e_2)},...,a_{s_0^{-1}(e_m)})$ of $\Act_\cfact(S'_0)$, where $s_0^{-1}(e_j)$ is the unique elementary segment of $S_0$ that maps to $e_i\in S'_0$ (which exists since $s_0$ is surjective) and $a_{s_0^{-1}(e_j)}$ is an element of $\Act_\cfact(e_j)$ corresponding to $a$. Assume that $x\in S'_a$ does not lie in the image of $a'$. In this case the fiber $s_a^{-1}(x)=*$, so the conditions of \cref{constr:wrr_factorization} are satisfied automatically. Now assume that $x=a(x_0)$. Denote by $S_{x_0}$ the substring of $S_0$ on those elementary segments that are sent to $x_0$ by $s_0$ and by $a_{x_0}:\sq(S_{x_0})\twoheadrightarrow \sq(S_{a,x_0})$ the active morphism corresponding to the restriction of $a$ to $S_{x_0}$, then the elementary segment of $S_a$ lying in $s_a^{-1}(x)$ are given by the elementary segments of $S_{a,x_0}$. However, all of them are in the image under $a_{x_0}$ of the elementary segments in $S_{x_0}$, and it is easy to see that the conditions of \cref{constr:wrr_factorization} are satisfied for them since they are satisfied for the elementary morphisms in $S_{x_0}$ by assumption. It follows that the right square indeed corresponds to a morphism in $\cfact^\wrr$ and the whole diagram provides the required factorization. The fact that it is unique follows from the uniqueness of the active/inert factorization in $\cfact$ and the universal property of the pushout.
\end{proof}
\begin{lemma}\label{lem:coinitiality}
For an object $s\in\cfact^\wrr$ denote by $(\cfact^\wrr)^\el_{0,s/}$ the full subcategory of $(\cfact^\wrr)^\el_{s/}$ on those morphisms $s\xinert{(i,i')} e$ for which both $i$ and $i'$ are inert in $\cfact$. Then $(\cfact^\wrr)^\el_{0,s/}$ is coinitial in $(\cfact^\wrr)^\el_{s/}$.
\end{lemma}
\begin{proof}
Indeed, by an easy inspection the only objects of $(\cfact^\wrr)^\el_{s/}$ that do \textit{not} belong to $(\cfact^\wrr)^\el_{0,s/}$ are given by diagrams of the form 
\[
\begin{tikzcd}[row sep=huge, column sep=huge]
e\arrow[d, equal]\arrow[r, tail, "i"]&\sq(S)\arrow[d, two heads, "s"]\\
e\arrow[r, "\{x\}"]&\sq(S')
\end{tikzcd}
\]
in $\cfact^\op$ in which the top row is an inclusion of an elementary segment of $S$ that is sent to an object $x\in\sq(S')$ by the surjective morphism $s$. However, observe that this morphism uniquely factors as
\[
\begin{tikzcd}[row sep=huge, column sep=huge]
e\arrow[d, equal]\arrow[r, equal]&e\arrow[r, tail, "i"]\arrow[d, two heads, "p"]&\sq(S)\arrow[d, two heads, "s"]\\
e\arrow[r, two heads, "p"]&*\arrow[r, "\{x\}"]&\sq(S'),
\end{tikzcd}
\]
where $p$ is the unique surjective morphism from $e$ to $*$. This concludes the proof since both morphisms in this diagram are inert in $\cfact^\wrr$ and the right square obviously belongs to $(\cfact^\wrr)^\el_{0,s/}$.
\end{proof}
\begin{prop}\label{wrr_theory}
The subcategory $\cE_{\cfact^\wrr}$ endows $k:(\cfact^\wrr)^\inrt\hookrightarrow\cfact^\wrr$ with the structure of a theory with arities.
\end{prop}
\begin{proof}
We need to prove that the monad $k_!k^*$ on $\mor_\cat((\cfact^\wrr)^\inrt,\cS)$ preserves the image of $\cE_{\cfact^\wrr}$. We first prove that for an object $s:\sq(S)\twoheadrightarrow\sq(S')$ of $\cfact^\wrr$ we have
\[\Act_{\cfact^\wrr}(s)\cong \underset{(s\rightarrowtail e)\in(\cfact^\wrr)^\el_{0,s/}}{\lim}\Act_{\cfact^\wrr}(e).\]
To do this, we first make the structure of $(\cfact^\wrr)^\el_{0,s/}$ more explicit: its objects are can be identified with pairs $(j:e\rightarrowtail sq(S),i)$, where $e$ is an elementary object of $\cfact$, $j$ is an inert morphisms in $\cfact^\op$ and $i\in\{0,1\}$ is equal to $0$ if $s\circ j$ sends $e$ to a single point and to $1$ if $s\circ j$ is an inert morphism. Morphisms between objects of $(\cfact^\wrr)^\el_{0,s/}$ are given by morphisms in $\cfact^\el_{S/}$. Now observe that by definition $\Act_{\cfact^\wrr}(s)\cong\Act_\cfact (S)$, and so we have
\begin{align*}
    \Act_{\cfact^\wrr}(s)\cong&\Act_\cfact (S)\\
    \cong & \underset{(j:S\rightarrowtail e)\in \cfact^\el_{S/}}{\lim}\Act_\cfact(e)\\
    \cong & \underset{(j:S\rightarrowtail e, i)\in \cfact^\el_{0,S/}}{\lim}\Act_\cfact(e)\\
    \cong &\underset{(j:S\rightarrowtail e, i)\in \cfact^\el_{0,S/}}{\lim}\Act_{\cfact^\wrr}(e)\\
    \cong \cong &\underset{(j:s\rightarrowtail e)\in \cfact^\el_{S/}}{\lim}\Act_{\cfact^\wrr}(e),
\end{align*}
in which the last isomorphism follows from \cref{lem:coinitiality}.\par
We will now prove that for any factorization square 
\[
\begin{tikzcd}[row sep=huge, column sep=huge]
s_0\arrow[r, two heads, "a"]\arrow[d, tail, "a^* i"]&s_1\arrow[d, tail, "i"]\\
i_!s_0\arrow[r, two heads,"i_! a"]&e
\end{tikzcd}
\]
in $\cfact^\wrr$ we have 
\[(\cfact^\wrr)^\el_{0,s_0/}\cong \underset{i\in(\cfact^\wrr)^\el_{0,s_1/}}{\lim}(\cfact^\wrr)^\el_{0,i_!s_0/}.\]
Observe that $(\cfact^\wrr)^\el_{0,i_!s_0/}$ is a full subcategory of $(\cfact^\wrr)^\el_{0,s_0/}$ on those elementary segments that lie in the image of $e$ under $a$ (considered as a morphism in $(\cfact^\wrr)^\op$). Since $a$ is an active morphism, we see that those images cover $s_0$ and only possibly intersect along boundary points, which easily implies our claim. Now we can finish the proof of the proposition in exactly the same way as in \cref{prop:factorization_theory}.
\end{proof}
\begin{construction}\label{constr:wrr_alternative_construction}
We will endow the category $[1]\times\cfact$ with the structure of a theory as follows: we declare $([1]\times\cfact)^\inrt\bydef[1]\times\cfact^\inrt$, $([1]\times\cfact)^\inc\bydef[1]\times\cfact^\inc$ and $([1]\times\cfact)^\act\bydef\cfact^\act\coprod\cfact^\act$. We declare an object $(i,e)\in[1]\times\cfact$ elementary if $e$ is an elementary object of $\cfact$. Finally, we define $\cE_{[1]\times\cfact}$ to be the full subcategory of $\mor_\cat([1]\times\cfact^\inrt,\cS)$ on those functors whose restriction to $([1]\times\cfact)^\inc$ is the right Kan extension of their restriction to $([1]\times\cfact)^\el$. \par
We define $\widehat{\cfact^\wrr}$ to be the full subcategory of $([1]\times\cfact)$ containing all objects except $(0,*)$. We define the structure of a theory on $\widehat{\cfact^\wrr}$ to be the one induced from $[1]\times\cfact$.
\end{construction}
\begin{prop}\label{wrr_models}
The category of models for $\widehat{\cfact^\wrr}$ is isomorphic to the category of triples $(\cG,\cF,F:\cG\rightarrow\cF)$, where $\cF$ and $\cG$ are objects of $\fact$ and $F$ induces an isomorphism on the space of objects.
\end{prop}
\begin{proof}
It follows easily from the definition that the category of models for $[1]\times\cfact$ is the category $[1]\times\fact$ of triples $(\cG,\cF,F:\cG\rightarrow\cF)$ (without any additional assumptions on $F$). We will denote by $\widehat{i}:\widehat{\cfact^\wrr}\hookrightarrow[1]\times\cfact$ the natural inclusion. Since $\widehat{i}$ is an inclusion of a full subcategory, \[\widehat{i}_*:\mor_\cat(\widehat{\cfact^\wrr},\cS)\rightarrow \mor_\cat([1]\times\cfact,\cS)\]
induces an isomorphism between $\mor_\cat(\widehat{\cfact^\wrr},\cS)$ and the full subcategory of $\mor_\cat([1]\times\cfact,\cS)$ on those $\cF$ for which $\widehat{i}_*\widehat{i}^*\cF\cong\cF$.\par
It now suffices to prove that $\widehat{i}_*$ induces a functor between the corresponding categories of models. Indeed, it is easy to see that for $\widehat{\cF}\in \mor_\cat(\widehat{\cfact^\wrr},\cS)$ we have $\widehat{i}_*\widehat{\cF}(i,S)\cong \widehat{\cF}(i,S)$ for $(i,S)\neq (0,*)$ and $\widehat{i}_*\widehat{\cF}(0,*)\cong\widehat{\cF}(1,*)$, so in particular $\widehat{i}_*\widehat{\cF}\in[1]\times\fact$. It follows that $\widehat{\fact^\wrr}$ is isomorphic to the full subcategory of $[1]\times\fact$ on those objects $X\bydef(\cG,\cF,F:\cG\rightarrow\cF)$ for which $F$ induces an isomorphism between $X(0,*)\cong \cF_0$ and $X(1,*)\cong \cG_0$, which is exactly what we needed to prove.
\end{proof}
\begin{prop}\label{prop:wrr_comparison}
The category of model for $\cfact^\wrr$ and $\widehat{\cfact^\wrr}$ are isomorphic.
\end{prop}
\begin{proof}
We denote by $j:\widehat{\cfact^\wrr}\hookrightarrow\cfact^\wrr$ the functor that sends $(0,S)$ to the unique morphism $S\rightarrow*$, $(1,S)$ to $\id:S=\joinrel=S$ and sends the natural morphism $(0,S)\rightarrow(1,S)$ to the diagram
\[
\begin{tikzcd}[row sep=huge, column sep=huge]
S\arrow[d, equal]\arrow[r, equal]&S\arrow[d]\\
S\arrow[r]&*
\end{tikzcd}
\]
in $\cfact$. Observe that $j$ is an inclusion of a full subcategory. Moreover, it respects elementary objects, active and inert morphisms, and inclusions. We will prove that $j_*$ induces an isomorphism between the categories of models.\par
To do this we first need to show that it restricts to $j_*:\widehat{\fact^\wrr}\rightarrow\fact^\wrr$. Observe that \[\Act_{\widehat{\cfact^\wrr}}(i,S)\cong \Act_{\cfact^\wrr}(j(i,S))\cong\Act_{\cfact}(S).\]
It follows that $j$ uniquely lifts active morphisms. In particular, if for a given $s\in\cfact^\wrr$ we denote by $(s/j)^\inrt\hookrightarrow (s/j)$ the full subcategory on inert morphisms $s\rightarrowtail j(x)$, then it is coinitial. Moreover, observe that any inert morphism $s\xinert{k} j(x)$ admits a further unique factorization of the form $j(c)\circ i$, where $i$ is an inclusion in $\cfact^\wrr$ and $c$ is a covering morphism. Namely, for a morphism $s\xinert{k'}j(S'_0,1)$ given by the diagram
\[
\begin{tikzcd}[row sep=huge, column sep=huge]
\sq(S'_0)\arrow[d, equal]\arrow[r, tail, "k'"]&\sq(S)\arrow[d, two heads, "s"]\\
\sq(S'_0)\arrow[r, "s\circ k'"]&\sq(S')
\end{tikzcd}
\]
the factorization is given by
\[
\begin{tikzcd}[row sep=huge, column sep=huge]
\sq(S'_0)\arrow[d, equal]\arrow[r, tail, "c'"]&\sq(S'_1)\arrow[d, equal]\arrow[r, tail, "i'"]&\sq(S)\arrow[d, two heads, "s"]\\
\sq(S'_0)\arrow[r, tail, "c'"]&\sq(S'_1)\arrow[r, "s\circ i'"]&\sq(S')
\end{tikzcd}.
\]
For a morphism $s\xinert{k}j(S_0,0)$ represented by the diagram
\[
\begin{tikzcd}[row sep=huge, column sep=huge]
\sq(S_0)\arrow[d, two heads]\arrow[r, tail, "k"]&\sq(S)\arrow[d, two heads, "s"]\\
*\arrow[r, "\{x\}"]&\sq(S')
\end{tikzcd}
\]
in which $\sq(S_0)\neq*$ (which is covered by the previous case) it follows directly from the conditions of \cref{constr:wrr_factorization} and the fact that $j$ is inert (and in particular injective, so it does not send non-trivial elementary morphisms to identities) that we have the following diagram
\[
\begin{tikzcd}[row sep=huge, column sep=huge]
\sq(S_0)\arrow[d, two heads]\arrow[r, tail, "c"]&\sq(S_1)\arrow[r, tail, "i"]\arrow[d, two heads]&\sq(S)\arrow[d, two heads, "s"]\\
*\arrow[r, equal]&*\arrow[r, "\{x\}"]&\sq(S')
\end{tikzcd},
\]
in which the top row is the factorization of \cref{prop:coverings_factorization}. 
It is obvious that in both cases this factorization is unique. It follows that the inclusion $(s/j)^\inc\hookrightarrow (s/j)^\inrt$ of the full subcategory on inclusions is also a coinitial morphism. Using this we compute for $\widehat{\cF}\in\widehat{\fact^\wrr}$
\begin{align*}
    j_*\widehat{\cF}(s)\cong& \underset{(s\xrightarrow{f}j(y))\in(s/j)}{\lim}\widehat{\cF}(y)\\
    \cong &\underset{(s\xinert{i}j(x))\in(s/j)^\inc}{\lim}\widehat{\cF}(x)\\
    \cong &\underset{(s\xinert{i}j(x))\in(s/j)^\inc}{\lim}\underset{(x\xinert{k}e)\in(\widehat{\cfact^\wrr})^\el_{x/}}{\lim}\widehat{\cF}(e)\\
    \cong & \underset{(s\xinert{i}e)\in(\cfact^\wrr)^\el_{s/}}{\lim}\underset{(e\xinert{k}j(x))\in(e/j)^\inc}{\lim}\widehat{\cF}(x)\\
    \cong& \underset{(s\xinert{i}e)\in(\cfact^\wrr)^\el_{s/}}{\lim}j_*\widehat{\cF}(e),
\end{align*}
in which the first isomorphism follows by definition, the second by the coinitiality of $(s/j)^\inc$, the third since $\widehat{\cF}$ is a model, the fourth by the naturality of right Kan extensions and the existence of a commutative diagram
\[
\begin{tikzcd}[row sep=huge, column sep=huge]
\widehat{\cfact^\wrr}^\el\arrow[d]\arrow[r]&(\cfact^\wrr)^\el\arrow[d]\\
\widehat{\cfact^\wrr}^\inc\arrow[r]&(\cfact^\wrr)^\inc
\end{tikzcd}.
\]
It follows that $j_*$ takes $\widehat{\fact^\wrr}$ to $\fact^\wrr$. Since $j$ is an inclusion of a full subcategory, we have $j^*j_*\cong \id$. It remains to prove that $j_*j^*\cong \id$. To do this, observe that since every elementary object is in the image of $j$ we have $j_*j^*\cF(e)\cong\cF(e)$ for every $\cF\in\fact^\wrr$. However, since both $\cF$ and $j_*j^*\cF(e)$ satisfy the Segal condition, it follows that they are isomorphic.
\end{proof}
\begin{definition}\label{def:completeness}
For a non-negative integer $n$ we will denote by $c_n$ the category given by the string of $n$ invertible morphisms. We can also describe it as a category with $(n+1)$ objects such that for any pair of objects $(i,j)$ in $c_n$ we have $\mor_{c_n}(i,j)\cong*$. For every $n,m\geq 0$ we will consider $c_m\times c_n$ as an object of $\fact$ using \cref{rem:products}. Observe that there is a morphism $i:t([1])\rightarrow c_1\times c_1$ given by the composition of inclusions 
\[t([1])\rightarrow[1]\times[1]\rightarrow c_1\times c_1.\] \par 
We will call an object $\cF\in\fact$ a \textit{groupoid} if every morphism $t([1])\rightarrow\cF$ uniquely factors through $i$. We will call it \textit{complete} if every morphism $c_1\times c_1\rightarrow \cF$ factors through $*$. We will denote by $\fact^\comp$ the full subcategory of $\fact$ on complete objects and by $\fact^\gpd$ the full subcategory of $\fact$ on groupoids.
\end{definition}
\begin{prop}\label{prop:groupoid_definition}
An object $\cF$ of $\fact$ is a groupoid if and only if its underlying category $t^*\cF$ is a groupoid.
\end{prop}
\begin{proof}
If $\cF$ is a groupoid, then the fact that $t^*\cF$ is a groupoid follows immediately from the definition. Assume that $t^*\cF$ is a groupoid, we need to show that every $f:t([1])\rightarrow\cF$ factors through $c_1\times c_1$. The morphism $f$ is essentially given by a composable pair $h\circ v$ of a vertical and a horizontal morphism in $\cF$. Since $t^*\cF$ is a groupoid, both $h$ and $v$ have inverses. Let $v^{-1}\circ h^{-1}\cong \widetilde{h}^{-1}\circ \widetilde{v}^{-1}$ be the horizontal/vertical factorization in $\cF$. Observe that we have $h\circ v\cong \widetilde{v}\circ \widetilde{h}$ since both morphisms are inverse to $v^{-1}\circ h^{-1}$ and any two inverses are isomorphic. It follows that $(v,h,\widetilde{v},\widetilde{h})$ define the required morphism $c_1\times c_1\rightarrow\cF$.
\end{proof}
\begin{construction}\label{constr:groupoid_theory}
Denote by $\cfact^\gpd$ the opposite of full subcategory of $\fact$ on objects of the form $c_n\times c_m$. For an object $S\in\cfact$ such that $S$ is a longest composable string of morphisms in $[n]\times[m]$ denote by $j(S)$ the category $c_n\times c_m$. Observe that $j$ extends to a functor from $\cfact$ to $\cfact^\gpd$. Indeed, by the same reasoning as in \cref{prop:groupoid_definition} we see that every morphism in $j(S)$ is obtained from taking factorizations of morphisms of $S$, taking inverses to those factorizations or by taking factorizations of those inverses. All those operations are preserved by morphisms in $\fact$, so it follows that a morphism $f:j(S)\rightarrow c_l\times c_t$ is uniquely determined by its restriction to $S$. With this in mind, for a morphism $f:\sq(S)\rightarrow\sq(S')$ we define $j(f)$ to correspond to the composition
\[S\hookrightarrow \sq(S)\xrightarrow{f}\sq(S')\hookrightarrow j(S').\]
We denote by $\widetilde{\fact^\gpd}$ the full subcategory of $\mor_\cat(\cfact^\gpd, \cS)$ on those functors $\cG$ that satisfy the following condition: for every maximal composable string $S\subset [n]\times[m]$ we have
\[\cG(c_n\times c_m)\cong \underset{(S\rightarrowtail e)\in\cfact^\el_{S/}}{\lim}\cG(j(e)).\]
\end{construction}
\begin{prop}\label{prop:max_subgroupoid}
The category $\widetilde{\fact^\gpd}$ of \cref{constr:groupoid_theory} is isomorphic to $\fact^\gpd$. Moreover, the natural inclusion $i:\fact^\gpd\hookrightarrow\fact$ admits a right adjoint.
\end{prop}
\begin{proof}
Observe that by the definition of $\widetilde{\fact^\gpd}$ the morphism $J^*$ induces a morphism $j^*:\widetilde{\fact^\gpd}\rightarrow \fact$. We first prove that its right adjoint also restricts to a functor $j_*;\fact\rightarrow \widetilde{\fact^\gpd}$. To do this, first observe that every morphism $\sq(S)\rightarrow c_n\times c_m$ uniquely factors through $j(S)$. Using this fact we see that
\begin{align*}
    j_*\cF(c_n\times c_m)\cong &\underset{(j(S)\rightarrow c_n\times c_m)}{\lim}\cF(S)\\
    \cong &\underset{(\sq(S)\rightarrow c_n\times c_m)}{\lim}\cF(S)\\
    \cong &\underset{(\sq(S)\rightarrow c_n\times c_m)}{\lim}\mor_\fact(\sq(S),\cF)\\
    \cong &\mor_\fact(\underset{(\sq(S)\rightarrow c_n\times c_m)}{\colim}\sq(S),\cF)\\
    \cong &\mor_\fact(c_n\times c_m,\cF).
\end{align*}
We have already observed in \cref{constr:groupoid_theory} that any morphism $j(S)\rightarrow \cF$ is uniquely determined by its restriction to $S$. Since every maximal composable chain $S\in[n]\times [m]$ induces an isomorphism $j(S)\cong c_n\times c_m$ and a morphism from $S$ is uniquely determined by its restriction to the elementary segments we see that 
\[j_*\cF(c_n\times c_m)\cong\mor_\fact(c_n\times c_m,\cF)\cong\underset{(S\rightarrowtail e)\in\cfact^\el_{S/}}{\lim}\mor_\fact(j(e),\cF)\cong \underset{(S\rightarrowtail e)\in\cfact^\el_{S/}}{\lim}j_*\cF(j(e)).\]
We now demonstrate that $j^*j_*$ is an idempotent comonad on $\fact$. It suffices to prove that $(j^*j_*)^2\cF(e)\cong j^*j_*\cF(e)$ for all $e\in\cfact^\el$. The case of $e=*$ is trivial since $j^*j_*\cF(*)\cong \cF(*)$. The space $j^*j_*\cF(h)$ (resp. $j^*j_*\cF(v)$) is the space of invertible morphisms in the horizontal category $h^*\cF$ (resp. vertical category $v^*\cF$), and so the fact that $(j^*j_*)^2\cF(h)\cong j^*j_*\cF(h)$ (resp. $(j^*j_*)^2\cF(v)\cong j^*j_*\cF(v)$) follows since the maximal subgroupoid of a groupoid is equivalent to the groupoid itself.\par
The adjunction $j^*\dashv j_*$ is moreover comonadic. Indeed, it follows since $j^*$ preserves all limits and is conservative and \cite[Theorem 4.7.3.5.]{luriehigher}. To finish the proof, observe that the category of coalgebras for $j^*j_*$ is the category of $\cF\in\fact$ for which $j^*j_*\cF\cong \cF$. In other words, of those $\cF$ for which every morphism $\sq(S)\rightarrow\cF$ uniquely factors through $j(S)$. It is easy to see that this is the same category as $\fact^\gpd$, which proves our first claim, and the required right adjoint is then given by $j_*$.
\end{proof}
\begin{cor}\label{cor:completion}
The inclusion $\fact^\comp\hookrightarrow\fact$ admits a left adjoint.
\end{cor}
\begin{proof}
Denote by $p:\cfact^\wrr\rightarrow\cfact$ the morphism that sends $\sq(S)\xactive{s}\sq(S)$ to $\sq(S)$. It is easy to see that $p^*$ induces a morphism $p^*:\fact\rightarrow\fact^\wrr$. Namely, under the identifications of \cref{wrr_models} and \cref{prop:wrr_comparison}, the morphism $p^*$ sends $\cF$ to the pair $(\cF_0\hookrightarrow \cF)$, where $\cF_0$ denotes the space of objects of $\cF$. Observe that it has a left adjoint given by $Lp_!$, where $L$ is the left adjoint to inclusion $\fact\hookrightarrow\mor_\cat(\cfact,\cS)$.\par
Observe that the counit morphism $(j^*j_*\cF\rightarrow\cF)$ is an element of $\fact^\wrr$. We claim that 
\[\widehat{\cF}\bydef Lp_!(j^*j_*\cF\rightarrow\cF)\]
is the required left adjoint. Indeed, a morphism $\widehat{\cF}\rightarrow\cG$ is the same thing as a morphism $f:\cF\rightarrow\cG$ such that the composition $j^*j_*\cF\rightarrow\cF\xrightarrow{f}\cG$ factors through $\cG_0$. However, since $\cG$ belongs to $\fact^\comp$, every morphism from a groupoid to $\cG$ factors through $\cG_0$, so $\mor_\fact(\widehat{\cF},\cG)\cong \mor_\fact(\cF,\cG)$.
\end{proof}
\begin{remark}
It is not true in general that the underlying category of $\widehat{\cF}$ is complete. This is because $j^*j_*\cF$ only contains invertible morphisms for which both the horizontal and vertical components of the factorization are also invertible, but there can be invertible morphisms in $t^*\cF$ for which this is not the case. However, the converse is true: if the underlying category of $\cF$ is complete, then $\cF$ itself is also complete as an object of $\fact$.
\end{remark}
\section{Lax functors to span}\label{sect:five}
This section contains our second main result, \cref{thm:overcategory}, that establishes an equivalence between categories with factorization systems over $\cF$ and lax functors to $\spanc$. To prove it, we first construct a theory $\cfact_/\cF$ in \cref{constr:overcategory} whose category of models is isomorphic to $\fact_{/\cF}$. We then characterize the category of lax functors from $\sq(S)$ to $\spanc$ as a subcategory of a certain presheaf category in \cref{prop:aux2}. Finally, we provide an equivalence between $\fact_{/\cF}$ and that category in \cref{thm:overcategory}. As a corollary, we can characterize the category $\fact$ itself as a category of distributive laws in $\spanc$.
\begin{definition}\label{def:slice_theory}
Assume we have a theory $\cT$ with arities $\cE\subset\mor_\cat(\cT_0,\cS)$ together with a model $X$. To $X$, viewed as a functor $X:\cT_1\rightarrow\cS$, we can associate a left fibration $p:X_1\rightarrow\cT_1$, we also denote by $X_0$ its restriction to $\cT_0$. Denote by $\cE_X$ the full subcategory of $\mor_{\cat}(X_0,\cS)$ on those objects $\cF$ for which $p_!\cF\in\cE$, this gives $(X_0\rightarrow X_1)$ the structure of a theory which we denote $\cT_{/X}$. Observe that under the equivalence
\[\mor_\cat(X_1,\cS)\cong\mor_\cat(\cT_1,\cS)_{/X}\]
of \cite[Proposition 3.12.]{kositsyn2021completeness} the category of models for $\cT_{/X}$ can be associated with the overcategory $\modl_\cT(\cS)_{/X}$.
\end{definition}
\begin{construction}\label{constr:overcategory}
Assume we are given an object $\cF\in\fact$. We can associate to it a left fibration $\cfact_{/\cF}\rightarrow\cfact$ as well as its restriction $\cfact_{/\cF}^\act$ (resp. $\cfact_{/\cF}^\inrt$, $\cfact_{/\cF}^\inc$, $\cfact_{/\cF}^\el$) to $\fact^\act$ (resp. $\fact^\inrt$, $\fact^\inc$, $\fact^\el$), denote by $i_\cF:\cfact_{/\cF}^\el\hookrightarrow\cfact_{/\cF}^\inc$ and $j_\cF:\cfact_{/\cF}^\inc\hookrightarrow\cfact_{/\cF}^\inrt$ the natural inclusions. We endow it with the structure of a theory by declaring $\cE_{\cfact_{/\cF}}$ to be the full subcategory on those $\cG:\cfact_{/\cF}^\inrt\rightarrow\cS$ whose restriction to $\cfact_{/\cF}^\inc$ is the right Kan extension of its restriction to $\cfact_{/\cF}^\el$.
\end{construction}
\begin{prop}\label{prop:model_for_overcategories}
The category of models for $\cfact_{/\cF}$ is equivalent to $\fact_{/\cF}$.
\end{prop}
\begin{proof}
We need to prove that for a morphism $(\cG\xrightarrow{f}\cF)$ of presheaves on $\fact^\op$ the presheaf $\cG$ is an object of $\fact$ if and only if $f$ viewed as a functor $\cfact_{/\cF}\rightarrow\cS$ belongs to $\modl_{\cfact_{/\cF}}(\cS)$. First, assume that $\cG$ is an object of $\fact$ and fix $S\in\fact$ and $x\in\cF(S)$. Then we have the following pullback square
\[
\begin{tikzcd}[row sep=huge, column sep=huge]
\cG_x\arrow[r]\arrow[d]&\cG(S)\cong\underset{\cfact^\el_{S/}}{\lim}\cG(e)\arrow[d,"f"]\\
\underset{\cfact^\el_{S/}}{\lim}*\cong*\arrow[r,"x"]&\cF(S)\cong\underset{\cfact^\el_{S/}}{\lim}\cF(e)
\end{tikzcd}
\]
and the claim follows from the commutativity of limits. Now assume that $i_{\cF,*}i^*_\cF j^*_\cF f\cong j^*_\cF f$. Observe that by \cite[Corollary 7.17]{chu2019homotopy} and the reasoning in \cite[Proposition 2.1.]{kositsyn2021completeness} we have the following commutative diagram
\[
\begin{tikzcd}[row sep=huge, column sep=huge]
\cfact_{/\cF}^\el\arrow[r, "i_{\cF,*}"]\arrow[d, "p^\el_!"]&\cfact_{/\cF}^\inc\arrow[d, "p^\inc_!"]\\
\fact^\el\arrow[r, "i_*"]&\fact^\inc
\end{tikzcd},
\]
and so it follows that 
\[i_*i^*j^*p^\inrt_!f\cong i_*i^*p^\inc_! j^*_\cF f\cong i_* p^\el_! i^*_\cF j^*_\cF f\cong p^\inrt_! i_{\cF,*}i^*_\cF j^*_\cF f\cong p^\inc_! j^*_\cF f\cong j^*p^\inrt_! f,\]
where the first two and the last equivalence follow from the Beck-Chevalley isomorphism.
\end{proof}
\begin{cor}\label{cor:overcategories_and_colimits}
For $\cF\in\fact$ we have an equivalence
\[\fact_{/\cF}\cong\underset{(\sq(S)\xrightarrow{f}\cF)\in\cfact_{/\cF}}{\lim}\fact_{/\sq(S)}.\]
\end{cor}
\begin{proof}
Indeed, since $\cP(\cfact^\op)$ is an $\infty$-topos and $\cF\cong\underset{(\sq(S)\xrightarrow{f}\cF)\in\cfact_{/\cF}}{\colim}\sq(S)$ we have 
\[\cP(\cfact_{/\cF}^\op)\cong\cP(\cfact^\op)_{/\cF}\cong \underset{(\sq(S)\xrightarrow{f}\cF)\in\cfact_{/\cF}}{\lim}\cP(\cfact^\op)_{/\sq(S)}\cong \underset{(\sq(S)\xrightarrow{f}\cF)\in\cfact_{/\cF}}{\lim}\cP(\cfact_{/\sq(S)}^\op).\]
Moreover, it is easy to see that this isomorphism respects the Segal condition on both sides, so it in fact restricts to the required isomorphism.
\end{proof}

\begin{construction}\label{constr:aux2}
Denote by $\widehat{\cfact_{/\sq(S)}}$ the following category: its objects are given by strings of morphisms $[m]\xinert{i}[n]\xrightarrow{s}\sq(S')\xrightarrow{f}\sq(S)$ in which $f$ is an arbitrary morphism in $\fact$, $s$ is a morphism in $\cat$ such that the image of $[n]$ under $s$ belongs to $S'\subset\sq(S')$, $s(0)$ is the initial object of $\sq(S')$ and $s(n)$ is the final object of $\sq(S')$ and $i$ is an inert morphism in $\ccat$. The morphisms are given by commutative diagrams of the form
\begin{equation}\label{eq:fourteen}
\begin{tikzcd}[row sep=huge, column sep=huge]
{[m_0]}\arrow[r, tail, "i_0"]\arrow[dd, "v"]&{[n_0]}\arrow[dd, "w"]\arrow[r, "s_0"]&\sq(S'_0)\arrow[dr, "f"]\arrow[dd, "u"]\\
{}&{}&{}&\sq(S)\\
{[m_1]}\arrow[r, tail, "i_1"]&{[n_1]}\arrow[r, "s_1"]&\sq(S'_1)\arrow[ur, "g"]
\end{tikzcd}
\end{equation}
in which $u$ belongs to $\cfact^\ap$.
\end{construction}
\begin{prop}\label{prop:aux2}
Denote by $\mor^{\lax,!}_\fact(\sq(S),\spanc)$ the category whose objects are lax functors $F:\sq(S)\rightsquigarrow \spanc$ and whose morphisms are given by lax natural transformations $\eta$ such that for every object $x\in\sq(S)$ the component $\eta_x$ is given by a morphism of the form $h_!$ for a morphism $h$ in $\cS$. Then there is an equivalence between $\mor^{\lax,!}_\fact(\sq(S),\spanc)$ and the category of functors $\cF:\widehat{\cfact_{/\sq(S)}}\rightarrow\cS$ satisfying the following conditions:
\begin{itemize}
    \item for every object $y\bydef([m]\xinert{i}[n]\xrightarrow{s}\sq(S')\xrightarrow{f}\sq(S))$ we have
    \[\cF(y)\cong \underset{e\in\widehat{\cfact_{/\sq(S)}}^\el_{y/}}{\lim}\cF(e),\]
    where $\widehat{\cfact_{/\sq(S)}}^\el_{y/}$ is the category whose objects are diagrams of the form 
    \[
    \begin{tikzcd}
    {[l]}\arrow[dd, tail, "j"]\arrow[dr, "i\circ j"]\\
    {}&{[n]}\arrow[r, "s"]&\sq(S')\arrow[r, "f"]&\sq(S)\\
    {[m]}\arrow[ur, tail, "i"]
    \end{tikzcd}
    \]
    in which $l\in\{0,1\}$ and $j$ is an inert morphism in $\ccat$ and morphisms are induced by inert morphisms in $\ccat^\el$;
    \item $\cF$ takes morphisms of the form
    \[
    \begin{tikzcd}[row sep=huge, column sep=huge]
    {[m]}\arrow[r, tail, "i"]\arrow[dd, equal]&{[n]}\arrow[dd, tail, "j"]\arrow[r, "s"]&\sq(S'_j)\arrow[dr, "f"]\arrow[dd, tail, "j'"]\\
    {}&{}&{}&\sq(S)\\
    {[m]}\arrow[r, tail, "j\circ i"]&{[n']}\arrow[r, "s'"]&\sq(S')\arrow[ur, "f"]
    \end{tikzcd},
    \]
    where $j'$ is an inclusion of the segment $S'_j$ of $S'$ between $s'\circ j(0)$ and $s'\circ j(n)$, and 
    \[
    \begin{tikzcd}[row sep=huge, column sep=huge]
    {[m]}\arrow[r, tail, "i"]\arrow[dd, two heads, "\widetilde{a}"]&{[n]}\arrow[dd, two heads, "a"]\arrow[r, "s"]&\sq(S')\arrow[dr, "f"]\arrow[dd, equal]\\
    {}&{}&{}&\sq(S)\\
    {[\widetilde{m}]}\arrow[r, tail, "\widetilde{i}"]&{[n']}\arrow[r, "s'"]&\sq(S')\arrow[ur, "f"]
    \end{tikzcd}
    \]
    to isomorphisms.
\end{itemize}
\end{prop}
\begin{proof}
We begin by providing an isomorphism on the underlying space of objects. First, assume we are given a functor $F:\sq(S)\rightsquigarrow\spanc$. Recall from \cref{prop:lax_for_objects_of_fact} that giving a lax functor $F$ is equivalent to giving a functor $\widetilde{\cL^\lax\sq(S)}\rightarrow\spanc$, where $\widetilde{\cL^\lax\sq(S)}$ is the category from \cref{constr:elementary_lax}. We can identify $\widetilde{\cL^\lax\sq(S)}$ with a coCartesian fibration over $\Delta^\op$: the value of the corresponding functor on $[n]$ is the category whose objects are strings of morphisms $[n]\xrightarrow{s}\sq(S')\xrightarrow{f}\sq(S)$, where $f$ is a morphism in $\fact$ and $s$ is a morphism in $\cat$ such that $s([n])\subset S'\subset\sq(S')$, $s(0)$ is the initial object of $\sq(S')$ and $s(n)$ is the final object of $\sq(S')$, and whose morphisms are diagrams of the form 
\[
\begin{tikzcd}[row sep=huge, column sep=huge]
{}&\sq(S')\arrow[dr, "f"]\arrow[dd, "h"]\\
{[n]}\arrow[ur, "s"]\arrow[dr, "t"]&{}&\sq(S)\\
{}&\sq(S'')\arrow[ur, "g"]
\end{tikzcd},
\]
where $h$ is a morphism in $\cfact^\ap$. An inert morphism $j:[n]\rightarrowtail[n']$ then induces $j^*:\widetilde{\cL^\lax(\sq(S))}([n'])\rightarrow\widetilde{\cL^\lax(\sq(S))}([n])$ given by sending $[n']\xrightarrow{s'}\sq(S')\xrightarrow{f}\sq(S)$ to $[n]\xrightarrow{s}\sq(S'_j)\xrightarrow{f}\sq(S)$ and an active morphism $a:[n]\twoheadrightarrow[n']$ induces $a^*:\widetilde{\cL^\lax(\sq(S))}([n'])\rightarrow\widetilde{\cL^\lax(\sq(S))}([n])$ given by sending $[n']\xrightarrow{s'}\sq(S')\xrightarrow{f}\sq(S)$ to $[n]\xrightarrow{s'\circ a}\sq(S')\xrightarrow{f}\sq(S)$. Recall from \cite[Corollary 2.16.]{kositsyn2021completeness} that $\spanc$, viewed as a coCartesian fibration over $\Delta^\op$, has $\spanc([n])\cong \cat^\inrt_{/[m]}$. Here $\ccat^\inrt_{/[n]}$ is the full subcategory of $\ccat_{/[n]}$ on inert morphisms and $\cat^\inrt_{/[n]}$ is the full subcategory of $\mor_\cat(\ccat^\inrt_{/[n]},\cS)$ on those functors $G$ for which 
\[G([m]\xinert{i}[n])\cong \underset{(j:[m]\rightarrowtail e)\in\ccat^\el_{[m]/}}{\lim}G(e\xinert{i\circ j}[n]).\]
For an inert morphism $[n]\xinert{j}[n']$ we have
\[j^*G([m]\xinert{i}[n])\cong G([m]\xinert{j\circ i}[n'])\]
and for an active morphism $[n]\xactive{a}[n']$ we have
\[a^*G([m]\xinert{i}[n])\cong G([\widetilde{m}]\xinert{\widetilde{i}}[n']),\]
where $a\circ i\cong \widetilde{i}\circ \widetilde{a}$ is the active/inert factorization in $\ccat$.\par
The functor $F$ can be identified with a morphism $F:\widetilde{\cL^\lax\sq(S)}\rightarrow\spanc$ of coCartesian fibrations over $\Delta^\op$. In particular, for every $n$ we have a morphism of categories $F_n:\widetilde{\cL^\lax\sq(S)}([n])\rightarrow\spanc([n])$. Using the definition of $\spanc$, we can view $F_n$ as an element of 
\[\mor_\cat(\widetilde{\cL^\lax\sq(S)}([n]), \mor_\cat(\ccat^\inrt_{/[n]},\cS))\cong \mor_\cat(\widetilde{\cL^\lax\sq(S)}([n])\times\ccat^\inrt_{/[n]}, \cS).\]
Observe that there is a forgetful functor $U:\widehat{\cfact_{/\sq(S)}}\rightarrow \widetilde{\cL^\lax(\sq(S))}$ that sends $([m]\xinert{i}[n]\xrightarrow{s}\sq(S')\xrightarrow{f}\sq(S))$ to $([n]\xrightarrow{s}\sq(S')\xrightarrow{f}\sq(S))$. Also denote by $p:\widehat{\cfact_{/\sq(S)}}\rightarrow\Delta^\op$ the morphism that sends $([m]\xinert{i}[n]\xrightarrow{s}\sq(S')\xrightarrow{f}\sq(S))$ to $[n]$. It is easy to see that $p^{-1}([n])\cong \widetilde{\cL^\lax\sq(S)}([n])\times\ccat^\inrt_{/[n]}$, so we can identify the functors $F_n$ defined above with a functor
\[\widehat{\cF}:\coprod_{n\geq 0} p^{-1}([n])\rightarrow\cS.\]
Also observe that, since all $F_n$ land in $\spanc([n])\subset \mor_\cat(\ccat^\inrt_{/[n]},\cS)$, $\widehat{\cF}$ satisfies the first condition in our claim. Now we need to extend $\widehat{\cF}$ to a morphism $\cF:\widehat{\cfact_{/\sq(S)}}\rightarrow\cS$, to do this we need to describe its value on the morphisms that do not lie over identity morphisms in $\Delta^\op$. Observe that the morphism in $\widehat{\cfact_{/\sq(S)}}$ represented by the diagram (\ref{eq:fourteen}) can be decomposed as 
\[
\begin{tikzcd}[row sep=huge, column sep=huge]
{[m_0]}\arrow[r, tail, "i_0"]\arrow[d, equal]&{[n_0]}\arrow[d, equal]\arrow[r, "s_0"]&\sq(S'_0)\arrow[d, "u'"]\arrow[dr, "f"]\\
{[m_0]}\arrow[d, "v"]\arrow[r, tail, "i_0"]&{[n_0]}\arrow[d, "w"]\arrow[r, "u'\circ s_0"]&\sq(S'_{0,u})\arrow[d, tail, "j_u"]&\sq(S)\\
{[m_1]}\arrow[r, tail, "i_1"]&{[n_1]}\arrow[r, "s_1"]&\sq(S'_1)\arrow[ur, "g"]
\end{tikzcd},
\]
where $j_u:\sq(S'_{0,u})\rightarrowtail\sq(S_1')$ is the inclusion of a subinterval between $s_1\circ w(0)$ and $s_1\circ w(n_0)$. Observe that the top morphism lies in $p^{-1}([n_0])$, so the value of $\cF$ on it is already well-defined. Assume that $w\cong i\circ a$ is the active/inert decomposition of $w$, then observe that we can further decompose the bottom morphism in the above diagram as
\begin{equation}\label{eq:fifteen}
\begin{tikzcd}[row sep=huge, column sep=huge]
{[m_0]}\arrow[r, tail, "i_0"]\arrow[d, two heads, "\widetilde{a}"]&{[n_0]}\arrow[d, two heads, "a"]\arrow[r, "u'\circ s_0"]&\sq(S'_{0,u})\arrow[d, equal]\arrow[r, "g\circ j_u"]&\sq(S)\\
{[\widetilde{m}_0]}\arrow[r, tail, "\widetilde{i}_0"]\arrow[d, equal]&{[\widetilde{n}_0]}\arrow[d, tail, "i"]\arrow[r, "\widetilde{s}_1"]&\sq(S'_{0,u})\arrow[d, tail, "j_u"]\\
{[\widetilde{m}_1]}\arrow[d, tail, "i_2"]\arrow[r, tail, "i\circ \widetilde{i}_0"]&{[n_1]}\arrow[r, "s_1"]\arrow[d, equal]&\sq(S'_1)\arrow[d, equal]\\
{[m_1]}\arrow[r, tail, "i_1"]&{[n_1]}\arrow[r, "s_1"]&\sq(S'_1)\arrow[uuur, "g"]
\end{tikzcd}.
\end{equation}
The bottom morphism lies in $p^{-1}([n_1])$, so $\cF$ is already defined on it too. So it suffices to define $\cF$ on the top two morphisms. However, observe that, since $F$ is a functor, those morphisms are sent to isomorphisms. More precisely, by untangling the identifications we made above we see that 
\[\cF([m_0]\xinert{i_0}[n_0]\xrightarrow{u\circ s_0}\sq(S'_{0,u})\xrightarrow{g\circ j_u}\sq(S))\]
can be identified with $F_{n_0}(a^*([\widetilde{n}_0]\xrightarrow{\widetilde{s}_1}\sq(S'_{0,u})\xrightarrow{g\circ j_u}\sq(S)))([m_0]\xinert{i_0}[n_0])$, while 
\[\cF([\widetilde{m}_0]\xinert{\widetilde{i}_0}[\widetilde{n}_0]\xrightarrow{\widetilde{s}_1}\sq(S'_{0,u})\xrightarrow{g\circ j_u}\sq(S))\]
can be identified with $a^*(F_{\widetilde{n}_0}([\widetilde{n}_0]\xrightarrow{\widetilde{s}_1}\sq(S'_{0,u})\xrightarrow{g\circ j_u}\sq(S)))([m_0]\xinert{i_0}[n_0])$. Since $F$ is assumed to be a functor, we have $a^*\circ F_{\widetilde{n}_0}\cong F_{n_0}\circ a^*$, which proves that these two spaces are isomorphic. Similarly, the value
\[\cF([\widetilde{m}_0]\xinert{\widetilde{i}_0}[\widetilde{n}_0]\xrightarrow{\widetilde{s}_1}\sq(S'_{0,u})\xrightarrow{g\circ j_u}\sq(S))\]
can be identified with $F_{\widetilde{n}_0}(i^*([n_1]\xrightarrow{s_1}\sq(S'_1)\xrightarrow{g}\sq(S)))([\widetilde{m}_0]\xinert{\widetilde{i}_0}[\widetilde{n}_0])$, while
\[\cF([\widetilde{m}_1]\xinert{i\circ \widetilde{i}_0}[n_1]\xrightarrow{s_1}\sq(S'_1)\xrightarrow{g}\sq(S))\]
can be identified with $i^*(F_{n_1}([n_1]\xrightarrow{s_1}\sq(S'_1)\xrightarrow{g}\sq(S)))([\widetilde{m}_0]\xinert{\widetilde{i}_0}[\widetilde{n}_0])$, those values are once again isomorphic by functoriality of $F$. We have thus defined the required morphism $\cF:\widehat{\cfact_{/\sq(S)}}\rightarrow\cS$, moreover, the morphisms $(\widetilde{a},a)$ and $(i,j_u)$ appearing in the diagram (\ref{eq:fifteen}) are precisely the morphisms mentioned in the second condition of \cref{prop:aux2}, so $\cF$ indeed satisfies it since it sends them to isomorphisms by construction. Conversely, given a functor $\cF$ as above that satisfies the conditions of the proposition, we can define $F:\widetilde{\cL^\lax(\sq(S))}\rightarrow\spanc$ by setting $F_n$ to be equal to the restriction $\cF|_{p^{-1}([n])}$. The fact that $F_n$ lands in $\spanc([n])\subset\mor_\cat(\ccat^\inrt_{/[n]},\cS)$ then follows from the first condition of the proposition and the fact that $F$ is a functor from the second. This establishes the required isomorphism on the spaces of objects.\par
To prove the isomorphism on morphisms, observe that a lax natural transformation $\alpha:F\rightarrow G$ amounts to giving a morphism $\alpha_X:F(x)\rightarrow G(x)$ for every object $x\in\sq(S)$ together with, for each morphism $v:x\rightarrow y$ represented by $([i]\rightarrowtail[1]\xrightarrow{s}\sq(S')\xrightarrow{f}\sq(S))$, morphisms
\[\alpha_v\in \mor_{\spanc(F(x),G(y))}(\alpha_{y,!}\circ F(v),G(v)\circ \alpha_{x,!})\cong \mor_{\cS_{/F(x)\times F(y)}}(F(v),\alpha^*_y\circ G(v)\circ \alpha_{x,!}),\]
where the isomorphism follows since $\alpha_{y,!}$ and $\alpha^*_y$ are adjoint and $\spanc(F(x),F(y))\cong \cS_{/F(x)\times F(y)}$, such that $\alpha_v$ are compatible with units and compositions. In other words, if we identify $F(v)$ and $G(v)$ with spans $(F(x)\leftarrow F(v)\rightarrow F(y))$ and $(G(x)\leftarrow G(v)\rightarrow G(y))$, then $\alpha_v$ is a morphism $F(v)\rightarrow G(v)$ making the following diagram commute
\[
\begin{tikzcd}[row sep=huge, column sep=huge]
F(x)\arrow[d, "\alpha_x"]&F(v)\arrow[r]\arrow[l]\arrow[d, "\alpha_v"]&F(y)\arrow[d, "\alpha_y"]\\
G(x)&G(v)\arrow[r]\arrow[l]&G(y)
\end{tikzcd}.
\]
This can be identified with a morphism from $F(v)$ to $G(v)$ in $\ccat^\inrt_{/[1]}$. The data of $\alpha_u$ for all morphisms $u$ can then be identified with a morphism $\widetilde{\alpha}_1:\cF|_{p^{-1}([1])}\rightarrow\cG|_{p^{-1}([1])}$, where $\cF$ and $\cG$ are the objects of $\mor_\cat(\widehat{\cfact_{/\sq(S)}},\cS)$ corresponding to $F$ and $G$, and the fact that $\alpha$ is a natural transformation then means that $\widetilde{\alpha}_1$ can be extended to $\widetilde{\alpha}:\cF\rightarrow\cG$. Conversely, for every $\widetilde{\alpha}:\cF\rightarrow\cG$ we can construct the corresponding lax natural transformation by setting $\alpha_x$ to be the restriction of $\widetilde{\alpha}$ to $([0]=\joinrel=[0]\rightarrow*\xrightarrow{x}\sq(S))$ and $\alpha_v$ to be the restriction of $\widetilde{\alpha}$ to $([1]=\joinrel=[1]\rightarrow\sq(S')\xrightarrow{v}\sq(S))$. This establishes the required equivalence.
\end{proof}
\begin{theorem}\label{thm:overcategory}
For any $\cF\in\fact$ we have an equivalence of categories
\[\mor_\fact^{\lax,!}(\cF,\spanc)\cong \fact_{/\cF}.\]
\end{theorem}
\begin{proof}
By \cref{cor:overcategories_and_colimits} and \cref{prop:lax_functors_respect_colimits} it suffices to prove the claim for $\cF\cong\sq(S)$ for some $S\in\cfact$. In this case, we can further use \cref{prop:aux2} and \cref{prop:model_for_overcategories} to reduce our claim to the following simpler statement: there is an equivalence between the category of presheaves on $(\cfact_{/\sq(S)})^\op$ satisfying the Segal condition and the category of presheaves on $\widehat{\cfact_{/\sq(S)}}$ satisfying the conditions of \cref{prop:aux2}. To prove this, first denote by $\overline{p}:\widehat{\cfact_{/\sq(S)}}\rightarrow\cfact_{/\sq(S)}$ the morphism that sends $([m]\xinert{i}[n]\xrightarrow{s}\sq(S')\xrightarrow{f}\sq(S))$ to $(\sq(S'_i)\xrightarrow{f}\sq(S))$, where $S'_i$ is the substring of $S'$ on those segments that lie between $s\circ i(0)$ and $s\circ i(m)$ and sends a morphism represented by diagram (\ref{eq:fourteen}) to the induce morphism $u_i:\sq(S'_{i_0})\rightarrow\sq(S'_{i_1})$ over $\sq(S)$. Denote by $K\subset \widehat{\cfact_{/\sq(S)}}$ the subcategory generated by morphisms described in the second condition of \cref{prop:aux2} and denote by $\widehat{\cfact_{/\sq(S)}}[K^{-1}]$ the category obtained by inverting those morphisms. Observe that every functor $\cF:\widehat{\cfact_{/\sq(S)}}\rightarrow\cS$ that satisfies the conditions of \cref{prop:aux2} (particularly the second condition) factors uniquely through $\widehat{\cfact_{/\sq(S)}}[K^{-1}]$. Also observe that the morphism $\overline{p}$ sends all the morphisms in $K$ to identities, so it also factors as
\[
\begin{tikzcd}[row sep=huge, column sep=huge]
\widehat{\cfact_{/\sq(S)}}\arrow[d, "\overline{p}"]\arrow[r]&{\widehat{\cfact_{/\sq(S)}}[K^{-1}]}\arrow[dl, "p" swap]\\
\cfact_{/\sq(S)}
\end{tikzcd}.
\]
We will prove that the functor $p$ in the diagram above is in fact an equivalence of categories. This will imply that the category of presheaves on $(\cfact_{/\sq(S)})^\op$ is equivalent to the category of presheaves on $\widehat{\cfact_{/\sq(S)}}^\op$ satisfying the second condition of \cref{prop:aux2}. Moreover, it will be clear from the construction of the equivalence that the presheaves on $(\cfact_{/\sq(S)})^\op$ that satisfy the Segal condition would correspond to those presheaves on $\widehat{\cfact_{/\sq(S)}}^\op$ that additionally satisfy the first condition of \cref{prop:aux2}, thus concluding the proof of our claim.\par
In order to prove that those categories are equivalent we first construct a morphism $s:\cfact_{/\sq(S)}\rightarrow\widehat{\cfact_{/\sq(S)}}[K^{-1}]$. On objects $s$ acts by sending $\sq(S')\xrightarrow{f}\sq(S)$ to $([1]=\joinrel=[1]\xrightarrow{a}\sq(S')\xrightarrow{f}\sq(S))$, where $a$ is the morphism that sends $0$ to the initial and $1$ the final object of $\sq(S')$. By \cref{prop:generators} every morphism in $\cfact_{/\sq(S)}$ decomposes into elementary morphisms, so in order to describe the action of $s$ on morphisms it suffices to describe it on the elementary morphisms. Assume $e:\sq(S'_0)\rightarrow\sq(S'_1)$ is an elementary morphism over $\sq(S)$ that belongs to $\cfact^\ap$, in that case we define $s(e)$ to be given by the diagram 
\[
\begin{tikzcd}[row sep=huge, column sep=huge]
{[1]}\arrow[r, equal]\arrow[dd, equal]&{[1]}\arrow[dd, equal]\arrow[r, "a_0"]&\sq(S'_0)\arrow[dr, "f_0"]\arrow[dd, "e"]\\
{}&{}&{}&\sq(S)\\
{[1]}\arrow[r, equal]&{[1]}\arrow[r, "a_1"]&\sq(S'_1)\arrow[ur, "f_1"]
\end{tikzcd}.
\]
Now assume that the elementary morphism has the form $\delta^h_0$ (the same argument will work for any other elementary inclusion). This morphism is sent to 
\[
\begin{tikzcd}[row sep=huge, column sep=huge]
{[1]}\arrow[d, tail, "\delta_0"]\arrow[r, equal]&{[1]}\arrow[r, "a_0"]\arrow[d, tail, "\delta_0"]&\sq(S'_0)\arrow[dr, "f_0"]\arrow[d, "\delta^h_0"]\\
{[2]}\arrow[r, equal]&{[2]}\arrow[r, "a'"]&\sq(S'_1)\arrow[r, "f_1"]&\sq(S)\\
{[1]}\arrow[r, equal]\arrow[u, two heads, "\delta_1"]&{[1]}\arrow[u, two heads, "\delta_1"]\arrow[r, "a_1"]&\sq(S'_1)\arrow[u, equal]\arrow[ur, "f_1"]
\end{tikzcd}
\]
(observe that the bottom morphism in this diagram is invertible in $\widehat{\cfact_{/\sq(S)}}[K^{-1}]$, so in order to make a composable chain of morphisms from $s(f_0)$ to $s(f_1)$ we take the inverse of it), where $a'$ sends $1<2$ into $S'_0\subset S'_1$ using $a_0$ and $0<1$ is sent to the first segment of $S'_1$. It is easy to see that $p\circ s\cong \id$. In order to finish the proof we will construct an equivalence $\eta:s\circ p\xrightarrow{\sim}\id$. To do so we first need to define an isomorphism $\eta_x:x\xrightarrow{\sim}s\circ p(x)$ for every object $x\in\widehat{\cfact_{/\sq(S)}}$. Assume that $x$ corresponds to the string $([m]\xinert{i}[n]\xrightarrow{s}\sq(S')\xrightarrow{f}\sq(S))$, then we define $\eta_x$ to be given by the diagram
\[
\begin{tikzcd}[row sep=huge, column sep=huge]
{[1]}\arrow[r, equal]\arrow[d, two heads, "a'"]&{[1]}\arrow[d, two heads, "a'"]\arrow[r, "a"]&\sq(S'_i)\arrow[d, equal]\arrow[dr, "f_i"]\\
{[m]}\arrow[r, equal]\arrow[d, equal]&{[m]}\arrow[r, "s\circ i"]\arrow[d, tail, "i"]&\sq(S'_i)\arrow[d, tail, "j'"]\arrow[r, "f_i"]&\sq(S)\\
{[m]}\arrow[r, tail, "i"]&{[n]}\arrow[r, "s"]&\sq(S')\arrow[ur, "f"]
\end{tikzcd}.
\]
Observe that both of the morphisms appearing in the definition of $\eta_x$ lie in $K$, so is is indeed an isomorphism. To prove that $\eta$ defines a natural transformation we need to show that for any morphism $g:x\rightarrow y$ in $\widehat{\cfact_{/\sq(S)}}$ the following naturality square
\[
\begin{tikzcd}[row sep=huge, column sep=huge]
x\arrow[r, "g"]&y\\
s\circ p(x)\arrow[u, "\eta_x"]\arrow[r, "s\circ p(g)"]&s\circ p(y)\arrow[u, "\eta_y"]
\end{tikzcd}
\]
commutes. Assume that the morphism $g$ is represented by the diagram (\ref{eq:fourteen}). We begin by decomposing it as
\[
\begin{tikzcd}[row sep=huge, column sep=huge]
{[m_0]}\arrow[r, tail, "i_0"]\arrow[d, two heads, "\widetilde{a}"]&{[n_0]}\arrow[d, two heads, "a"]\arrow[r, "s_0"]&\sq(S_0')\arrow[d, "u'"]\arrow[dr, "f"]\\
{[\widetilde{m}_0]}\arrow[r, tail, "\widetilde{i}_0"]\arrow[d, tail, "i_2"]&{[\widetilde{n}_0]}\arrow[r, "\widetilde{s}_1"]\arrow[d, tail, "i"]&\sq(S'_{0,u})\arrow[d, tail, "j_u"]\arrow[r, "g\circ j_u"]&\sq(S)\\
{[m_1]}\arrow[r, tail, "i_1"]&{[n_1]}\arrow[r, "s_1"]&\sq(S'_1)\arrow[ur, "g"]
\end{tikzcd}
\]
using the notation from the proof of \cref{prop:aux2}. It suffices to prove the commutativity of the naturality squares separately for the top and bottom morphism, we start from the top. Observe that the image under $p$ of the top morphism belongs to $\cfact^\ap$. The relevant commutative diagram is a large three-dimensional diagram, for convenience we will replace it with three commutative squares given by the images of that diagram under $p_1:\widehat{\cfact_{/\sq(S)}}\rightarrow\ccat$, $p_2:\widehat{\cfact_{/\sq(S)}}\rightarrow\ccat$ and $p_3:\widehat{\cfact_{/\sq(S)}}\rightarrow\cfact$, where $p_1$ sends $([m]\xinert{i}[n]\xrightarrow{s}\sq(S')\xrightarrow{f}\sq(S))$ to $[m]$, $p_2$ to $[n]$ and $p_3$ to $\sq(S')$. Using this convention and the definition of $s$ on morphisms in $\cfact^\ap$, we see that the diagrams have the form
\[
\begin{tikzcd}[row sep=huge, column sep=huge]
{[1]}\arrow[r, equal]\arrow[d, two heads]&{[1]}\arrow[d, two heads]\\
{[m_0]}\arrow[r, two heads, "\widetilde{a}"]&{[\widetilde{m}_0]}
\end{tikzcd},
\begin{tikzcd}[row sep=huge, column sep=huge]
{[1]}\arrow[r, equal]\arrow[d, two heads]&{[1]}\arrow[d, two heads]\\
{[n_0]}\arrow[r, two heads, "a"]&{[\widetilde{n}_0]}
\end{tikzcd},
\begin{tikzcd}[row sep=huge, column sep=huge]
\sq(S'_{i_0})\arrow[r, "p(u')"]\arrow[d, tail, "j'_{i_0}"]&\sq(S'_{0,u,\widetilde{i}_0})\arrow[d, tail, "j'_{\widetilde{i}_0}"]\\
\sq(S'_0)\arrow[r, "u'"]&\sq(S'_{u,0})
\end{tikzcd}.
\]
The fact that these diagrams commute is obvious. Now we need to prove the commutativity of the naturality squares for the bottom morphism. Observe that we can further decompose $j_u$ into elementary inclusions, so we can assume that $j_u\cong\delta^h_0$. In this case $i\cong\delta_0$ and $i_2$ is either $\id$ or $\delta_0$. The case of $i_2\cong \id$ is trivial, so we can assume $i_2\cong \delta_0$. In this case the relevant diagrams are
\[
\begin{tikzcd}[row sep=huge, column sep=huge]
{[1]}\arrow[r, tail, "\delta_0"]\arrow[d, two heads]&{[2]}\arrow[dr, dotted]&{[1]}\arrow[l, two heads, "\delta_1" swap]\arrow[d, two heads]\\
{[\widetilde{m}_0]}\arrow[rr, tail, "\delta_0"]&{}&{[\widetilde{m}_0+1]}
\end{tikzcd},
\]
\[
\begin{tikzcd}[row sep=huge, column sep=huge]
{[1]}\arrow[r, tail, "\delta_0"]\arrow[d, two heads]&{[2]}\arrow[dr, dotted]&{[1]}\arrow[l, two heads, "\delta_1" swap]\arrow[d, two heads]\\
{[\widetilde{m}_0]}\arrow[rr, tail, "\delta_0"]\arrow[d, tail, "\widetilde{i}_0"]&{}&{[\widetilde{m}_0+1]}\arrow[d, tail, "i_1"]\\
{[\widetilde{n}_0]}\arrow[rr, tail, "\delta_0"]&{}&{[\widetilde{n}_0+1]}
\end{tikzcd}
\]
and
\[
\begin{tikzcd}[row sep=huge, column sep=huge]
\sq(S'_{0,u,\widetilde{i}_0})\arrow[d, tail, "j'_{\widetilde{i}_0}"]\arrow[r, tail, "\delta^h_0"]&\sq(S'_{1,i_1})\arrow[d, tail, "j'_{i_1}"]\\
\sq(S'_{0,u})\arrow[r, "j_u"]&\sq(S'_1)
\end{tikzcd}
\]
(we again implicitly invert the morphisms that are "pointing the wrong way"). The commutativity of the third diagram is obvious. To prove the commutativity of the outer rectangles in the first two diagrams it would suffice to provide the dotted arrow shown on the diagrams that would make the inner rectangle and the triangle on the right commute. This arrow should be given by a morphism from \[([2]=\joinrel=[2]\xrightarrow{a'}\sq(S'_{1,i_1})\xrightarrow{g'}\sq(S))\]
to 
\[([\widetilde{m}_0+1]\xinert{i_1}[\widetilde{n}_0+1]\xrightarrow{s_1}\sq(S'_1)\xrightarrow{g}\sq(S)),\]
where $a'$ is a morphism sending $1<2$ to $S'_{0,u,\widetilde{i}_0}$ and $0<1$ to the first segment in $S'_{1,i_1}$. We define this morphism to be given by the diagram
\[
\begin{tikzcd}[row sep=huge, column sep=huge]
{[2]}\arrow[r, equal]\arrow[d, two heads, "a''"]&{[2]}\arrow[d, two heads, "a''"]\arrow[r, "a'"]&\sq(S'_{1,i_1})\arrow[d, equal]\arrow[dr, "g'"]\\
{[\widetilde{m}_0+1]}\arrow[r, equal]\arrow[d, equal]&{[\widetilde{m}_0+1]}\arrow[r, "s_1\circ i_1"]\arrow[d, tail, "i_1"]&\sq(S'_{1,i_1})\arrow[r, "g'"]\arrow[d, tail, "j'_{i_1}"]&\sq(S)\\
{[\widetilde{m}_0+1]}\arrow[r, tail, "i_1"]&{[\widetilde{n}_0+1]}\arrow[r, "s_1"]&\sq(S'_1)\arrow[ur, "g"]
\end{tikzcd},
\]
where $a''$ denotes a morphism sending $0$ to $0$, $1$ to $1$ and $2$ to $(\widetilde{m}_0+1)$. That this morphism makes the relevant diagrams commute is elementary, and with that we have finished the proof of the claim.
\end{proof}
\begin{cor}\label{cor:distributive_laws_in_span}
The category $\fact$ is isomorphic to the category of distributive laws in $\spanc$.
\end{cor}
\begin{proof}
It follows by applying \cref{thm:overcategory} to the case $\cF\cong*$.
\end{proof}
\medskip
\bibliographystyle{alpha}
\bibliography{ref.bib}

\begin{thebibliography}{BMW12}

\bibitem[Bec69]{beck1969distributive}
Jon Beck.
\newblock Distributive laws.
\newblock In {\em Seminar on triples and categorical homology theory}, pages
  119--140. Springer, 1969.

\bibitem[BMW12]{berger2012monads}
Clemens Berger, Paul-Andr{\'e} Mellies, and Mark Weber.
\newblock Monads with arities and their associated theories.
\newblock {\em Journal of Pure and Applied Algebra}, 216(8-9):2029--2048, 2012.

\bibitem[Bou77]{bousfield1977constructions}
Aldridge~K Bousfield.
\newblock Constructions of factorization systems in categories.
\newblock {\em Journal of Pure and Applied Algebra}, 9(2-3):207--220, 1977.

\bibitem[CH19]{chu2019homotopy}
Hongyi Chu and Rune Haugseng.
\newblock Homotopy-coherent algebra via segal conditions.
\newblock {\em arXiv preprint arXiv:1907.03977}, 2019.

\bibitem[FK72]{freyd1972categories}
Peter~J Freyd and G~Max Kelly.
\newblock Categories of continuous functors, i.
\newblock {\em Journal of pure and applied algebra}, 2(3):169--191, 1972.

\bibitem[FL16]{fiorenza2016t}
Domenico Fiorenza and Fosco Loregi{\`a}n.
\newblock t-structures are normal torsion theories.
\newblock {\em Applied Categorical Structures}, 24(2):181--208, 2016.

\bibitem[GHL20]{gagna2020gray}
Andrea Gagna, Yonatan Harpaz, and Edoardo Lanari.
\newblock Gray tensor products and lax functors of ($\infty$, 2)-categories.
\newblock {\em arXiv preprint arXiv:2006.14495}, 2020.

\bibitem[GR17]{gaitsgory2017study}
Dennis Gaitsgory and Nick Rozenblyum.
\newblock {\em A study in derived algebraic geometry}, volume~1.
\newblock American Mathematical Soc., 2017.

\bibitem[Hau18]{haugseng2018equivalence}
Rune Haugseng.
\newblock On the equivalence between $\theta_n$-spaces and iterated segal
  spaces.
\newblock {\em Proceedings of the American Mathematical Society},
  146(4):1401--1415, 2018.

\bibitem[Isb57]{isbell1957some}
John~R Isbell.
\newblock Some remarks concerning categories and subspaces.
\newblock {\em Canadian Journal of mathematics}, 9:563--577, 1957.

\bibitem[Joy08]{joyal2008notes}
Andr{\'e} Joyal.
\newblock Notes on quasi-categories.
\newblock {\em preprint}, 2008.

\bibitem[JT07]{joyal2007quasi}
Andr{\'e} Joyal and Myles Tierney.
\newblock Quasi-categories vs segal spaces.
\newblock {\em Contemporary Mathematics}, 431(277-326):10, 2007.

\bibitem[Kos21]{kositsyn2021completeness}
Roman Kositsyn.
\newblock Completeness for monads and theories.
\newblock {\em arXiv preprint arXiv:2104.00367}, 2021.

\bibitem[KT93]{korostenski1993factorization}
Mareli Korostenski and Walter Tholen.
\newblock Factorization systems as eilenberg-moore algebras.
\newblock {\em Journal of Pure and Applied Algebra}, 85(1):57--72, 1993.

\bibitem[KV91]{kapranov1991combinatorial}
Mikhail~M Kapranov and Vladimir~A Voevodsky.
\newblock Combinatorial-geometric aspects of polycategory theory: pasting
  schemes and higher bruhat orders (list of results).
\newblock {\em Cahiers de Topologie et G{\'e}om{\'e}trie diff{\'e}rentielle
  cat{\'e}goriques}, 32(1):11--27, 1991.

\bibitem[Lur]{luriehigher}
Jacob Lurie.
\newblock Higher algebra, september 2017.
\newblock {\em available at his webpage https://www. math. ias. edu/\~{}
  lurie}.

\bibitem[Lur09]{lurie2009higher}
Jacob Lurie.
\newblock {\em Higher topos theory}.
\newblock Princeton University Press, 2009.

\bibitem[LV17]{loregian2017factorization}
Fosco Loregian and Simone Virili.
\newblock Factorization systems on (stable) derivators.
\newblock {\em arXiv preprint arXiv:1705.08565}, 2017.

\bibitem[Ste06]{steiner2006orientals}
Richard Steiner.
\newblock Orientals.
\newblock {\em arXiv preprint math/0601383}, 2006.

\bibitem[Str72a]{street1972formal}
Ross Street.
\newblock The formal theory of monads.
\newblock {\em Journal of Pure and Applied Algebra}, 2(2):149--168, 1972.

\bibitem[Str72b]{street1972two}
Ross Street.
\newblock Two constructions on lax functors.
\newblock {\em Cahiers de topologie et g{\'e}om{\'e}trie diff{\'e}rentielle
  cat{\'e}goriques}, 13(3):217--264, 1972.

\bibitem[Str87]{street1987algebra}
Ross Street.
\newblock The algebra of oriented simplexes.
\newblock {\em Journal of Pure and Applied Algebra}, 49(3):283--335, 1987.

\bibitem[Str91]{street1991parity}
Ross Street.
\newblock Parity complexes.
\newblock {\em Cahiers de topologie et g{\'e}om{\'e}trie diff{\'e}rentielle
  cat{\'e}goriques}, 32(4):315--343, 1991.

\bibitem[Ver08]{verity2008weak}
Dominic~RB Verity.
\newblock Weak complicial sets i. basic homotopy theory.
\newblock {\em Advances in Mathematics}, 219(4):1081--1149, 2008.

\bibitem[Web04]{weber2004generic}
Mark Weber.
\newblock Generic morphisms, parametric representations and weakly cartesian
  monads.
\newblock {\em Theory Appl. Categ}, 13(14):191--234, 2004.

\end{thebibliography}
\bigskip
Roman Kositsyn, \textsc{ Department of Mathematics, National Research University Higher School of Economics, Moscow}, \href{mailto:roman.kositsin@gmail.com}{roman.kositsin@gmail.com}
\end{document}